\documentclass[11pt]{amsart}
\usepackage[utf8]{inputenc}
\usepackage[T1]{fontenc}
\usepackage{amssymb,amsmath}
\usepackage{graphics}
\usepackage{xcolor}
\usepackage{enumitem}
\usepackage{mathtools}
\usepackage[colorlinks = true, linkcolor = teal, urlcolor  = teal, citecolor = violet]{hyperref}
\usepackage[margin=1.0in]{geometry}
\usepackage{cleveref}
\usepackage{thmtools}
\crefname{equation}{Equation}{Equations}

\counterwithin{figure}{section}

\usepackage[normalem]{ulem}

\usepackage{physics}
\usepackage{amsmath}
\usepackage{tikz}
\usepackage{mathdots}
\usepackage{yhmath}
\usepackage{cancel}
\usepackage{color}
\usepackage{siunitx}
\usepackage{array}
\usepackage{multirow}
\usepackage{amssymb}
\usepackage{gensymb}
\usepackage{tabularx}
\usepackage{extarrows}
\usepackage{booktabs}
\usetikzlibrary{fadings}
\usetikzlibrary{patterns}
\usetikzlibrary{shadows.blur}
\usetikzlibrary{shapes}

\newtheorem{theorem}{Theorem}[section]
\newtheorem*{theorem*}{Theorem}
\newtheorem{definition}[theorem]{Definition}
\newtheorem*{definition*}{Definition} 
\newtheorem{proposition}[theorem]{Proposition}
\newtheorem{corollary}[theorem]{Corollary}

\newtheorem{lemma}[theorem]{Lemma}
\newtheorem{remark}[theorem]{Remark}

\renewcommand{\i}{{\mathrm{i}}}
\newcommand{\id}{{\mathrm{Id}}}
\newcommand{\Id}{{\mathrm{Id}}} 
\newcommand{\range}{\operatorname{range}}

\renewcommand{\d}{{\,\rm d}}
\newcommand{\mutensorn}{\mu^{\otimes n}}
\newcommand\mutensor[1]{\mu^{\otimes {#1}}}

\newcommand{\one}{{\mathbf 1}}
\newcommand{\dist}{{\Delta}}
\newcommand{\pp}{{\mathbb P}}
\newcommand{\ee}{{\mathbb E}}
\newcommand{\ppch}{\pp^{\mathrm{ch}}}

\newcommand{\rr}{{\mathbb R}}
\newcommand{\nn}{{\mathbb N}}
\newcommand{\nnone}{{\mathbb N^*}}
\newcommand{\nnzero}{{\mathbb N}}
\newcommand{\cc}{{\mathbb C}}
\renewcommand{\P}{{\mathrm P}}
\newcommand{\supp}{\operatorname{supp}}
\newcommand\pc[1]{{\P(\cc^{#1}) }}
\newcommand\pcr{{\P(\cc^{r_m}) }}
\newcommand\pcd{{\P(\cc^{d}) }}
\newcommand\supchi{{\supp \chi_{{\rm inv}}}}

\newcommand{\B}{{\mathcal{B}}}
\newcommand{\W}{{\mathcal W}}

\newcommand{\V}{{\mathcal V}}
\newcommand{\J}{{\mathcal J}}
 
\renewcommand{\O}{{\mathcal O}}
\newcommand{\Sd}{{\mathcal M_d^{1,+}}}
\newcommand{\Srm}{{\mathcal M_{r_m}^{1,+}}} 
\newcommand{\Sr}{{\mathcal M_{r_m}^{1,+}}} 
\newcommand{\Ud}{{\mathcal U(d)}}
\newcommand{\SUd}{{\mathcal{SU}(d)}}
\newcommand{\Urm}{{\mathcal U(r_m)}}
\newcommand{\Ld}{{\mathcal M_d}} 
\newcommand{\D}{{\mathcal D_m}}
\newcommand{\tJ}{{\widetilde{\mathcal J}}}

\makeatletter
\newcommand{\specialcell}[1]{\ifmeasuring@#1\else\omit$#1$\ignorespaces\fi}
\makeatother

\usepackage[backend=bibtex,giveninits=true,citestyle=alphabetic,bibstyle=alphabetic,maxbibnames =100,maxcitenames=5,isbn=false]{biblatex}
\renewbibmacro{in:}{}
\begin{document}
	
	\title[Invariant measures of quantum trajectories]{Dark subspaces and invariant measures \\ of Quantum Trajectories}
	
	\author{T. Benoist}
	\address{Institut de Math\'ematiques de Toulouse, \'Equipe de Statistique et Probabilit\'es,
		Universit\'e Paul Sabatier, 31062 Toulouse Cedex 9, France}
	\email{tristan.benoist@math.univ-toulouse.fr}
	\author{C. Pellegrini}
	\address{Institut de Math\'ematiques de Toulouse, \'Equipe de Statistique et Probabilit\'es,
		Universit\'e Paul Sabatier, 31062 Toulouse Cedex 9, France}
	\email{clement.pellegrini@math.univ-toulouse.fr}
	\author{A. Szczepanek}
	\address{Institute of Mathematics, Jagiellonian University,\\ 
		{\L}ojasiewicza 6, 30-348 Krak\'ow, Poland}
	\email{anna.szczepanek@uj.edu.pl}
	
	\setcounter{tocdepth}{1}
	
	
	\begin{abstract} 
		Quantum trajectories are Markov processes describing the evolution of a quantum system subject to indirect measurements. They can be viewed as place dependent iterated function systems or the result of products of dependent and non identically distributed random matrices. In this article, we establish a complete classification of their invariant measures. 
		
		The classification is done in two steps. First, we prove a Markov process on some linear subspaces called \emph{dark subspaces}, defined in (Maassen, Kümmerer 2006), admits a unique invariant measure. Second, we study the process inside the dark subspaces. Using a notion of minimal family of isometries from a reference space to dark subspaces, we prove a set of measures indexed by orbits of a unitary group is the set of ergodic measures of quantum trajectories.

	\end{abstract}
	
	\maketitle
	\thispagestyle{empty}   
	
	\tableofcontents

	\section{Introduction}

	\subsection{Model.}  
	Consider the canonical complex $d$-dimensional space $\mathbb{C}^d$ and its projective space $\pcd$ equipped with its Borel $\sigma$-algebra. 
	For a non-zero vector $x\in\mathbb C^d$, $\hat x$ stands for the equivalence class of $x$ in $\P(\mathbb C^d)$. For $\hat x\in \P(\cc^d)$,  by $x\in \cc^d$ we denote an arbitrary unit norm representative of $\hat x$, and by $\pi_{\hat x}$ the orthogonal projector onto $\cc x$.
	By $\Ld$ we denote the set of linear maps on $\cc^d$, and for a linear map $v\in \Ld$, $v\cdot \hat{x} $ is the element of the projective space represented by $v\,x$ whenever $v\,x\neq 0$. We equip $\Ld$ with its Borel $\sigma$-algebra and consider a measure $\mu$ on $\Ld$ such that $v\mapsto \|v\|$ is square integrable, i.e., \[\int_{\Ld} \|v\|^2\,\d\mu(v)<\infty,\] and the following stochasticity condition holds:
	\begin{equation*} 
		\int_{\Ld} v^* v \,\mathrm{d} \mu(v) = \id_{\cc^d}. 
	\end{equation*}
	A quantum trajectory is a Markov chain on $\pcd$ with transition kernel
	\begin{equation*}
		\Pi(\hat x,S)=\int_{\Ld} \mathbf{1}_{S}(v\cdot \hat x)  \|v x\|^2 \d \mu(v),
	\end{equation*}
	where $\hat x\in \pcd$ and   $S$ is a Borel subset of $\pcd$. A realization of this Markov chain $(\hat{x}_n)_n$ can be described by 
	\[
	\hat x_n\leadsto\hat x_{n+1}=V_{n+1}\cdot \hat x_{n},
	\]
	where $V_{n+1}$ is a matrix-valued random variable with law $ ||v x_n||^2 \d \mu(v).$ Every $\hat x_n$ can be written using a random product of matrices as 
	\[ \hat x_n = V_{n}\cdots V_{1}\cdot \hat x_0,\]
	where $(V_{n})_n$ is distributed according to the density $\|v_n\dotsb v_1 x_0\|^2 \d\mu^{\otimes n}(v_1,\dotsc,v_n).$
	
	In the context of quantum physics, a realization of the process $(\hat x_n)_n$ is called a \textit{quantum trajectory}. Its definition originates in quantum optics -- see \cite{carmichael}. A prominent example of an experimental realization of a quantum trajectory is the Nobel prize-winning experiment of Serge Haroche's group \cite{guerlin_progressive_2007}. For a mathematical exposition of the theory of quantum measurements, we direct the interested reader to the books \cite{holevo2001statistical,busch2016quantum} and the foundational article \cite{davies1970operational}.
	
	Since $\hat x_n$ can be defined as the result of the action of some random product of matrices on the initial element of the projective space $\hat x_0$, the study of quantum trajectories shares some similarity with that of products of independent and identically distributed (i.i.d.) random matrices -- see e.g. \cite{BLC} for an introduction. For example, as stated in \cite{BenFra}, uniqueness of the $\Pi$-invariant measure is a consequence of \cite[Theorem~2.6]{guivarc2016spectral} applied to the case where their parameter $s$ is equal to $2$, albeit for invertible matrices and under a strong irreducibility assumption. Uniqueness of the $\Pi$-invariant measure was proved in \cite{BenFra} for possibly non-invertible matrices and under the standard irreducibility assumption. This extension relied on the fact that $\|V_n\dotsb V_1x_0\|\neq 0$ almost surely by definition. Leveraging this property, an alternative proof of the fact that $(V_n\dotsb V_1/\|V_n\dotsb V_1\|)_n$ concentrates almost surely on a set of rank-one matrices with common pre-image could be found. Recently, a generalization of this new proof led to the extension of the results from \cite{guivarc2016spectral} to irreducible non-invertible matrices for any $s>0$ -- see \cite{hautecoeur2025}.
	
	Alas, both \cite{guivarc2016spectral} and \cite{BenFra} rely on an assumption of contractivity (called purification in \cite{BenFra} following the terminology of \cite{MaaKumm}). It ensures that, almost surely, any accumulation point of $(V_n\dotsb V_1/\|V_n\dotsb V_1\|)_n$ is of rank one. In the present article we propose to lift the assumption of purification. 
	
	As shown with an example in \cite[Appendix~C]{BenFra}, without purification, uniqueness of the $\Pi$-invariant measure may or may not be true. Our main result is thus a complete classification of the set of $\Pi$-invariant measures, assuming only that $\mu$ is irreducible.
	
	Note that quantum trajectories can also be viewed as place-dependent iterated function systems (IFS). However, the standard assumptions on IFS are not satisfied by quantum trajectories -- see \cite[Proposition~8.3]{BenFat}. Similarly, the usual approach to Markov chains using the notion of $\phi$-irreducibility fails -- see \cite[Proposition~8.1]{BenFat}. Note also that after the publication of \cite{BenFra}, the authors of this reference found that some elements of their proofs already appeared with a different formulation in \cite{kusuoka1989dirichlet} albeit for invertible matrices.

	\subsection{Irreducibility.} Our only assumption throughout the article is an irreducibility one for the measure $\mu$. We say a subspace $E$ of $\cc^d$ is $\mu$-invariant if $vE\subset E$ for $\mu$-almost all $v$. We say $\mu$ is irreducible if there does not exist a non-trivial $\mu$-invariant subspace. From now on we always assume $\mu$ is irreducible without necessarily mentioning it.
	
	This assumption is the standard irreducibility assumption for families of matrices. In the theory of random products of matrices a strong irreducibility assumption is often made. It states that there does not exist a finite set of non-trivial subspaces $\{E_i\}_{i}$ such that for $\mu$-almost all $v$, $v\bigcup_i E_i\subset \bigcup_i E_i$. Strong irreducibility implies the standard one. From now on, when we invoke the irreducibility assumption, we mean the standard irreducibility.
	
	The definition of irreducibility we require looks stronger than the {\bf ($\phi$-Erg)} assumption used in \cite{BenFra}. Indeed, under \mbox{\bf ($\phi$-Erg)}, $\{0\}$ and $\mathbb C^d$ are not necessarily the unique $\mu$-invariant subspaces. Actually, as discussed at the end of \Cref{subs:framework}, this apparent restriction is an illusion.
	
	\subsection{Random walk on Dark subspaces.} As mentioned above, compared to \cite{BenFra}, we forgo assuming purification. In \cite{MaaKumm}, the authors show purification is equivalent to the absence of so-called dark subspaces of dimension greater than or equal to two. Dark subspaces are thus instrumental objects in our proofs. We especially study a Markov process on maximal dark subspaces introduced in \cite[Section~5]{MaaKumm}. Denoting by $\SUd$ the group of $d\times d$ special unitary matrices, the definition of dark subspaces goes as follows.
	\begin{definition*}
		A non-zero linear subspace $D\subset \cc^d$ is called a \emph{dark subspace} if for all $n\in\mathbb N^*$ and $\mutensorn$-almost all $(v_1,\ldots,v_n) \in \Ld^{\times n}$ there exists $\lambda  \in \cc$ and $u\in \SUd$ such that    \begin{equation*} 
			v_n\cdots v_1\big|_D = \lambda u\big|_D. 
		\end{equation*}
	\end{definition*}
	
	Note that any one-dimensional subspace satisfies the definition. 
	Since on dark subspaces any product of matrices $v\in \supp\mu$ is equivalent to a mapping proportional to a unitary matrix, the evolution induced by the product of matrices on the dark subspace can be reversed almost surely according to the postulates of quantum mechanics. That makes dark subspaces of high dimension potential candidates for quantum error correcting codes -- see \cite{knil,blume2008characterizing,knill2000theory}.  
	
	\medskip
	A first step towards the classification of $\Pi$-invariant measures and a first relevant result on its own is the proof that a Markov process on maximal dark subspaces admits a unique invariant measure.
	
	Let $\mathcal D_m$ be the set of dark subspaces of maximal dimension. We call them the maximal dark subspaces. Let $r_m$ be their dimension.
	Let $(D_n)_n$ be the Markov chain on $\mathcal D_m$ defined by
	\[D_n\leadsto D_{n+1}=V_{n+1}D_n\]
	with $V_{n+1}$ distributed according to $\tr\!\big(v\tfrac{\pi_{D_n}}{r_m}v^*\big)
	\mathrm{d}\mu(v)$, where $\pi_{D_n}$ is the orthogonal projector onto $D_n$. This is the Markov chain on dark subspaces introduced in \cite[Section~5]{MaaKumm}.
	
	We denote $K$ the associated Markov kernel. It is given by
	\begin{equation*}
		K(D, B)=\int_{\Ld} {\mathbf{1}}_B(vD)\tr(v\tfrac{\pi_D}{r_m}v^*)\d\mu(v)
	\end{equation*}
	for any $D\in \mathcal D_m$ and $B$ a Borel subset of $\mathcal D_m$.
	Interpreting $K$ as an operator on the Banach space of continuous functions on $\mathcal D_m$, we recall that for any Borel probability measure $\chi$, $\chi K$ is the Borel probability measure defined by $\chi K(B)=\int_{\mathcal D_m} K(D,B)\d\chi(D)$.
	
	Our first relevant result is that, assuming $\mu$ is irreducible, there exists a unique $K$-invariant probability measure and the convergence towards it is exponential in the Wasserstein metric.
	\begin{theorem}\label{thm:invmeasDark}
		There exists a unique $K$-invariant Borel probability measure on $\mathcal D_m$.
		
		Moreover, denoting this unique $K$-invariant Borel probability measure $\chi_{\rm inv}$, there exist an integer $p$ and two positive constants $C$ and $\lambda <1$ such that for any Borel probability measure $\chi$ on $\mathcal D_m$,
		\[
		W_1\Bigg(\frac 1p \sum_{\ell=0}^{p-1}\chi K^{pn+\ell}, \chi_{\rm inv} \Bigg) \leq C\lambda^n 
		\]
		for every $n \in \nn^*$, where   $W_1$   denotes the  $1$-Wasserstein metric. 
	\end{theorem}

	The proof of this theorem is provided in \Cref{sec:invdark}. In addition, in \Cref{thm:conv to Dark} we show that the process $(\hat x_n)_n$ converges exponentially fast to $\bigcup_{D\in \D} \P(D)$.
	
	We will rely on \Cref{thm:invmeasDark} to construct the whole set of $\Pi$-invariant measures. In that endeavor we need to track the evolution inside the dark subspaces induced by the matrices in $\supp \mu$. Schematically, the matrices have two actions: they map a maximal dark subspace to another one and, thanks to the definition of dark subspaces, they induce a unitary transformation \emph{inside} the maximal dark subspaces. To separate these two evolutions, we introduce a notion of \emph{minimal family of isometries}.
	
	\subsection{Minimal families of isometries.} We split $\cc^d$ into $\mathcal D_m\times \cc^{r_m}$, where $\cc^{r_m}$ plays the part of a reference space where the evolution inside the dark subspaces is recorded. For that, we establish a correspondence between each $D \in \mathcal D_m$ and $\cc^{r_m}$ by fixing a family of maps $ \J = \{J_D\}_{ D \in \mathcal D_m}$, where each $J_D \colon \cc^{r_m} \to D$ is a linear isometry. This allows us to turn partial isometries on $\cc^d$ into special unitary matrices on the common reference space $\cc^{r_m}$.
	Then we define a group $G_\J$ as the smallest closed subgroup of $\mathcal {SU}(r_m)$ containing the set 
	\begin{equation*} 
		S:=\big\{u_{v,D} \in \mathcal {SU}(r_m) \,|\, v \in \supp\mu,\: D \in \supp\chi_{{\rm inv}},\: \tr(v \pi_D v^*)>0\big\},
	\end{equation*}
	where  $u_{v,D} \propto J^{-1}_{vD} v J_D$ with the phase adjusted  so that $\det u_{v,D} = 1$, provided that $\tr(v \pi_D v^*)>0$. That is, $u_{v,D}$ is a~special unitary matrix acting on $\cc^{r_m}$ when $v$ acts on the dark subspace $D$. Recall that $\chi_{\rm inv}$ is the unique $K$-invariant Borel probability measure.
	
	The choice of $\J$ is far from unique. Amongst all the possible choices, we single out some minimal ones.

	\begin{definition*}
		We say that  $\mathcal J =\{J_D:\cc^{r_m}\to D\}_{D\in \mathcal D_m}$  is   a \emph{minimal family}  if the group $G_\J$ is (up to a unitary equivalence) minimal as a subgroup of the special unitary group, i.e., for any other family $\tJ=\{\tilde{J}_D:\cc^{r_m}\to D\}_{D\in \mathcal D_m}$  there exists $Q \in \Urm$ such that $G_\J \subset Q G_{\tJ}Q^{-1}$.
	\end{definition*}
	
	Based on this definition we build the set of $\Pi$-invariant Borel probability measures.
	
	\subsection{Classification of invariant measures}
	Let $\pcr$ be the projective space of $\cc^{r_m}$. We use the same notational conventions as for $\P(\cc^d)$.
	Let $\J$ be a family of isometries as defined in the previous subsection. Fix $\hat x \in \pcr$.  
	Let $m_{\hat x,\J}$ be the uniform measure on the $G_\J$-orbit of $\hat x$, i.e., the image of the $G_\J$-Haar measure by $G_\J\ni u \mapsto u\cdot \hat x \in \pcr$. 
	Let the map back from $\mathcal D_m\times \pcr$ to $\pcd$ be
	\[\Psi_\J \colon \mathcal D_m \times \P(\cc^{r_m}) \ni (D,\hat z) \longmapsto J_D\cdot \hat z \in  \pcd.\]   
	We define $\nu_{\hat x,\J}$ as the image measure of $\chi_{\rm inv}\otimes m_{\hat x,\J}$ by $\Psi_\J$. 
	
	Direct computation shows that $\nu_{\hat x, \J}$ is $\Pi$-invariant for any $\hat x \in \pcr$. We denote the set of all such measures by 
	\[\Upsilon_\J = \{\nu_{\hat x,\J}\colon \hat x \in \pcr\}.\] 
	Importantly,  all minimal families lead to the same set of measures. Indeed, if  $\J_1$ and $\J_2$ are both minimal families of isometries, then by definition, $G_{\J_1}$ and $G_{\J_2}$ are unitarily conjugate. Hence, as proved in \Cref{thm:minimalfammeasures}, 
	$\Upsilon_{\J_1} = \Upsilon_{\J_2}.$
	Our main result is that, when $\J$ is minimal, these measures are exactly the $\Pi$-ergodic measures. The following theorem is a direct consequence of \Cref{thm:charPiinvmeas} and \Cref{cor:Piinv}.
	\begin{theorem}\label{thm:intro-ergodic} 
		If $\J$ is minimal, then 
		$\Upsilon_\J$ is the set of   $\Pi$-ergodic measures. In other words, a probability measure $\nu$ is $\Pi$-ergodic if and only if $\nu=\nu_{\hat x,\J}$ for some $\hat x\in\pcr$.

		As a consequence, we have a complete description of $\Pi$-invariant measures: a probability measure $\nu$ on $\pcd$ is  $\Pi$-invariant  if and only if
		\[\nu = (\Psi_\J)_\star\left(\chi_{{\rm inv}} \otimes \int_{ \P(\cc^{r_m})} m_{\hat x,\J}  \   \mathrm{d}\lambda(\hat x)\right)=\int_\pcr \nu_{\hat x,\J}\d\lambda(\hat x)\]
		for some probability measure $\lambda$ on $ \P(\cc^{r_m})$.
	\end{theorem}
	
	Moreover, we provide a necessary and sufficient condition under which there exists a unique $\Pi$-invariant Borel probability measure. Let $\nu_{\rm unif}$ be the Borel probability measure defined by 
	\[\nu_{\rm unif}= \int_\pcr \nu_{\hat x,\J}\d\mathrm{Unif}(\hat x),\]
	where $\mathrm{Unif}$ is the uniform measure over $\pcr$, i.e., the image measure of the $\mathcal{SU}(r_m)$-Haar measure by the map $u\mapsto u\cdot\hat y$ for some $\hat y\in\pcr$. \Cref{thm:intro-ergodic} implies that it is always $\Pi$-invariant. The necessary and sufficient condition for its uniqueness reads as follows.
	\begin{theorem}\label{thm:introuniq} There exists a unique $\Pi$-invariant Borel probability measure on $\pcd$ if and only if the group $G_{\rm min}=G_\J$ induced by a minimal family of isometries $\J$ acts transitively on $\pc{r_m}$. Namely, if and only if:
		\begin{itemize}
			\item $G_{\rm min} = \mathcal{SU}(r_m)$ if $r_m$ is odd,
			\item $G_{\rm min} = \mathcal{SU}(r_m)$ or $G_{\min}^{\circ}$ is unitarily equivalent to ${\rm Sp}(r_m/2)$ if $r_m$ is even. Here, $G_{\min}^{\circ}$ is the identity component of $G_{\min}$, and  $\mathrm{Sp}(r_m/2)$ denotes the subgroup of $\mathcal {SU}(r_m)$ called the \emph{compact symplectic group} (or the \emph{hyperunitary group}). 
		\end{itemize}
		If there is a unique $\Pi$-invariant Borel probability measure, it is equal to $\nu_{\rm unif}$.
	\end{theorem} 
	This theorem is a consequence of \Cref{cor:uniquePiinvmeas}.

	\subsection{Outline of the proof of \Cref{thm:intro-ergodic}.}
	\Cref{thm:invmeasDark} is the first step of the proof. It is proved in \Cref{sec:invdark}. The proof is similar to that of \cite{BenFra} with non-trivial adaptations to take into account the possibility that $r_m>1$. It not only shows there exists a unique $K$-invariant Borel probability measure, but its proof also allows us to show the process $(\hat x_n)_n$ in $\pcd$ concentrates exponentially fast on  $\bigcup_{D\in \D} \P(D)$. This is the content of \Cref{thm:conv to Dark}. Hence, we can restrict our search for $\Pi$-invariant measures to measures supported on this union of projective spaces.
	
	It remains to show that the measures $\nu_{\hat x,\J}$ are ergodic and are the only $\Pi$-ergodic Borel probability measures. For that we need a
	natural correspondence between the dark subspaces; we shall study certain
	``minimal'' families of isometries in \Cref{sec:proc}. In particular, \Cref{thm:smartiffminimal} provides an equivalence between minimal families and so-called \emph{smart} families, which are families of isometries constructed from a reference maximal dark subspace and products of matrices from $\supp \mu$.  
	Then, in \Cref{thm:piminimalsets}, using the smartness property, we prove that any closed minimal $\Pi$-invariant set is of the form $\supp\nu_{\hat x,\J}$, where $\hat x\in \pcr$ and $\J$ is an arbitrary fixed minimal family of isometries. To the best of our knowledge, these last two intermediate results are new and constitute the main technical difficulties of the proof.
	
	For equicontinuous Markov kernels (i.e., Markov kernels $T$ such that for any continuous function $f$, $(T^nf)_n$ is equicontinuous) there is a one-to-one correspondence between minimal sets and ergodic measures -- see \cite[Proposition~3.2]{Raugi1992} and \cite[Proposition~2.9]{BenoistQuint2014}. In \Cref{thm:Piequicont} we show that $\Pi$ is equicontinuous. Then, leveraging the identification of $\Pi$-minimal sets, we obtain the classification of $\Pi$-ergodic measures. 
	
	The structure of the proof is summarized in \Cref{fig:diagram1}.

	\begin{figure}[ h]
		\begin{center}
			\scalebox{0.88}{

				\tikzset{every picture/.style={line width=0.75pt}}  
				
				\begin{tikzpicture}[x=0.75pt,y=0.75pt,yscale=-1,xscale=1]

					\draw  [color={rgb, 255:red, 74; green, 74; blue, 74 }  ,draw opacity=1 ] (235.67,41) .. controls (235.67,30.57) and (244.13,22.12) .. (254.56,22.12) -- (435.85,22.12) .. controls (446.28,22.12) and (454.74,30.57) .. (454.74,41) -- (454.74,97.67) .. controls (454.74,108.1) and (446.28,116.56) .. (435.85,116.56) -- (254.56,116.56) .. controls (244.13,116.56) and (235.67,108.1) .. (235.67,97.67) -- cycle ;
					
					\draw [color={rgb, 255:red, 74; green, 74; blue, 74 }  ,draw opacity=1 ]   (454.81,68.98) .. controls (505.54,68.32) and (482.95,106.87) .. (563.66,108.49) ;
					\draw [shift={(566.13,108.52)}, rotate = 180.6] [fill={rgb, 255:red, 74; green, 74; blue, 74 }  ,fill opacity=1 ][line width=0.08]  [draw opacity=0] (10.72,-5.15) -- (0,0) -- (10.72,5.15) -- (7.12,0) -- cycle    ;
					
					\draw [color={rgb, 255:red, 74; green, 74; blue, 74 }  ,draw opacity=1 ]   (232.1,68.65) .. controls (188.03,74.55) and (234.95,119.51) .. (133.5,117.89) ;
					\draw [shift={(235.67,68.27)}, rotate = 175.5] [fill={rgb, 255:red, 74; green, 74; blue, 74 }  ,fill opacity=1 ][line width=0.08]  [draw opacity=0] (10.72,-5.15) -- (0,0) -- (10.72,5.15) -- (7.12,0) -- cycle    ;
					
					\draw [color={rgb, 255:red, 74; green, 74; blue, 74 }  ,draw opacity=1 ]   (232.1,204.23) .. controls (188.03,198.32) and (234.95,153.36) .. (133.5,154.98) ;
					\draw [shift={(235.67,204.6)}, rotate = 184.5] [fill={rgb, 255:red, 74; green, 74; blue, 74 }  ,fill opacity=1 ][line width=0.08]  [draw opacity=0] (10.72,-5.15) -- (0,0) -- (10.72,5.15) -- (7.12,0) -- cycle    ;
					
					\draw  [color={rgb, 255:red, 74; green, 74; blue, 74 }  ,draw opacity=1 ] (7.87,106.71) .. controls (7.87,95.61) and (16.86,86.61) .. (27.96,86.61) -- (112.7,86.61) .. controls (123.8,86.61) and (132.79,95.61) .. (132.79,106.71) -- (132.79,166.98) .. controls (132.79,178.08) and (123.8,187.08) .. (112.7,187.08) -- (27.96,187.08) .. controls (16.86,187.08) and (7.87,178.08) .. (7.87,166.98) -- cycle ;
					
					\draw [color={rgb, 255:red, 74; green, 74; blue, 74 }  ,draw opacity=1 ]   (454.13,205.94) .. controls (504.87,206.6) and (482.28,167.85) .. (562.99,166.22) ;
					\draw [shift={(565.47,166.18)}, rotate = 179.4] [fill={rgb, 255:red, 74; green, 74; blue, 74 }  ,fill opacity=1 ][line width=0.08]  [draw opacity=0] (10.72,-5.15) -- (0,0) -- (10.72,5.15) -- (7.12,0) -- cycle    ;
					
					\draw  [color={rgb, 255:red, 74; green, 74; blue, 74 }  ,draw opacity=1 ] (235.67,177.67) .. controls (235.67,167.24) and (244.13,158.78) .. (254.56,158.78) -- (435.85,158.78) .. controls (446.28,158.78) and (454.74,167.24) .. (454.74,177.67) -- (454.74,234.33) .. controls (454.74,244.77) and (446.28,253.22) .. (435.85,253.22) -- (254.56,253.22) .. controls (244.13,253.22) and (235.67,244.77) .. (235.67,234.33) -- cycle ;
					
					\draw  [color={rgb, 255:red, 74; green, 74; blue, 74 }  ,draw opacity=1 ] (566.13,106.71) .. controls (566.13,95.61) and (575.13,86.61) .. (586.23,86.61) -- (665.11,86.61) .. controls (676.2,86.61) and (685.2,95.61) .. (685.2,106.71) -- (685.2,166.98) .. controls (685.2,178.08) and (676.2,187.08) .. (665.11,187.08) -- (586.23,187.08) .. controls (575.13,187.08) and (566.13,178.08) .. (566.13,166.98) -- cycle ;
					
					\draw (340,19.55) node [anchor=south] [inner sep=0.75pt]  [font=\footnotesize] [align=left] {\textit{dynamics ``between'' dark spaces}};
					
					\draw (300,158.37) node [anchor=south] [inner sep=0.75pt]  [font=\footnotesize] [align=left] {\textit{``inner'' dynamics}};

					\draw (575,105) node [anchor=north west][inner sep=0.75pt]  [font=\normalsize] [align=left] {\begin{minipage}[lt]{73.33pt}\setlength\topsep{0pt}
							\begin{center}
								{ \vspace{1mm}
									\small $\{\nu_{\hat x,\J}\}_{\hat x\in \pcr}$ \\[0.3em] are all $\Pi$-ergodic prob. measures} 
							\end{center}
					\end{minipage}};

					\draw (136,83) node [anchor=north west][inner sep=0.75pt]  [font=\small] [align=left] 
					{{ Kernel $K$}\\[-0.25em] { on $\mathcal{D}_{m}$}};
					
					\draw (250,36.64) node [anchor=north west][inner sep=0.75pt]  [font=\normalsize] [align=left]  
					{\begin{minipage}[lt]{134.29pt}\setlength\topsep{0pt}
							\begin{center}
								{ \vspace{2mm}
									{{ \small Random walk $(D_{n})_{n}$ on $\D$}}
									\\[0.75em]
									{{\small Existence of a unique \\[-0.2em] $K$-invariant  prob. measure}}
								}  
							\end{center}
					\end{minipage}};

					\draw (250,168) node [anchor=north west][inner sep=0.75pt]  [font=\normalsize] [align=left] {\begin{minipage}[lt]{134.29pt}\setlength\topsep{0pt}
							\begin{center}
								\vspace{1mm}
								{\small \mbox{$G_\J\subset \mathcal {SU}(r_m)$ minimal}}
								\\[0.4em]
								{  \small Construction of $\nu_{\hat x,\J}$
									\\[0.4em]
									Equivalence of minimal \\[-0.25em] and smart families
								}
							\end{center}
					\end{minipage}};

					\draw (136,170) node [anchor=north west][inner sep=0.75pt]  [font=\small] [align=left] {\begin{minipage}[lt]{50.7pt}\setlength\topsep{0pt}
							Family of  \\[-0.25em]  isometries 
							\vspace{-2mm}
							$$\vspace{-2mm}
							\mathcal{J}\! = \!(J_{D})_{D\in \D}$$
							$$J_{D}\colon \mathbb{C}^{r_{m}}\to D$$ 
					\end{minipage}};

					\draw (455,116.5) node [anchor=north west][inner sep=0.75pt]  [font=\fontsize{0.95em}{0.99em}\selectfont] [align=left] 
					{Equicontinuity of $\Pi$};
					
					\draw (455,134.5) node [anchor=north west][inner sep=0.75pt]  [font=\fontsize{0.95em}{0.99em}\selectfont] [align=left] 
					{Identification of \\[-0.2em] $\Pi$-minimal subsets};

					\draw (11,105) node [anchor=north west][inner sep=0.75pt]  [font=\small] [align=left] 
					{\begin{minipage}[lt]{86.58pt}\setlength\topsep{0pt}
							\begin{center}
								{ 
									\vspace{1mm}
									$(\hat{x}_n)_{n}$  valued in dark subspaces} 
								\\[0.75em]
								{Kernel $\Pi$ on $\pcd$}
							\end{center}   
					\end{minipage}};   
				\end{tikzpicture} 
			}
		\end{center}
		\caption{
			Structure of the proof of \Cref{thm:intro-ergodic}  
		}
		\label{fig:diagram1}
	\end{figure}
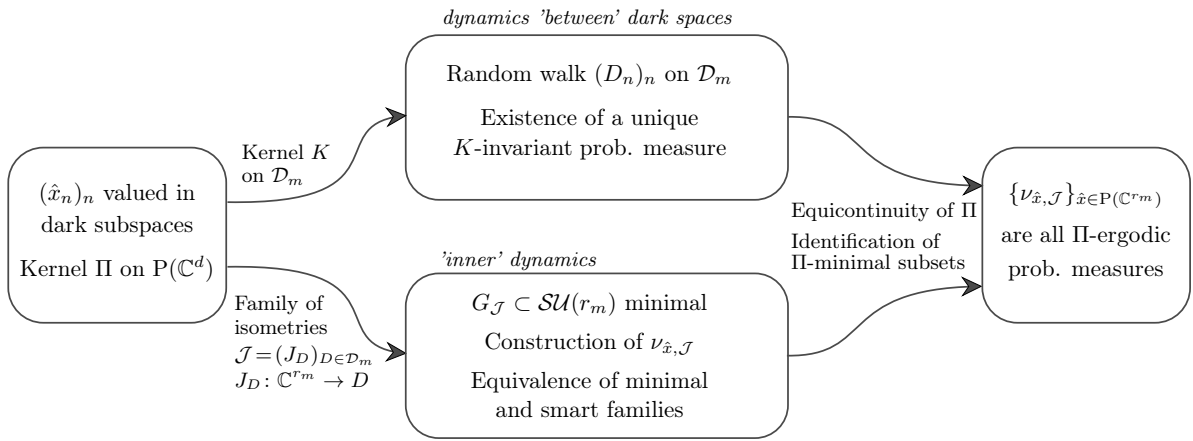    
	
	\subsection{Structure of the paper.} In \Cref{sec:traj_dark}, we give a precise definition of quantum trajectories and, in particular, the probability space on which we work. The notion of dark subspaces is then described as well as that of maximal dark subspaces. In \Cref{sec:invdark}, we introduce  the Markov chain on dark subspaces and show that it admits a unique invariant measure. Furthermore, we show that the convergence towards this measure is exponentially fast. This section concerns \Cref{thm:invmeasDark}. Next, in \Cref{sec:convergence}  we show that quantum trajectories converge exponentially fast towards maximal dark subspaces. \Cref{sec:proc} discusses the isometries between the maximal dark subspaces and $\cc^{r_m}$ and introduces the notion of minimal families of isometries. Then \Cref{sec:invmeasures} provides  the classification of invariant measures for quantum trajectories, i.e., \Cref{thm:intro-ergodic} and \Cref{thm:introuniq} after proving the equicontinuity of $\Pi$ and the characterization of $\Pi$-minimal closed subsets. In the final \Cref{sec:examples} we discuss some enlightening examples that provide an illustration of the notions we use and various subtleties that can arise. Note that all these results concern the invariant measures for quantum trajectories on $\pcd$. In \Cref{sec:app}, we provide the classification of all invariant measures of quantum trajectories on density matrices.

	\numberwithin{theorem}{section}
	\section{Quantum Trajectories, Dark subspaces}\label{sec:traj_dark}
	\addtocontents{toc}{\protect\setcounter{tocdepth}{2}}
	
	We begin this section by presenting the general mathematical model of quantum trajectories and the Markov chain on dark subspaces. Then, we introduce a process which has been one of the key tools in \cite{BenFra} and which turns out to be intricately linked to maximal dark subspaces. In particular, we use it in \Cref{sec:convergence}  to show  
	that quantum trajectories converge exponentially fast towards maximal dark subspaces.

	\subsection{Framework and main assumption}\label{subs:framework} 
	Let us first introduce some notation.  
	For a fixed dimension $d \geq 2$ we consider the complex vector space $\cc^d$ as a Hilbert space equipped with the usual Hermitian inner product anti-linear in the first coordinate. We put $\Ld$ for the set of linear maps on $\cc^d$ equipped with the operator norm, $\Ud$ for the group of $d\times d$ unitary matrices, $\SUd$ for the associated special unitary group, and $\Sd$ for the set of \emph{density matrices} (also called \emph{states}), i.e.,
	\[\Sd =\{\rho\in \Ld\ | \ \rho\geq0,\, \tr(\rho)=1\}.\]
	The extreme points of $\Sd$   are the rank-one orthogonal projectors on $\cc^d$. They constitute the set  of \emph{pure states}, which is isomorphic to
	the complex projective space~$\pcd$, i.e.,  the space of rays in $\mathbb \cc^d$.  Throughout the paper we use the  following identifications. For a non-zero vector $x\in \cc^{d}$ we denote  by $\pi_{\hat x}$  the orthogonal projection on $\cc x$, and by $\hat x$  the equivalence class of $x$ in  $\pcd$. In particular, if $v \in \Ld$, then  $v \cdot \hat x$ stands for  the equivalence class of  $vx$, provided that $v x \neq 0$; i.e., $v \cdot \hat x = \widehat{{vx}}$. 
	For $\hat x \in \pcd$ we denote by  $x$   an arbitrary normalized  representative of $\hat x$ in $\cc^d$.  Also, we put $\nnone$ for the set of positive integers and  $\nnzero = \nnone \cup \{0\}$.

	Now we introduce more precisely the model of quantum trajectories. Let $\mu$ be a Borel measure on $\Ld$ that has a finite second moment, $\int_{\Ld} \|v\|^2 \d \mu(v) < \infty$, and satisfies the stochasticity condition
	\begin{equation}
		\label{eq:stoch}
		\int_{\Ld} v^*v \d \mu(v) = \id.
	\end{equation}

	Such a measure $\mu$ defines a \emph{quantum trajectory}. It is a Markov chain that takes values in $\Sd$ and has the following transition kernel
	\begin{equation}\label{eq:kernelPi}
		\Pi  (\rho, B) = \int_{\Ld} {\mathbf{1}}_{B}\left(\frac{v \rho v^*}{\tr(v \rho v^*)}\right) \tr(v \rho v^*) \d\mu(v),
	\end{equation}
	where $\rho \in \Sd$ and $B$ is a Borel subset of $\Sd$. In terms of stochastic processes, one can define the Markov chain $(\rho_n)_{n\in \nn}$ valued in $\Sd$ by
	\begin{equation}
		\rho_n\longmapsto\rho_{n+1}=\frac{V_{n+1}\rho_n V_{n+1}^*}{\tr(V_{n+1}\rho_n V_{n+1}^*)},
	\end{equation}
	where $V_{n+1}$ is a matrix-valued random variable whose law is $\tr(v\rho_n v^*)\d\mu(v)$. Note that if $\rho_0=\pi_{\hat x_0}$ for some unit vector $x_0 \in \cc^d$, we have 
	\[\frac{v\pi_{\hat  x_0} v^*}{\operatorname{tr}(v\pi_{\hat  x_0} v^*)}=  \pi_{v\cdot \hat x_0}\]
	for every $v \in \Ld$ such that $v x_0\neq 0$.
	That is, if the initial state is a rank-one projector, the quantum trajectory remains valued in rank-one projectors. Using the identification between rank-one projectors and equivalence classes of vectors in $\pcd$, we recover the description of (pure) quantum trajectories provided in the Introduction. Recall that the transition kernel then reads
	\begin{equation} \label{eq_deftranskernel}
		\Pi(\hat x,B)=\int_{\Ld} \mathbf{1}_{B}(v\cdot \hat x)  \|v x\|^2 \d \mu(v),
	\end{equation}
	where $\hat x\in \pcd$ and  $B$ is a Borel subset of $\pcd$, and the Markov chain $(\hat{x}_n)_{n \in \nnzero}$ on $\pcd$ is given by 
	\[
	\hat x_n\longmapsto\hat x_{n+1}=V_{n+1}\cdot \hat x_{n},
	\]
	where $V_{n+1}$ has law $||v x_n||^2 \d \mu(v)$. 
	Our main goal is to classify the invariant measures of $(\hat x_n)_{n\in\nnzero}$. Analogous results for general quantum trajectories, which for the most part follow via a straightforward adaptation of the methods and results derived for pure trajectories, are relegated to  Appendix \ref{sec:app}.
	
	For a topological space $\mathcal T$, we denote $\mathcal C(\mathcal T)$ the set of continuous functions $\mathcal T\to \cc$.
	Then, on $\mathcal C(\Sd)$ and $\mathcal C(\P(\cc^d))$, the kernel $\Pi$ acts as
	\[\Pi f(\rho)=\int_{\Ld} f\left(\frac{v \rho v^*}{\tr(v \rho v^*)}\right) \tr(v \rho v^*) \d\mu(v)\]
	and
	\[\Pi f(\hat x)=\int_{\Ld} f(v\cdot \hat x)  \|v x\|^2 \d \mu(v),\]
	respectively. Also, by duality, the kernel acts on measures as 
	\[\nu\Pi(B)=\int_{\Sd}\Pi(\rho,B)\d\nu(\rho)\quad\mathrm{or}\quad \nu\Pi(B)=\int_{\pcd}\Pi(\hat x,B)\d\nu(\hat x),\]
	where $\nu$ is a Borel measure on $\Sd$ or $\pcd$, and $B$ is a Borel subset of the respective set. An \textit{invariant probability measure} for $\Pi$ is a probability measure $\nu$ that satisfies $\nu\Pi=\nu$. The set of invariant probability measures is convex and its extreme points  are called \textit{ergodic measures}.
	
	\smallskip
	
	Since the kernel $\Pi$ is Feller (see \cite[Prop 2.2]{BenFat}) and $\Sd$ is compact, quantum trajectories always admit  an invariant measure. An important question is the uniqueness of this measure and convergence towards it. A key notion in this context is that of a \emph{dark subspace}, which was introduced by K\"{u}mmerer and Maassen in \cite{MaaKumm}. Let us recall its definition.
	
	\begin{definition}
		A non-zero linear subspace $D\subset \cc^d$ is called a \emph{dark subspace} when for all $n\in\nnone$ and $\mutensorn$-almost all $(v_1,\ldots,v_n) \in {\Ld^{\times n}}$ there exists $\lambda  \geq 0$ and $U\in \Ud$ such that 
		\begin{equation}\label{darkprop2}
			v_n\cdots v_1x = \lambda Ux \ \   \ \     \forall x\in D
		\end{equation}
		Equivalently, $D \subset \cc^d$  is dark if and only if 
		\begin{equation}\label{darkprop}
			\pi_D v_1^*\cdots v_{n}^*v_{n}\cdots v_{1}\pi_D \propto \pi_D 
		\end{equation}
		for all $n \in \nnone$ and  $\mutensorn$-almost all  $(v_1,\ldots,v_n)$, where   $\pi_D$ stands for the orthogonal projection on $D$.
		The orthogonal projection on a dark subspace is called a \emph{dark projection}.
		
		\emph{Maximal dark subspaces} are those with  highest dimension. We denote this dimension by $r_m$ and the set of maximal dark subspaces by $\mathcal D_m$.
	\end{definition}

	Note that every one-dimensional subspace satisfies \Cref{darkprop2} trivially, so $r_m$ is well defined. In the case $r_m=1$ we say that purification occurs. This condition is equivalent to contractivity in the theory of products of random matrices -- see \cite[Proposition~A.1]{BenFra}.

	\begin{remark}
		According to the original definition in \cite{MaaKumm}, dark subspaces are required to have dimension of at least two; in particular, dark subspaces may not exist for a given $\mu$.
		However, for the sake of consistency, in the present paper  we consider rays in $\cc^d$ as (trivial) dark subspaces and rank-one projections as (trivial) dark projections.
	\end{remark}

	The following straightforward proposition gathers basic properties of dark subspaces. 
	In particular, it states explicitly that the probability of the system undergoing a given update depends only on the dark subspace which the system occupies, i.e., it is constant for all states supported on this dark subspace. This means in particular that if two dark subspaces intersect non-trivially, they share all transition probabilities.
	We do not provide a proof since it follows directly from the definition -- see also \cite[Sec. 5]{MaaKumm}.

	\begin{proposition}\label{prop:darkdefprop} Let $D$ be a dark subspace.
		\begin{enumerate}[label={\rm (\roman*)}]
			\item 
			Let $v\in \supp \mu$. If $\lambda \geq 0$  and $U\in\Ud$ are such that
			$v|_D = \lambda U|_D$, which is equivalent to $\pi_D v^*v\pi_D=\lambda^2\pi_D$, then for every unit vector $x \in D$, 
			\[\|vx\|^2 =\lambda^2=\tr(v\tfrac{\pi_D}{\tr(\pi_D)}v^*).\]
			Furthermore, if $\lambda>0$, then $vD$ is a dark subspace {and $\dim vD = \dim D$}. 
			
			\item  Since every point $(v_1, \ldots, v_n) \in \supp \mutensorn$ can be approximated by a sequence of points from $\supp \mutensorn$ that satisfy \Cref{darkprop},  the continuity of matrix multiplication implies that  \Cref{darkprop} holds for every $(v_1, \ldots, v_n) \in \supp\mutensorn$. 
		\end{enumerate}   
	\end{proposition}
	
	One of the main results of \cite{MaaKumm} is that the asymptotic behaviour of  quantum trajectories is to perform a random walk between dark subspaces. More precisely, the transition rule  
	\[ D_n \longmapsto D_{n+1}=v D_n\]
	with probability  \[\tr(v\frac{\pi_{D_n}}{\operatorname{tr}(\pi_{D_n})}v^*)\d\mu(v)\] defines a Markov chain $(D_n)_{n \in \nnzero}$ in $\mathcal D_m$.  This Markov chain is well defined by \Cref{prop:darkdefprop}(i). The corresponding kernel is 
	\begin{equation*} 
		K(D, B)=\int_{\Ld} {\mathbf{1}}_B(vD)\tr(v\tfrac{\pi_D}{r_m}v^*)\d\mu(v),
	\end{equation*}
	where $D \in \mathcal D_m$ and $B$ is a Borel subset of $\mathcal D_m$.
	Let us point out that in terms of  normalized  dark projectors  this transition reads
	\[\rho_n=\frac{\pi_{D_n}}{\tr(\pi_{D_n})} \longmapsto \rho_{n+1}=\frac{v \pi_{D_n}v^*}{\tr(v \pi_{D_n}v^*)}=\frac{\pi_{D_{n+1}}}{\tr(\pi_{D_{n+1}})},\]
	and so it coincides with the Markov chain defined by $\Pi$ on density matrices.

	\smallskip 
	
	As already mentioned in the Introduction, our only assumption throughout the paper is that $\mu$ satisfies the following irreducibility assumption.
	
	\smallskip
	\noindent 
	\textbf{Irreducibility.} There does not exist a non-trivial subspace $E\subset \mathbb C^d$ (\emph{i.e.} $E\notin\{\cc^d,\{0\}\}$) such that $\mu$-almost surely we have $vE\subset E$.
	\smallskip 
	
	\noindent 
	In \cite{BenFra}, assumption {\bf ($\phi$-Erg)} corresponded to the existence of a unique minimal non-zero $\mu$-invariant subspace $E$. Here,  the irreducibility assumption moreover enforces $E=\cc^d$. Both assumptions are related to the uniqueness of the fixed point of some map on $\Sd$ called a quantum channel.
	
	Let
	\[\phi\colon X\mapsto\int_{\Ld} vX v^*\d\mu(v).\]
	This map is positive and \Cref{eq:stoch} implies that $\Sd$ is invariant for $\phi$. Then, following the  Perron--Frobenius theorem for positive maps on matrix algebras -- see \cite{Evans1978} -- the irreducibility assumption is equivalent to $\phi$ having a unique fixed point $\rho_\infty$ in $\Sd$ and this matrix is of full rank.  
	Following \cite[Proposition~5.7]{carbone2016irreducible}, assumption {\bf ($\phi$-Erg)} in \cite{BenFra} is equivalent to $\phi$ having a unique fixed point $\rho_\infty$ in $\Sd$; however, this matrix can be rank deficient. 
	
	When one is only concerned with invariant measures, the apparently stronger condition of irreducibility is actually not a restriction. Indeed, assume {\bf ($\phi$-Erg)} from \cite{BenFra} holds and let $E=\supp(\rho_\infty)$. Then,
	following \cite[Theorem~1.3]{benoist2017exponential}, quantum trajectories $(\rho_n)_{n}$ are almost surely asymptotically supported on $E$ and the convergence to it is exponentially fast. Hence, all invariant measures only charge states with support in $E$. 
	Now, since the restriction of $\phi$ to states supported in $E$ has a unique fixed point in $\Sd$ that is of full rank on $E$, the restricted channel $\phi\vert_E$ is irreducible. Hence, assuming irreducibility is not restrictive and our result can be extended assuming only {\bf ($\phi$-Erg)} from \cite{BenFra}.

	\subsection{Probability space}
	We work with a specific probability space on which we define a process that has the same law as the quantum trajectory defined in \Cref{eq:kernelPi}. It is essentially the one used in \cite{BenFra}.
	Let $\Omega$ be the set of infinite sequences of maps $\Ld^{\nnone}$.  For any $\omega\in \Omega$, let $\omega = (v_1, v_2, \ldots)$, and let $\pi_n \colon \Omega \to \Ld^{\times n}$  be the canonical projection on the first $n$ components, i.e. $\pi_n(\omega) = (v_1, \ldots, v_n)$.
	The usual cylinder $\sigma$-algebra on $\Omega$ is denoted $\mathcal O$. That is, denoting by $\mathfrak M$ the Borel $\sigma$-algebra on $\Ld$, for $n\in\nnone$ we define the
	$\sigma$-algebra $\O_n$ as  $\pi^{-1}_n( \mathfrak M^{\otimes n})$, i.e., $\O_n$ is the smallest $\sigma$-algebra on $\Omega$ such that all $n$-cylinder sets $\{(v_1, \ldots, v_n, \ldots ) \in \Omega \: |\: v_i \in A_i,\, i=1, \ldots, n\}$, where $A_1, \ldots, A_n \in \mathfrak M$,  are   measurable. 
	Then $\mathcal O$ is the smallest $\sigma$-algebra such that $\mathcal O_n\subset \mathcal O$ for all $n\in \nnone$.
	
	Throughout the paper we use a series of identifications which improve the readability even if they constitute small abuses of notation. Namely, a cylinder set $O_n \in \O_n$ is identified with its base $\pi_n(O_n)$, a function $f$ on $\Ld^{\times n}$ with $ f \circ \pi_n$ on $\Omega$, and a measure $\mutensorn$ on $(\Ld^{\times n}, \mathfrak M^{\otimes n})$ with $\mutensorn\circ \pi_n$ on $(\Omega, \O_n)$.
	Note that we cannot extend $\mu$ to the whole $\Omega$ as it is not assumed to be finite. 
	
	For any $\rho \in \Sd$ we define a law $\pp^\rho$ for the sequences in $\Omega$. It is defined as
	\[\mathbb P^{\rho}(O_n)  = \int_{ O_n } \tr(v_{n}\cdots v_{1}\rho v_{1}^*\cdots v_{n}^*)\d \mutensorn(v_1, \ldots, v_n),\]
	where $O_n \in \O_n$, $n \in \nnone$. The stochasticity condition, see  \Cref{eq:stoch}, guarantees this sequence of measures is a consistent family of probability measures, so via Kolmogorov's extension theorem it defines a unique probability measure $\pp^\rho$ on $\Omega$. In quantum mechanics, this law is that of repeated quantum measurements modeled by $\mu$.
	
	For the maximally mixed  {(chaotic)}   state $\Id/d$ we put $\mathbb P^{\rm ch}=\mathbb P^{\Id/d}$. We often use $\mathbb P^{\rm ch}$   as a reference measure. Importantly, since $\rho\mapsto\pp^\rho$ is positive and extends to a linear map from $\Ld$ to finite complex measures, for every $\rho \in \Sd$,
	\[\mathbb P^\rho \ll \mathbb P^{\rm ch}.\]
	Indeed, for any $\rho\in\Sd$, $\rho \leq \|\rho\|  \Id$. Hence, $\pp^\rho\leq d\|\rho\|\ppch$. 
	
	Let $\mathcal B$ be the Borel $\sigma$-algebra  on $\Sd$  and   $\nu$   a Borel probability measure on $\Sd$. We define the probability measure $\mathbb P_\nu$ on $(\Sd \times \Omega,\,  \mathcal B\otimes \mathcal O)$ by
	\begin{equation}\label{eq:pmu}
		\mathbb P_\nu(B\times O)= \int_{B} \pp^\rho(O)\d\nu(\rho)
	\end{equation}
	for  $B\in \mathcal B$ and  $O \in \O$. Note that the restriction of $\pp_\nu$ to $\mathcal B \otimes \{\varnothing, \Omega\}$ is, by construction, equal to $\nu$.
	
	Let $\ee_{\nu}$ denote the expectation with respect to $\pp_\nu$. Then, convexity of $\Sd$ implies $\rho_\nu = \mathbb E_\nu(\rho)$ is an element of $\Sd$. Moreover, since $\rho\mapsto\pp^\rho$ is affine by definition, the marginal of $\mathbb P_\nu$ on $\O$ is $\pp^{\rho_\nu}$, i.e.
	\[\mathbb P_\nu(\Sd \times O) = \pp^{\rho_\nu}(O)\]
	for any $O\in \O$.    We identify the sub-$\sigma$-algebra $\{\varnothing,\Sd\} \otimes \O$ with $\O$, and an $\O$-measurable function $f$ on $\Omega$ with the $\B \otimes \O$-measurable function $f$ on $\Sd \times \Omega$ defined as $f(\rho, \omega) = f(\omega)$ for all $(\rho, \omega) \in \Sd \times \Omega$. This construction gives  $\ee_{\nu}(  f) =  \ee^{\rho_\nu}(f)$ for any  $\O$-measurable $f$.
	
	\bigskip
	We now construct on $(\Sd\times\Omega,\pp_\nu)$ a process whose law is that of the  Markov chain defined by $\Pi$ and the initial law $\nu$.
	Let $V_i \colon \Omega \to \Ld$, $i \in \nnone$, be the $i^{\text{th}}$ coordinate random variable
	\[V_i (v_1,v_2,\ldots,v_i,\ldots)   =   v_{i},\]
	and let $(W_n)_{n}$ be the $(\O_n)_n$-adapted sequence of products of them:
	\[W_n=V_{n}\cdots V_1\]
	for $n\in\nnone$. For the sake of consistency we set $W_0 = \id$.
	Next, we define the $(\B \otimes \O_n)_n$-adapted sequence of random states $(\rho_n)_{n}$ on $(\Sd\times \Omega,\,\mathcal B \otimes\mathcal O,\,\mathbb P_\nu)$ as $\rho_0(\rho, \omega) = \rho$ for every $\rho \in \Sd$ and $\omega \in \Omega$, and
	\[\rho_{n}=\frac{W_n^{\phantom{*}} \rho_0 W_n^*}{\tr(W_n^{\phantom{*}} \rho_0 W_n^*)}\]
	for $n\in\nnone$.
	Note that 
	\[\mathbb P_\nu\big(\operatorname{tr}(W_n^{\phantom{*}} \rho_0 W_n^*)=0\big)=0\] for every $n\in\nnzero$, so   $(\rho_{n})_{n}$ is  $\mathbb P_\nu$-almost surely well defined. Moreover,  $(\rho_n)_{n}$ has the same distribution as
	the Markov chain with the initial distribution $\nu$ and  transition kernel $\Pi$.

	\section{Invariant measure on Dark subspaces, proof of \Cref{thm:invmeasDark}}\label{sec:invdark}
	In this section we prove \Cref{thm:invmeasDark}, i.e., the exponential convergence towards the unique invariant measure on maximal dark subspaces. The proof has the same structure as the one in \cite{BenFra}. After defining some topological structure on the set of Dark subspaces $\D$ and defining an explicit realization of the Markov chain, we prove two geometric convergences. First, that there exists an $\O$-measurable estimator $\widehat{D}_n$ of $D_n$ with their relative distance decreasing exponentially fast. Second, that for any $\O$-measurable random variable, some shift of $\pp^\rho$ converges exponentially fast in total variation to $\pp^{\rho_\infty}$. We conclude by piecing together these two convergences.
	
	\subsection{The structure of $\D$}
	Dark subspaces are linear subspaces of $\cc^d$, hence elements of a Grassmannian. 
	We denote $\mathcal G_{k,d}(\mathbb C)$ the set of $k$-dimensional subspaces of $\cc^d$ and equip it with the \emph{Grassmann topology}, which is the final topology induced by the function mapping a complex matrix of size $d \times k$ with linearly independent columns to the linear span of these columns. 
	
	With respect to this topology, $\mathcal G_{k,d}(\mathbb C)$ is a compact space  -- see \cite[I-1-7]{Ferrer94}.
	Moreover, the \emph{gap metric} defined by
	\[d_{\rm G}(Q_1, Q_2) = \|\pi_{Q_1}- \pi_{Q_2}\|,\]
	for $Q_1, Q_2 \in \mathcal G_{k,d}(\mathbb C)$, metrizes this topology -- see \cite[I-2-6]{Ferrer94}.
	Later on, we shall exploit the fact that the gap metric can be expressed as
	\begin{equation}\label{eq:gapmax}
		d_{\rm G}(Q_1, Q_2)=\max\left\{\sup\nolimits_{{x \in Q_1, \|x\|=1}}\dist(x,Q_2),\ \sup\nolimits_{{x \in Q_2, \|x\|=1}}\dist(x,Q_1)\right\},
	\end{equation}
	where   $\dist(x,Q)=\Vert x-\pi_Qx\Vert$  is the  distance between the vector  $x \in \cc^d$ and the subspace $Q \subset \cc^d$ -- see \cite[Theorem~13.1.1]{GohLan06}.  
	In fact, the two suprema in \Cref{eq:gapmax} coincide  -- see \cite[Lemma~3.2]{Morris10} -- so
	\begin{equation}\label{eq:gapsup}
		d_{\rm G}(Q_1, Q_2)= \sup\nolimits_{{x \in Q_1, \|x\|=1}}\dist(x,Q_2) =  \sup\nolimits_{{x \in Q_2, \|x\|=1}}\dist(x,Q_1).
	\end{equation}
	The last structural result we need is the compactness of $\mathcal D_m$.
	\begin{proposition}
		\label{prop:compact}
		The set $\mathcal D_m$ is compact.
	\end{proposition}
	\begin{proof}
		Since $\mathcal G_{r_m,d}(\mathbb C)$ is compact and $\mathcal D_m\subset \mathcal G_{r_m,d}(\mathbb C)$, it suffices to prove that $\mathcal D_m$ is closed. To this end, consider a sequence of maximal dark projectors $(\pi_n)_{n\in \nn}$ that converges to some projector $\pi$. We will show that $\pi$ is a dark projector of rank $r_m$. It is clear that $\pi$ is an orthogonal projector, because the set of orthogonal projectors is closed. As for the rank of $\pi$, we just have to check that $\tr(\pi)=r_m$, which is straightforward since $\operatorname{tr}(\pi_n)=r_m$ for all $n$ and the trace is continuous. 
		
		It remains to show that $\pi$ satisfies \Cref{darkprop}. Fix $k \in \nnone$. Since every $\pi_n$ is dark, for $\mutensor{k}$-almost all  $(v_1, \ldots, v_k)$,
		\begin{equation*} 
			\pi_n v^*_k \cdots v_1^* v_1 \cdots v_k \pi_n-\tfrac{1}{r_m}\tr(\pi_n v^*_k \cdots v_1^* v_1 \cdots v_k) \pi_n=0.
		\end{equation*}
		As a function of $\pi_n$, the left hand side is continuous. Therefore, for $\mu^{\otimes k}$-almost all $(v_1, \ldots, v_k)$,
		\begin{equation*}
			\pi v^*_k \cdots v_1^* v_1 \cdots v_k \pi=  \tfrac{1}{r_m} \tr(\pi v^*_k \cdots v_1^* v_1 \cdots v_k)\pi  
		\end{equation*} 
		and  the proposition is proved.
	\end{proof}
	\begin{remark}
		This proposition implies $\mathcal D_m$ is Borel measurable as a closed set.
	\end{remark}
	
	\subsection{Existence of an invariant measure}\label{sec:proof existence}
	Since \Cref{prop:compact} ensures that $\D$ is compact, the existence  part of \Cref{thm:invmeasDark} is proved if $K$ is Feller, i.e., if $Kf$ is continuous for any continuous function $f\colon \D\to\rr$.
	
	Let $f\colon\D\to\rr$ be continuous. Let $\psi_f\colon \D\times \supp \mu\to\rr$ be the continuous function defined by \[\psi_f(D,v)=f(vD)\tr(v\frac{\pi_D}{r_m}v^*)\] when $v\pi_Dv^*\neq 0$ and $0 $ otherwise. Then, for any fixed $v$, the map $D\mapsto \psi_f(D,v)$ is continuous and bounded by $\|f\|_\infty\|v\|^2$, which is $\mu$-integrable. Then, Lebesgue's dominated convergence theorem implies that \[Kf\colon D\mapsto \int_{\Ld}\psi_f(D,v)\d\mu(v)\] is continuous, so $K$ is Feller. Therefore, there exists a probability measure $\chi_{\rm inv}$ over $\D$ such that $\chi_{\rm inv}K=\chi_{\rm inv}$.
	
	The remaining subsections of this section are devoted to the proof of uniqueness and geometric convergence in \Cref{thm:invmeasDark}.
	
	\subsection{Probability space}
	As in the case of quantum trajectories, we define a probability space on which the Markov chain is explicitly realized.
	
	Recall that the transition kernel for the Markov chain on maximal dark subspaces reads 
	\begin{equation}\label{kernelK} 
		K(D, B)=\int_{\Ld} {\mathbf{1}}_B(vD)\tr\big(v\tfrac{\pi_D}{r_m}v^*\big)\d\mu(v),
	\end{equation}
	where $D \in \mathcal D_m$ and $B$ is a Borel subset of $\mathcal D_m$. 
	
	We put $\mathcal Q$ for the Borel $\sigma$-algebra on $\mathcal D_m$ and consider  a Borel probability measure  $\chi$  on $\mathcal D_m$. We define the probability measure $\mathbb P_{\chi}$ on  $(\mathcal D_m \times \Omega,\, \mathcal Q \otimes \mathcal O)$  analogously to how  we defined  $\mathbb P_{\nu}$ from a measure $\nu$ on the set of states $\Sd$, i.e., for  $B\in \mathcal Q$ and $O \in \O$  we put
	\begin{equation*}
		\mathbb P_\chi(B\times O)= \int_{B} \pp^{\pi_D/r_m}(O) \d\chi(D).
	\end{equation*}
	We write $\mathbb E_{\chi}$ for the expectation with respect to $\mathbb P_{\chi}$ and denote $\rho_{\chi}:=\mathbb E_{\chi}(\pi_D/r_m) \in \Sd$. Recall that $\pp^{\rho_{\chi}} \ll \pp^{\rm ch}$. Also, as before, an $\O$-measurable map $f$ on $\Omega$ is identified with the $\mathcal Q \otimes \O$-measurable map $f$ on $\D \times \Omega$ satisfying $f(D,\omega) = f(\omega)$ for all $D \in \D$ and $\omega \in \Omega$. Then  $\ee_{\chi}(f) = \ee^{\rho_\chi}(f)$.
	
	Now, on $(\mathcal D_m \times \Omega, \mathcal Q \otimes \mathcal O,\mathbb P_\chi)$ we define a $(\mathcal Q \otimes \O_n)_n$-adapted process $(D_n)_n$  by setting $D_0(D,\omega) = D$ for every $(D, \omega) \in \mathcal D_m \times \Omega$ and
	$D_n = W_nD_0$
	for $n \geq 1$. This process has the same distribution as the Markov chain on $\mathcal D_m$ with the initial distribution $\chi$ and transition kernel $K$, see \Cref{kernelK}.
	In terms of dark projectors, we have
	\[\frac{\pi_{D_n}}{r_m} = \frac{W_n\pi_{D_0}W_n^*}{\tr(W_n\pi_{D_0}W_n^*)}.\]
	Observe that
	$\mathbb P_{\chi}(W_nD_0 \in \mathcal D_m) = \mathbb P_{\chi}((W_n\pi_{D_0}W_n^*  ) > 0) = 1$,
	so the process under consideration is $\mathbb P_{\chi}$-almost surely well defined.

	\subsection{A useful process in $\Sd$}\label{sec:M_n}
	In \cite{BenFra}, a key object in the proofs was the $(\O_n)_n$-adapted sequence of random variables $(M_n)_{n}$ taking values in $\Sd$ and defined by
	\[M_n=\frac{W_n^*W_n^{\phantom{*}}}{\tr(W_n^*W_n^{\phantom{*}})}.\]
	Note that $M_n$ is $\mathbb P_\chi$-almost surely well defined  since it is $\O$-measurable, $\mathbb P^{\rho_\chi} \ll \mathbb P^{\rm ch}$ and  $\mathbb P^{\rm ch}\big(\operatorname{tr}(W_n^* W_n )=0\big)=0$.
	The following key proposition was proved in \cite{BenFra}:
	\begin{proposition}[{\cite{BenFra}, Proposition~2.2}]\label{prop:existence Minf}
		There exists an $\O$-measurable $\Sd$-valued random variable $M_\infty$ such that 
		\[\lim_{n \to \infty} M_n=M_\infty,\]
		where the limit holds $\pp_\chi$-almost surely and in $L^1(\pp_\chi)$-norm for any Borel probability measure $\chi$ over $\D$.
		
		Moreover, for any $\rho\in \Sd$,
		\[\frac{\d\pp^\rho}{\d\ppch}=d\tr(\rho M_\infty).\]
	\end{proposition}
	\begin{proof}
		Proposition~2.2 in \cite{BenFra} was expressed with respect to $\pp_\nu$ with $\nu$ a probability measure over $\P(\cc^d)$ instead of $\pp_\chi$. However, the proof of convergence is done with respect to $\ppch$, which implies it holds with respect to $\pp^\rho$ for any $\rho\in\Sd$ since $\pp^\rho\leq d\ \ppch$. Hence, the  limit claims hold with respect to $\pp_\chi$ by definition of this measure.
		
		The Radon-Nikodym derivative expression is proved in \cite[Proposition~2.2]{BenFra}.
	\end{proof}
	
	From the polar decomposition of $W_n$ follows the existence of a unitary-valued random variable $U_n\colon  \Omega \to \Ud $ satisfying \[W_n=U_n\sqrt{W_n^*W_n^{\phantom{*}}},\]
	where $n \in \nnzero$. We can choose $(U_n)_n$ so that it is  $(\O_n)_n$-adapted. Note that by definition of $M_n$, $\ppch$-almost surely, $W_n\propto U_n\sqrt{M_n}$. Since $U_n(\omega)\in \Ud$ and $\Ud$ is compact, the sequence $(U_n(\omega))_{n}$ almost surely has an accumulation point. Let us point out that, when $\omega$ is fixed, both the accumulation points, and the extraction subsequence leading to one, depend on $\omega$.  
	
	The following two propositions, essential in the sequel, will lead to the concentration of quantum trajectories on maximal dark subspaces.
	\begin{proposition}\label{prop:Minfdark} For any $\omega\in \Omega$, let $\mathsf U_{\rm acc}(\omega)$ be the set of accumulation points of $(U_n(\omega))_n$. Then, for any $\rho\in \Sd$, $\pp^\rho$-almost surely, for any $U_\infty\in \mathsf U_{\rm acc}$, the range of $U_\infty M_\infty U_\infty^*$ is a dark subspace.
	\end{proposition}
	\begin{proof} This proof mirrors the second part of the proof of \cite[Proposition~2.2]{BenFra}. For the reader's convenience, we repeat the key steps. Remark that since $\pp^\rho\ll\ppch$, it is sufficient to prove the result $\ppch$-almost surely.
		
		Let $q \in \nnone$. As proved in \cite[Proposition~2.2]{BenFra}, $(M_n)_{n \in \nnzero}$ is a martingale with respect to $\mathbb P^{\rm ch}$. Since it is bounded, it converges $\ppch$-almost surely and in $L^2(\ppch)$. Then, a lengthy direct computation leads to\footnote{For the intermediate steps leading to this result, see  \cite[Eq. (17)]{BenFra}.}
		\begin{equation*}
			\lim_{n\to\infty} \mathbb E^{\rm ch} \big( (M_{n+q}-M_n)^2\:  |\: \mathcal O_n\big) = 0. 
		\end{equation*}
		Using Jensen's inequality, we obtain 
		\begin{equation}\label{limitEch}
			\lim_{n\to\infty} \mathbb E^{\rm ch} \big( \|M_{n+q}-M_n\|\:  |\: \mathcal O_n\big) = 0. 
		\end{equation}
		For any $n \in \nnone$, using the polar decomposition of $W_n$, $M_{n+q}$ becomes
		\[M_{n+q} = \frac{M_n^{\frac 12} U_n^* V_{n+1}^*\ldots V_{n+q}^* V_{n+q}  \ldots V_{n+1} U_n M_n^{\frac 12} }{\tr(  V_{n+1}^*\ldots V_{n+q}^* V_{n+q}  \ldots V_{n+1} U_n M_n  U_n^* )},\]
		and so 
		\begin{align}\nonumber
			\mathbb E^{\rm ch} \big(  \|&M_{n+q} - M_n\|    \:  | \:  \mathcal O_n\big) = \\[0.33em] \label{Ech} & \int\limits_{\Ld^{\times q}}	   	\big\|  M_n^{\frac 12} U_n^*  v^*_{1} \cdots v^*_q v_q \cdots v_1  U_n M_n^{\frac 12}   -   M_n  \operatorname{tr}(  v^*_{1} \cdots v^*_q v_q \cdots v_1  U_n M_n U_n^*) \big\|  \d\mu^{\otimes q}(v_1, \ldots, v_q). 
		\end{align}
		Now, fix a realization $\omega$ and let $(U_{n_k}(\omega))_{k}$ be a subsequence of $(U_{n}(\omega))_{n}$ that is convergent to some unitary operator  $U_\infty(\omega)$ as $k$ tends to infinity. Dropping the notation $\omega$, taking the limit of the right hand side of \Cref{Ech} along $(n_k)_{k}$, and using \Cref{limitEch}, we have 
		\begin{align*}\label{Ech0}
			\int\limits_{\Ld^{\times q}}   	\big\|  M_\infty^{\frac 12} U_\infty^*  v^*_{1} \cdots v^*_q v_q \cdots v_1  U_\infty M_\infty^{\frac 12}   -    M_\infty   \operatorname{tr}(v^*_{1} \cdots v^*_q v_q \cdots v_1  U_\infty M_\infty U_\infty^*)\big \|  \d\mu^{\otimes q}(v_1, \ldots, v_q)  = \!0 
		\end{align*}    
		for any $U_\infty\in \mathsf U_{\rm acc}$. 
		Since the expression under the integral is non-negative,
		\[U_\infty M_\infty^{\frac 12} U_\infty^* v^*_{1} \cdots v^*_q v_q \cdots v_1    U_\infty M_\infty^{\frac 12} U_\infty^* \propto  U_\infty M_\infty U_\infty^*\]
		for $\mu^{\otimes q}$-almost all $(v_1, \ldots, v_q)$. Thus, denoting by $\pi_\infty$ the orthogonal projection on the range of $U_\infty M_\infty U_\infty^*$  (which depends on $\omega$), 
		\[\pi_\infty  v^*_{1} \cdots v^*_q v_q \cdots v_1   \pi_\infty  \propto   \pi_\infty\]
		for $\mu^{\otimes q}$-almost all $(v_1, \ldots, v_q)$. This implies that $\pi_\infty$ is a  dark projection, as claimed.
	\end{proof}
	
	The following result implies that the dark subspaces of the form   $\operatorname{range}(U_\infty  M_\infty  U_\infty^*)$ considered in the previous proposition are (almost surely) maximal. This proposition is the same as \cite[Theorem~2.8]{kusuoka1989dirichlet}; however, our proof is different and we do not assume the matrices are invertible. Recall that $r_m$ is the dimension of maximal dark subspaces.
	
	\begin{proposition}\label{prop:rankM}
		For all $\rho\in\Sd$, the rank of $M_\infty$ is $\mathbb P^{\rho}$-almost surely equal to $r_m$.
	\end{proposition}
	\begin{proof} Again, using $\pp^\rho\ll\ppch$ we only need to prove the result $\ppch$-almost surely. Let $U_\infty$ be an accumulation point of $(U_n)_{n}$. Since $\rank M_\infty = \rank U_\infty M_\infty U_\infty^*$, \Cref{prop:Minfdark} implies
		\[\rank M_\infty \leq   r_m.\]
		
		For the inverse inequality, let $\pi$ be the orthogonal projector on some $D \in \mathcal D_m$. For every  $n \in \nnone$ and  for $\mutensorn$-almost every $(v_1, \ldots, v_n)$ there exists $\lambda \geq 0$ such that
		\[\pi v^*_{1} \cdots v^*_n v_n \cdots v_1 \pi = \lambda\pi,\]  so taking the trace on both sides of this equality yields
		\[\lambda = \tfrac{1}{r_m}\tr(\pi v^*_{1} \cdots v^*_n v_n \cdots v_1).\]
		In what follows, we adopt the usual  convention that   $v_n \cdots v_1$ for $n = 0$ is $\id$.  For all $n,q\in\nnzero$
		\begin{equation*}
			\pi v^*_{1} \cdots v^*_n   W_q^*W_q   v_n \cdots v_1 \pi= \tfrac{1}{r_m}\tr(\pi v^*_{1} \cdots v^*_n W_q^*W_q v_n \cdots v_1)\pi
		\end{equation*}
		from which, dividing both sides by $\operatorname{tr}(W_q^*W_q)$ leads to
		\begin{equation*}
			\pi v^*_{1} \cdots v^*_n M_q v_n \cdots v_1  \pi=\tfrac{1}{r_m}{\tr(\pi v^*_{1} \cdots v^*_n M_q v_n \cdots v_1)}\pi
		\end{equation*}
		$\mathbb P^{\rm ch}$-almost surely. Consequently, taking the limit as $q$ goes to infinity, for every $n \in \nnzero$ we obtain
		\begin{equation}\label{eq:piDMpiD}
			\pi v^*_{1} \cdots v^*_n M_\infty v_n \cdots v_1 \pi=\tfrac{1}{r_m}{\tr(\pi  v^*_{1} \cdots v^*_n M_\infty v_n \cdots v_1)}\pi
		\end{equation}
		$\mathbb P^{\rm ch}$-almost surely  and, if $n > 0$, for  $\mutensorn$-almost all $(v_1, \ldots, v_n)$. In fact, by continuity  (see \Cref{prop:darkdefprop}(ii)), 
		\Cref{eq:piDMpiD} holds for every   $(v_1, \ldots, v_n) \in \supp \mutensorn$.
		Assume for a moment that for $\ppch$-almost all $M_\infty$ there exists $(v_1,\dotsc,v_n)\in \supp\mu^{\otimes n}$ such that $\tr(\pi  v^*_{1} \cdots v^*_n M_\infty v_n \cdots v_1)\neq 0$. Then, \Cref{eq:piDMpiD} and the rank of product inequality yield
		\[\rank M_\infty\geq \rank \pi= r_m\]
		and the proposition is proved.
		
		It remains to prove that for $\ppch$-almost all $M_\infty$ there exists $(v_1,\dotsc,v_n)\in \supp\mu^{\otimes n}$ such that \[\tr(\pi  v^*_{1} \cdots v^*_n M_\infty v_n \cdots v_1)\neq 0.\]
		Let 
		\[E=\operatorname{linspan}\big\{v_n\dotsb v_1\pi x \:|\: x \in \cc^d,  (v_1,\dotsc,v_n)\in \supp\mu^{\otimes n}, n\in \nnzero\big\}.\]
		Then $vE\subset E$ for $\mu$-almost every $v$. Hence, $E=\cc^d$ as it is non-zero by definition. Assume \[\tr(\pi  v^*_{1} \cdots v^*_n M_\infty v_n \cdots v_1\pi)= 0\] for any $(v_1,\dotsc, v_n)\in \supp \mu^{\otimes n}$ with $n\in\nnzero$ arbitrary.  Then \Cref{eq:piDMpiD} implies all the elements of the generating family of $E$ are in $\ker M_\infty$. Hence, $E\subset \ker M_\infty$. Since $E=\cc^d$, that implies $M_\infty=0$. This contradicts $M_\infty\in \Sd$ and the proposition is proved.
	\end{proof}
	
	\begin{remark}
		For any quantum trajectory $(\rho_n)_n$, \[\rho_n=U_n\frac{\sqrt{M_n}\rho_0\sqrt{M_n}}{\tr(\rho_0 M_n)}U_n^*.\] Then all the accumulation points $\rho_\infty$ of $(\rho_n)_n$ can be written as \[\rho_\infty=U_\infty \frac{\sqrt{M_\infty}\rho_0\sqrt{M_\infty}}{\tr(\rho_0 M_\infty)}U_\infty^*\] with $U_\infty\in \mathsf U_{\rm acc}$. Hence, \Cref{prop:rankM,prop:Minfdark} imply that any accumulation point $\rho_\infty$ of $(\rho_n)_n$ is supported in a maximal dark subspace. Hence, asymptotically in $n$, $\rho_n$ is almost surely supported in a maximal dark subspace. We will come back to this property with more details in \Cref{sec:convergence}.
	\end{remark}
	
	\subsection{Markov chain $\O$-measurable estimator}
	
	Let $n \in \nnone$ and consider the following max-likelihood estimator
	\[\widehat{E}_n=\operatorname{argmax}_{E\in\mathcal G_{r_m,d}(\cc)}\tr(W_n\pi_{E} W_n^*).\]
	Note that, by definition, $\widehat{E}_n$ is an $r_m$-dimensional subspace spanned by the eigenvectors of $W_n^*W_n$ of respective eigenvalues $a_1^2(W_n),\dotsc,a_{r_m}^2(W_n)$, where $A\mapsto a_i(A)$ is the map giving the singular values of $A$ in decreasing order with multiplicity, so that $a_1(A)\geq  \dotsb \geq a_d(A)$. The choice of such a subspace can be made in an $(\O_n)_n$-adapted way.
	
	Then define 
	\[\widehat{D}_n = W_n \widehat{E}_n.\]
	We consider this $(\mathcal O_n)_n$-adapted process as an  estimator of $D_n$ according to the sampling $W_n$, given the max-likelihood estimator $\widehat{E}_n$ of the initial space $D_0$. Note however that $\widehat{D}_n$ is not necessarily a dark subspace. However, by definition of $r_m$, $\widehat{D}_n$ is an element of $\mathcal G_{r_m,d}(\cc)$. Indeed, by definition of $\widehat{E}_n$, $\dim \widehat{D}_n<r_m$ would imply $\rank W_n<r_m$, which contradicts the definition of $r_m$.  
	
	We now upper bound the expectation of $d_{\rm G}(D_n,\widehat{D}_n)$. This estimate relies on the expression of $\dist$ using the exterior algebra of $\cc^d$.
	
	Let  $q \in \nnone$. We denote by $x_1\wedge   \cdots \wedge x_q$   the exterior product of vectors $x_1,\ldots, x_q$ in $\mathbb C^d$, and   by $\bigwedge^q\mathbb C^d$ the vector space generated by the $q$-wedge products $x_1\wedge  \cdots \wedge x_q$. More precisely, $x_1\wedge \cdots \wedge x_q$ stands for the alternating multilinear form \[ (\cc^d)^q  \ni (y_1,\ldots, y_q)\mapsto \det(\langle x_i, y_j\rangle )_{i,j=1}^q.\] Then, $\{x_1\wedge \cdots \wedge x_q \: | \: x_1,\ldots, x_q \in \mathbb C^d\}$ is a generating family for  the space  $\bigwedge^q\cc^d$ of alternating multilinear forms on $\cc^d$. We define a hermitian inner product on $\bigwedge^q\cc^d$ as
	\[\langle x_1\wedge \cdots \wedge x_q,\, y_1\wedge \cdots \wedge y_q\rangle = \det\big(\langle x_i, y_j\rangle \big)_{i,j=1}^q, \] and denote the associated norm by $\|x_1\wedge \cdots \wedge x_q\|$.
	By linearity and from the properties of determinants, if ${x_1,\dotsc,x_q}$ is a linearly dependent family of vectors, then $x_1\wedge \dotsb\wedge x_q=0$. This implies the following standard proposition. For the reader's convenience, we provide a short proof.
	\begin{proposition}
		\label{prop:dist in wedge}
		For any linearly independent family of vectors ${y_1,\dotsc,y_q}$ and any $x\in \cc^d$,
		\[\dist(x,\operatorname{span}\{y_1,\dotsc,y_q\})=\frac{\|x\wedge y_1\wedge\dotsb\wedge y_q\|}{\| y_1\wedge\dotsb\wedge y_q\|}.\]
	\end{proposition}
	\begin{proof}
		Let $F=\operatorname{span}\{y_1,\dotsc,y_q\}$ and write $x= \pi_Fx+x^\perp$, where $\pi_F$ denotes the orthogonal projection onto $F$. Then  $\|x\wedge y_1\wedge\dotsb\wedge y_q\|=\|x^\perp\wedge y_1\wedge\dotsb\wedge y_q\|$ and $\dist(x,F)=\|x^\perp\|$. If \(x^\perp=0\), the claim is immediate. Otherwise, set $y_0=x^\perp/\|x^\perp\|$. Since $x^\perp$ is orthogonal to $F$, we have \(\langle y_0,y_j\rangle=\langle y_j,y_0\rangle=0\) for \(j \in \{1, \ldots, q\}\), while \(\langle y_0,y_0\rangle=1\).  Therefore, by the properties of determinants, $\|y_0\wedge y_1 \wedge \dotsb\wedge y_q\|=\|y_1\wedge\dotsb\wedge y_q\|$. Hence,
		\(\|x\wedge y_1\wedge\dotsb\wedge y_q\|=\dist(x,F)\|y_1\wedge\dotsb\wedge y_q\|\)
		and the proposition follows.
	\end{proof}

	The last notion we need is the action of a linear map on exterior products: for  $A\in {\mathcal L}(\mathbb C^d)$ we write $\bigwedge^q A$ for the   operator acting on the basis elements of $\bigwedge^q\cc^d$ as
	\(    \textstyle{\bigwedge^q} A (x_1\wedge \cdots \wedge x_q)=Ax_1\wedge \cdots  \wedge Ax_q
	\)
	and then extended linearly to the whole space.
	Since $\bigwedge^q (AB)=\bigwedge^q  A  \bigwedge^q B$,
	\begin{equation}\label{eq:wedge_submul}
		\|\textstyle{\bigwedge^q} (AB)\|\leq\|\textstyle{\bigwedge^q} A\|\|\textstyle{\bigwedge^q} B\|.
	\end{equation}
	Moreover, for any $q \leq d$,
	\begin{equation}\label{eq:wedge_singular val}
		\|\textstyle{\bigwedge^q} A\|=a_1(A)\cdots a_q(A),
	\end{equation}
	where $a_1(A), a_2(A), \ldots , a_q(A)$ are the upper singular values of $A$. See e.g. \cite[Lemma III.5.3]{BLC}.
	
	With that set, we introduce a sequence $(s(n))_n$ that plays the part of $(f(n))_n$ in \cite{BenFra}, to bound the distance between the real dark subspace $D_n$ and its estimate $\widehat{D}_n$ based on observations. For $r_m=1$ the sequences  $(s(n))_n$ and $(f(n))_n$ coincide, as expected. For any $n\in \nnone$, let
	\[s(n)=\int_{\Ld^{\times n}}\left\|\textstyle{\bigwedge^{r_m+1}}W_n\right\|^{\frac{2}{r_m+1}}\d\mu^{\otimes n}.\]
	First we show this sequence decreases exponentially fast to zero.
	\begin{lemma}\label{lem:sn}
		There exist positive constants $C$ and $\lambda<1$ such that 
		\(s(n)\leq C\lambda^n\) for every $n \in \nnzero$.
	\end{lemma}
	\begin{proof}
		First, by definition of $\ppch$, for any $n\in \nnone$,
		\[s(n)
		=  d\,\mathbb E^{\rm ch}\left[\left(\frac{\Vert \textstyle{\bigwedge^{r_m+1}}W_n\Vert}{\tr(W_n^*W_n^{\phantom{*}})^{\frac{r_m+1}{2}}}\right)^{\frac{2}{r_m+1}}\right].\]
		Then, \Cref{eq:wedge_singular val} implies
		\[\frac{\Vert \textstyle{\bigwedge^{r_m+1}}W_n\Vert}{\tr(W_n^*W_n^{\phantom{*}})^{\frac{r_m+1}{2}}}=
		\Bigg\Vert  \textstyle{\bigwedge^{r_m+1}} \bigg(\frac{W_n}{\sqrt{\tr(W_n^*W_n^{\phantom{*}})}}\bigg)\Bigg\Vert
		=
		\prod_{i=1}^{r_m+1}a_i(U_n\sqrt{M_n}).\]
		Then, since $|M_n|\leq 1$, we have $a_i(U_n\sqrt{M_n})\leq 1$ for any $i\in \{1,\dotsc,d\}$ and
		\[\prod_{i=1}^{r_m+1}a_i(U_n\sqrt{M_n})\leq 1.\]
		From \Cref{prop:rankM}, \[\lim_{n\to\infty} \prod_{i=1}^{r_m+1}a_i(U_n\sqrt{M_n})=0.\] Hence, Lebesgue's dominated convergence theorem implies that $\lim_{n\to\infty} s(n)=0$.
		
		Second, \Cref{eq:wedge_submul} implies that $(s(n))_n$ is submultiplicative. Thus, there exists $n_0\in \nnone$ such that $s(n_0)<1$ and $s(n)\leq s(l)s^{\lfloor n/n_0\rfloor}(n_0)$, where $l\in \{0,\dotsc,n_0-1\}$ is the remainder of the Euclidean division of $n$ by $n_0$. Then fixing $1>\lambda\geq s^{1/{n_0}}(n_0)$ non-null and $C=\max_{l\in \{0,\dotsc,n_0-1\}}s(l)\lambda^{-l-1}$ yields the lemma. 
	\end{proof}

	We now show that $\sqrt{s(n)}$ upper bounds the expectation of $d_{\rm G}(D_n,\widehat{D}_n)$.
	\begin{lemma}
		\label{lem:dn}
		For any Borel probability measure $\chi$ over $\D$,
		\[\ee_\chi(d_{\rm G}(D_n,\widehat{D}_n))\leq \sqrt{s(n)}.\]
	\end{lemma}
	\begin{proof}
		Let $(y_1,\dotsc,y_{r_m})$ be an orthonormal basis of $\widehat{E}_n$ such that $W_ny_i=a_i(W_n)U_ny_i$ for each $i\in \{1,\dotsc,r_m\}$. Then, for any $x\in \cc^d$, from \Cref{prop:dist in wedge}
		\[\dist(W_nx,\widehat{D}_n)=\frac{\|\textstyle{\bigwedge^{r_m+1}}W_n\ (x\wedge y_1\wedge\dotsb\wedge y_{r_m})\|}{\prod_{i=1}^{r_m}a_i(W_n)}\leq a_{r_m+1}(W_n)\dist(x,\widehat{E}_n).\]
		Then, for any $D\in \D$ such that $\|W_n\pi_D\|\neq 0$, using $\|W_nx\|^2=\frac1{r_m}\tr(W_n\pi_DW_n^*)$ for any $x\in D$ of norm $1$ and $\dist(x,\widehat{E}_n)\leq 1$,
		\[d_{\rm G}^2(W_nD,\widehat{D}_n)\leq \frac{a^2_{r_m+1}(W_n)}{\tfrac 1{r_m} \tr(W_n\pi_DW_n^*)}\leq \frac{\|\textstyle{\bigwedge^{r_m+1}}W_n\|^{\frac{2}{r_m+1}}}{\frac{1}{r_m}\tr(W_n\pi_DW_n^*)}.\]
		Taking the expectation with respect to $\chi$, Jensen's inequality applied to $t\mapsto t^2$ implies that 
		\[\ee_\chi(d_{\rm G}(D_n,\widehat{D}_n))^2\leq \ee_\chi(d_{\rm G}^2(D_n,\widehat{D}_n))\leq \ee_\chi\left(\frac{\|\textstyle{\bigwedge^{r_m+1}}W_n\|^{\frac{2}{r_m+1}}}{\frac{1}{r_m}\operatorname{tr}(W_n\pi_{D_0}W_n^*)}\right)=s(n)\]
		and the lemma is proved.
	\end{proof}

	In what follows, $\theta$ stands for the left-shift operator on $\Omega$, i.e., $\theta(v_1, v_2, \ldots) = (v_2, v_3, \ldots)$.
	\begin{proposition}\label{prop_bound1}
		There exist    positive constants $C$ and $\lambda<1$ such that for every Borel probability measure $\chi$ over $\D$ and every $n, l\in \nnzero$
		\[\mathbb{E}_\chi (d_{\rm G}(D_{n+l}, \widehat{D}_n\circ \theta^l))\leq C\lambda^n\]
	\end{proposition}
	\begin{proof}
		Markov property implies that 
		\[\mathbb{E}_\chi (d_{\rm G}(D_{n+l}, \widehat{D}_n\circ \theta^l))=\mathbb{E}_{\chi K^l} (d_{\rm G}(D_{n}, \widehat{D}_n))\]
		for any $n,l\in \nnzero$. Then, \Cref{lem:dn} gives
		\[\mathbb{E}_\chi (d_{\rm G}(D_{n+l}, \widehat{D}_n\circ \theta^l))\leq \sqrt{s(n)}.\]
		Hence, \Cref{lem:sn} yields the proposition.
	\end{proof}
	
	\subsection{Convergence for $\O$-measurable functions}\label{sec:conv O-meas}
	
	The integer $p$ appearing in \Cref{thm:invmeasDark} is the \emph{period} of the  channel $\phi$, see \cite[Definition~3.1]{BenFra}. For the reader's convenience, we  recall  its definition adapted to the current context, i.e., with $\phi$ assumed to be irreducible.
	
	\begin{definition}
		Let $(E_1, \ldots ,E_k)$ be an orthogonal partition of $\cc^d$, i.e., the subspaces $E_1, \ldots ,E_k$ are non-zero and mutually orthogonal and $E_1 \oplus \ldots \oplus E_k = \cc^d$.
		We say that $(E_1, \ldots ,E_k)$ is a \emph{$k$-cycle} of $\phi$ if $v E_j \subset E_{j+1}$ for $\mu$-almost all $v \in \Ld$ and for every $j \in \{1, \ldots, k\}$, with the convention $E_{k+1} = E_1$.
		The \emph{period} $p$ of $\phi$ is the largest $k \in \nnone$ such that there exists a $k$-cycle of $\phi$.
		
	\end{definition}
	The full space $\cc^d$ is a $1$-cycle for $\phi$ and $p\leq d$ since $\dim \cc^d=d$. Hence, $p \in \{1, \ldots, d\}$. From  the Perron--Frobenius theorem for Schwarz maps\footnote{A completely positive map $T\colon \Ld\to\Ld$ preserving the identity is necessarily Schwarz in the sense that for any $A\in\Ld$, $T(A^*A)\geq T(A)^*T(A)$. That follows from the positivity of $T\otimes \id_{\mathcal M_2}$ applied to $\begin{pmatrix}A^*A& A^*\\A &\id_{\cc^d}\end{pmatrix}$.} -- \cite[Theorem~4.4]{Evans1978} -- there exist two positive constants  $C$  and $\lambda <1$ such that for every  $\rho \in \Sd$ and   every $n \in \nn$,
	\[
	\Bigg\Vert
	\frac 1p \sum_{\ell=0}^{p-1}\phi^{pn+\ell}(\rho) -    {\rho_\infty}
	\Bigg\Vert \leq
	C  \lambda^n.
	\]
	This bound was used in \cite{BenFra} to show that $\mathbb P^\rho$ converges in total variation under the shift $\theta$
	towards $\mathbb P^{\rho_\infty}$, and that the rate of this convergence is exponential.
	\begin{proposition}[\cite{BenFra}, Proposition 3.4]\label{prop_bound2}
		There exist  positive constants $C$ and $\lambda <1$ such that for any $\mathcal O$-measurable essentially bounded function $f$, any $\rho  \in \Sd$, and any $n \in \nn$ 
		\[
		\left\vert
		\mathbb E^\rho \left( \frac 1p \sum_{\ell=0}^{p-1}f \circ \theta^{pn+\ell}\right) - \mathbb E ^{\rho_\infty}(f)
		\right\vert \leq
		C \Vert f \Vert_{\infty}\lambda^n, 
		\]
		where $\Vert f \Vert_{\infty}$ is the essential supnorm of $f$, and $\rho_\infty$ is the unique fixed point of $\phi\colon \Sd\to\Sd$. 
	\end{proposition}

	\subsection{Uniqueness of the invariant measure on Dark subspaces}
	We are now ready to prove the convergence towards a unique $K$-invariant measure $\chi_{\textrm{inv}}$ on $\mathcal D_m$. Recall that the Wasserstein distance of order $1$ between two Borel probability measures $\mu$, $\nu$ on $\mathcal D_m$ can be expressed  using Kantorovich--Rubinstein duality  as
	\[W_1(\mu, \nu) = \sup \big\{   \mathbb E_\mu (f) -     \mathbb E_\nu (f) \,|\, f \in {\mathrm{Lip}_1(\mathcal D_m)} \big\},\]
	where $\mathrm{Lip}_1(\mathcal D_m)$ denotes the set of all Lipschitz functions on $\D$ with Lipschitz constant equal to~$1$.
	
	Recall that the existence of the invariant measure has been proved in \Cref{sec:proof existence}. The following theorem proves the geometric convergence and directly implies uniqueness of the $K$-invariant measure. Therefore, it completes the proof of \Cref{thm:invmeasDark}.
	\begin{theorem}\label{thm_conv}
		There exist positive constants $C$ and $\lambda <1$ such that for any Borel probability measure $\chi$ over $\D$, any $K$-invariant Borel probability measure $\chi_{\rm inv}$ over $\D$, and $n \in \nnone$ 
		\[
		W_1\Bigg(\frac 1p \sum_{\ell=0}^{p-1}\chi K^{pn+\ell}, \chi_{\rm inv} \Bigg) \leq C\lambda^n.
		\]
	\end{theorem}
	\begin{proof} We follow the strategy of proof of \cite[Theorem~1.1]{BenFra}.
		We need to show that
		\[
		\sup\left\{
		\mathbb{E}_\chi\Bigg(\frac 1p \sum_{\ell=0}^{p-1}f(D_{pn+\ell}) \Bigg) - \mathbb{E}_{\chi_{\rm inv}}(f(D_0))
		\:\Big|\: f \in \mathrm{Lip}_1(\mathcal D_m) \right\}
		\]
		decays   exponentially as $n \to \infty$. The expression being maximized is invariant under adding an arbitrary constant to $f$. Hence, we can restrict the supremum to functions $f\in \mathrm{Lip}_1(\mathcal D_m)$ such that $f(D^*)=0$ for an arbitrary $D^*\in\D$. Then, since $\sup_{D,D'\in \D}d_{\rm G}(D,D')\leq1$, any $f\in\mathrm{Lip}_1(\mathcal D_m)$ such that $f(D^*)=0$ has its range contained in $[-1,1]$. Therefore, for the remainder of the proof we can restrict our attention to functions $f \in \mathrm{Lip}_1(\mathcal D_m)$ such that $\Vert f \Vert_{\infty} \leq 1$.

		Let $n\in \nnone$ and denote $u = \lfloor  n/2 \rfloor$,  $q = \lceil  n/2 \rceil$ such that $u+q = n$. From the $K$-invariance of $\chi_{{\rm inv}}$,
		\begin{equation}
			\label{eq:conv_ineq1}
			\mathbb{E}_{\chi_{\rm inv}}(f(D_{0})) = \mathbb{E}_{\chi_{\rm inv}}(f(D_{pn+\ell}))   \ \ \textrm{ for every }   \ell \in \{0, \ldots, p-1\}.
		\end{equation} 
		Using this invariance and the triangle inequality,
		\begin{align}\label{eq:split wasser}
			\begin{split}
				\Bigg| \mathbb{E}_\chi\bigg(\frac 1p \sum_{\ell=0}^{p-1}f(D_{pn+\ell}) \bigg) - \mathbb{E}_{\chi_{\rm inv}}  (f(D_0))
				\Bigg|= \ &
				\Bigg| \frac 1p \sum_{\ell=0}^{p-1}\left(\mathbb{E}_\chi\big(f(D_{pn+\ell}) \big) - \mathbb{E}_{\chi_{\rm inv}} (f(D_{pn+\ell}))\right) \Bigg|
				\\   \leq\ & 
				\frac 1p \sum_{\ell=0}^{p-1}| \mathbb E_{\chi}(f(D_{p(u+q)+\ell}))-\mathbb E_{\chi}(f(\widehat{D}_{pu}\circ \theta^{pq+\ell})) |\\
				& +	\frac 1p \sum_{\ell=0}^{p-1}| \mathbb E_{\chi_{\rm inv}}(f(\widehat{D}_{pu}\circ \theta^{pq+\ell}))-\mathbb E_{\chi_{\rm inv}}(f(D_{p(u+q)+\ell})) | \\
				& + \Bigg| \frac1p \sum_{\ell=0}^{p-1}\left(\mathbb E_{\chi}(f(\widehat{D}_{pu}\circ \theta^{pq+\ell})) - \mathbb{E}_{\chi_{\rm inv}}(f(\widehat{D}_{pu}\circ \theta^{pq+\ell}))\right) \Bigg|.
			\end{split}
		\end{align}
		We first bound the first two terms of the sum on the right hand side. Since $f$ is $1$-Lipschitz, for every $\ell \in \{0, \ldots, p-1\}$,
		\[|f(D_{p(u+q)+\ell}) - f(\widehat{D}_{pu}\circ \theta^{pq+\ell})| \leq d_{\rm G}(D_{p(u+q)+\ell}\,,\, \widehat{D}_{pu}\circ \theta^{pq+\ell}),\]
		so \Cref{prop_bound1} implies there exist $C>0$ and $\lambda\in [0,1)$ independent of $\chi$ and $\chi_{\rm inv}$ such that
		\[| \mathbb E_{\chi}(f(D_{p(u+q)+\ell}))-\mathbb E_{\chi}(f(\widehat{D}_{pu}\circ \theta^{pq+\ell})) | \leq C\lambda^{u}\]
		and
		\[| \mathbb E_{\chi_{\rm inv}}(f(D_{p(u+q)+\ell}))-\mathbb E_{\chi_{\rm inv}}(f(\widehat{D}_{pu}\circ \theta^{pq+\ell})) | \leq C\lambda^{u}.\]
		The third term in the sum on the right-hand side of \Cref{eq:split wasser} is bounded using \Cref{prop_bound2}. Since $\widehat{D}_{pu}\circ\theta^{pq+\ell}$ is $\O$-measurable,
		\[ \mathbb E_{\chi}(f(\widehat{D}_{pu}\circ \theta^{pq+\ell}))=\mathbb E^{\rho_\chi}(f(\widehat{D}_{pu}\circ \theta^{pq+\ell}))\]
		Then, \Cref{prop_bound2} and $\|f\|_\infty\leq 1$ imply there exists $C>0$ and $\lambda\in [0,1)$ independent of $\chi$ and $\chi_{\rm inv}$ such that
		\begin{align*}
			\Bigg| \frac1p \sum_{\ell=0}^{p-1}\left(\mathbb E_{\chi}(f(\widehat{D}_{pu}\circ \theta^{pq+\ell})) - \mathbb{E}^{\rho_\infty}(f(\widehat{D}_{pu}))\right) \Bigg| \leq C\lambda^{q}.
		\end{align*}
		Following the same argument, the same bound holds with $\chi$ replaced by $\chi_{\rm inv}$. Hence, the triangle inequality leads to
		\[\Bigg| \frac1p \sum_{\ell=0}^{p-1}\left(\mathbb E_{\chi}(f(\widehat{D}_{pu}\circ \theta^{pq+\ell})) - \mathbb{E}_{\chi_{\rm inv}}(f(\widehat{D}_{pu}\circ \theta^{pq+\ell}))\right) \Bigg|\leq 2 C\lambda^q.\]
		
		Combining the bounds on the three terms in the sum on the right hand side of \Cref{eq:split wasser}, there exist $C>0$ and $\lambda\in [0,1)$ independent of $\chi$ and $\chi_{\rm inv}$ such that
		\[\left|
		\mathbb{E}_\chi\left(\frac 1p \sum_{\ell=0}^{p-1}f(D_{pn+\ell}) \right) - \mathbb{E}_{\chi_{\rm inv}}(f(D_0))
		\right| \leq  C\lambda^{n}\]
		and the theorem is proved.
	\end{proof}
	
	\section{Exponential convergence towards dark subspaces}\label{sec:convergence} 
	In this section we show that quantum trajectories converge exponentially fast towards dark subspaces. Note that for $\rho\in\Sd$ and $D\in\mathcal D_m$, we have $\tr(\rho\pi_D)=1$ if and only if $\supp{\rho}\subset D$. Then, the exponential convergence towards dark subspaces is expressed in the next theorem.
	
	\begin{theorem}
		\label{thm:conv to Dark}
		There exist two positive constants $C$ and $\lambda<1$ such that for any Borel probability measure $\nu$ over $\Sd$ and every $n \in \nn$
		\[\ee_\nu\big(\inf_{D\in \D}(1-\operatorname{tr}(\pi_D\rho_n))\big)\leq C\lambda^n.\]
	\end{theorem}
	
	To prove this result, we need a few lemmas. 
	First, using \Cref{prop_bound1}, we show that $D$ can be replaced by $(\widehat{D}_n)_n$. 
	
	\begin{lemma}\label{lem:exp conv estim to Dark}
		There exist two positive constants $C$ and $\lambda<1$ such that for any Borel probability measure $\nu$ over $\Sd$ and any $n\in \nnzero$,
		\[\ee_\nu(\inf_{D\in \D} d_{\rm G}(D,\widehat{D}_n))\leq C \lambda^n.\]
	\end{lemma}
	\begin{proof}
		Since $\inf_{D\in \D} d_{\rm G}(D,\widehat{D}_n)$ is $\O$-measurable, 
		\[\ee_\nu(\inf_{D\in \D} d_{\rm G}(D,\widehat{D}_n))=\ee^{\rho_\nu}(\inf_{D\in \D} d_{\rm G}(D,\widehat{D}_n)).\]
		Let $\chi_{\rm inv}$ be the unique $K$-invariant Borel probability measure over $\D$ from \Cref{thm:invmeasDark}. Then \[\rho_{\chi_{\rm inv}}=\ee_{\chi_{\rm inv}}(\pi_{D_0}/r_m)=\ee_{\chi_{\rm inv}}(\pi_{D_1}/r_m)=\phi(\rho_{\chi_{\rm inv}}).\] Hence $\rho_{\chi_{\rm inv}}=\rho_\infty$ with $\rho_\infty$ the unique fixed point of $\phi\colon\Sd\to\Sd$. Since we assume that $\mu$ is irreducible, $\rho_\infty$ is positive-definite and there exists $c>0$ such that  $\rho\leq c\rho_\infty$ for any $\rho\in \Sd$. 
		The linear extension of the map $\rho\mapsto \ee^\rho$ to $\Ld$ is positive, thus
		\[\ee^{\rho_\nu}(\inf_{D\in \D} d_{\rm G}(D,\widehat{D}_n))\leq c\ee^{\rho_\infty}(\inf_{D\in \D} d_{\rm G}(D,\widehat{D}_n))\]
		with $c$ independent of $\nu$.
		Then, choosing $D=D_n$, we obtain 
		\[\ee_\nu(\inf_{D\in \D} d_{\rm G}(D,\widehat{D}_n))\leq c\ee_{\chi_{\rm inv}}(d_{\rm G}(D_n,\widehat{D}_n))\] and \Cref{prop_bound1} yields the lemma.
	\end{proof}
	
	One can remark that $\inf_{D\in \D} d_{\rm G}(D,\widehat{D}_n)$ corresponds to the usual distance to the set of maximal dark subspaces $\mathcal D_m$, that is $d_{\rm G}(\mathcal D_m,\widehat{D}_n)$. Second, $\rho_n$ concentrates exponentially fast on $\widehat{D}_n$.

	\begin{lemma}\label{lem:convds} There exist two positive constants $C$ and $\lambda<1$ such that for any Borel probability measure $\nu$ over $\Sd$ and any $n\in \nnzero$,
		\[\ee_\nu\Big(1-\operatorname{tr}\!\big(\pi_{\widehat D_n}\rho_n\big)\Big)\leq C \lambda^n.\]
	\end{lemma}
	\begin{proof}
		Let $n \in \nn^*$. Since $\widehat{E}_n$ is the span of eigenvectors of $W_n^*W_n$ with non-zero corresponding eigenvalues, we have  $\widehat{D}_n=W_n\widehat{E}_n=U_n\widehat{E}_n$ with $U_n$ the unitary in the polar decomposition of $W_n$. By trace cyclicity, it follows that 
		\begin{align*}
			\operatorname{\tr}\!\big(\pi_{\widehat D_n}\rho_n\big) \tr(W_n^{\phantom{*}}\rho_0W_n^*) &   
			= 
			\tr(\pi_{\widehat D_n}W_n\rho_0 W_n^*)\\
			& =  \tr(U_n\pi_{\widehat{E}_n}U_n^*U_n\sqrt{W_n^*W_n^{\phantom{*}}}\rho_0 \sqrt{W_n^*W_n^{\phantom{*}}}U_n^*)\\
			& =  \tr(\pi_{\widehat{E}_n}\sqrt{W_n^*W_n^{\phantom{*}}}\rho_0 \sqrt{W_n^*W_n^{\phantom{*}}}).
		\end{align*}
		Then, since $\pi_{\widehat{E}_n}W_n^*W_n=W_n^*W_n\pi_{\widehat{E}_n}$,
		\begin{equation*}
			\Big(1-\operatorname{\tr}\!\big(\pi_{\widehat D_n}\rho_n\big)\Big)\tr(W_n^{\phantom{*}}\rho_0W_n^*) = \operatorname{tr}\! \Big( \big(\id_{\cc^d} - \pi_{\widehat{E}_n}\big)W_n^*W_n^{\phantom{*}}\rho_0\Big).
		\end{equation*}
		
		Since $\id_{\cc^d}-\pi_{\widehat{E}_n}$ is the orthogonal projector on a subspace spanned by the eigenvectors of $W_n^*W_n$ corresponding to eigenvalues $a^2_{r_m+1}(W_n), \dotsc, a^2_d(W_n)$, we have  
		\[ \Big\|\big(\id_{\cc^d} - \pi_{\widehat{E}_n}\big) W_n^*W_n^{\phantom{*}}\Big\|=a^2_{r_m+1}(W_n)\] 
		and 
		\[\Big(1-\operatorname{\tr}\!\big(\pi_{\widehat D_n}\rho_n\big)\Big)\tr(W_n^{\phantom{*}}\rho_0W_n^*)\leq a^2_{r_m+1}(W_n).\]
		Hence, using the ordering of singular values and \Cref{eq:wedge_singular val},
		\[ \Big(1-\operatorname{\tr}\!\big(\pi_{\widehat D_n}\rho_n\big)\Big)\tr(W_n^{\phantom{*}}\rho_0W_n^*)\leq \left\|\textstyle{\bigwedge^{r_m+1}}W_n\right\|^{\frac2{r_m+1}}.\]
		Then, by definition of $s(n)$ and $\pp_\nu$,
		\[\ee_\nu \Big(1-\operatorname{\tr}\!\big(\pi_{\widehat D_n}\rho_n\big)\Big)\leq s(n)\]
		and \Cref{lem:sn} concludes the proof.
	\end{proof}
	
	We can finally prove the theorem.
	\begin{proof}[Proof of \Cref{thm:conv to Dark}]
		For any $D\in \D$,
		\[\tr(\pi_D\rho_n)=\operatorname{tr}\!\big(\pi_{\widehat{D}_n}\rho_n\big)+\operatorname{tr}\!\big(\big(\pi_D-\pi_{\widehat{D}_n}\big)\rho_n\big).\]
		Hence, matrix H\"{o}lder's inequality in the form $|\tr(A^*B)| \leq \|A\| \|B\|_{\rm tr }$  for $A,B \in \Ld$ gives  
		\[\inf_{D\in \D} \big(1-\tr(\pi_D\rho_n)\big)\leq 1-\operatorname{tr}\!\big(\pi_{\widehat{D}_n}\rho_n\big)+\inf_{D\in \D} d_{\rm G} (D,\widehat{D}_n).\]
		Then, \Cref{lem:exp conv estim to Dark,lem:convds} yield the theorem.
	\end{proof}

	\section{Isometries between dark subspaces and the reference space}\label{sec:proc}
	
	In this section we focus on isometries between the maximal dark subspaces and the reference space $\cc^{r_m}$. They are key to classifying the invariant measures for quantum trajectories, i.e., the $\Pi$-invariant    measures on $\pcd$. To motivate their introduction, let us first consider a straightforward $\Pi$-invariant measure $\nu_{\rm unif}$ coming from the $K$-invariant measure $\chi_{\rm inv}$ on $\mathcal D_m$.  Namely, $\nu_{\rm unif}$ is constructed as the combination of uniform measures on the (projective) dark subspaces with weights given by $\chi_{{\rm inv}}$,  i.e., 
	\[\nu_{\rm unif}= \int_{\mathcal D_m} \mathrm{Unif}_{\P(D)}\,  \mathrm{d}\chi_{\rm inv}(D),\]
	where  $\P(D) = \{\hat x \:|\: x \in D, x \neq 0\}$  for $D \in \D$.
	We will verify its $\Pi$-invariance by direct computation. Recall that on pure states the action of $\Pi$ reads
	\[\Pi f(\hat y) =
	\int_{\Ld} f(v \cdot \hat y) \|v y\|^2 \d\mu(v) \ \ \text{ for } \hat  y \in \pcd.\]
	\begin{proposition}
		$\nu_{\rm unif}$ is $\Pi$-invariant.
	\end{proposition}
	\begin{proof}
		For every $D \in \supchi$ and $v \in \supp\mu$ the map $\P(D) \ni \hat y \mapsto v\cdot \hat y \in \P(vD)$ is bijective, provided that $\tr(v\pi_Dv^*)>0$, so
		\[  \int_{\pcd}
		f(v \cdot \hat y )  \, \d\mathrm{Unif}_{\P(D)}(\hat y)\
		=   \int_{\pcd}
		f(\hat y )\, \d\mathrm{Unif}_{\P(v D)}(\hat y).\]
		Thus, since $\|v y\|^2 = \operatorname{tr}\!\big(v \tfrac{\pi_{ D}}{r_m} v^*\big)$ for every $\hat y \in \P(D)$, it follows that
		\begin{align*}
			\int_{\pcd} \Pi f(\hat y )\, \d \nu_{\rm unif}(\hat y) & 
			=
			\int_{\Ld} \int_{\mathcal D_m}  \int_{\pcd}
			f(\hat y ) \,\d\mathrm{Unif}_{\P(v D)}(\hat y)\,  \tr(v \tfrac{\pi_{ D}}{r_m} v^*) \d\chi_{\rm inv}(D)\d\mu(v)
			\\ & =
			\int_{\mathcal D_m}  \int_{\pcd}
			f(\hat y )\, \d\mathrm{Unif}_{\P(D)}(\hat y)\,   \d(\chi_{\rm inv}K)(D)
			\\ & =
			\int_{\pcd} f(\hat y )\, \d \nu_{\rm unif}(\hat y),
		\end{align*}
		as desired.
	\end{proof}
	
	A natural question now is whether $\nu_{\rm unif}$ is the unique $\Pi$-invariant measure. In \cite{BenFra} it was shown that it is not always the case. For example, in the case $\supp \mu \subset \mathcal {SU}(d)$ of ``random
	unitary evolution'', when the whole space $\cc^d$ is dark, many $\Pi$-invariant
	measures may exist. 
	\begin{proposition}[\cite{BenFra}, Appendix~C]\label{prop:SU(k)}
		Assume $\supp\mu\subset \mathcal{SU}(d)$ and let $G$ be the smallest closed subgroup of $\mathcal{SU}(d)$ such that $\supp\mu\subset G$. 
		For any $\hat x\in\pcd$, put $[\hat x]_G$ for the orbit of $\hat x$ with respect to $G$ and the action $G\times \pcd\ni (u,\hat x)\mapsto u\cdot \hat x$, that is, \[[\hat x]_G:=\big\{\hat y\in\pcd \ |\ \exists u\in G \mbox{ s.t. } \hat y=u\cdot \hat x\big\}.\] Then, for any $\hat x \in\pcd$, there exists a unique $\Pi$-invariant probability measure $m_{\hat x}$ supported on $[\hat x]_G$, and this measure is uniform in the sense that it is invariant under the map $\hat x\mapsto u\cdot \hat x$ for any $u\in G$. {More precisely, $m_{\hat x}$ is the image measure of the (normalized) Haar measure on $G$ via the map $g_{\hat x}\colon G \to [\hat x]_G$ acting as $g_{\hat x}(u) = u\cdot \hat x$.}
		
		Note that, in particular, if $G=\mathcal{SU}(d)$, then $[\hat x]_G = \pcd$ for every $\hat x$, so $\Pi$ has a unique invariant probability measure and this measure is the uniform one on $\pcd$.
	\end{proposition}
	
	Let us point out that, within the setting of this proposition, if $G$ is a proper subgroup of $\mathcal{SU}(d)$,  there may exist mutually singular $\Pi$-invariant measures  on $\pc{d}$. Indeed, this is the case when there exist disjoint orbits under $G$, i.e., $[\hat x]_G\cap [\hat y]_G=\varnothing$ for some $\hat x$ and $\hat y$. This happens, e.g., if $G$ is finite, in which case  there exist infinitely many disjoint trajectories, thus also infinitely many  mutually singular $\Pi$-invariant measures. 
	If, however, there exists a single orbit,  
	i.e., $[\hat x]_G=\pcd$ for any $\hat x$ (in other words,  the action of $G$ on $\pcd$ is transitive), then we have a unique $\Pi$-invariant probability measure on $\pc{d}$.  
	A complete discussion and characterization is provided in \Cref{subs:ergo}.  
	
	A natural approach to the case of multiple maximal dark subspaces is to establish a correspondence between each $D \in \mathcal D_m$ and a common reference space $\cc^{r_m}$ by fixing a  family of maps  $ \J = \{J_D\}_{ D \in \mathcal D_m}$, where $J_D \colon \cc^{r_m} \to D$ is a linear isometry.
	This allows us to turn  (random) sequences of isometries acting between maximal dark subspaces into (random) sequences of  unitary operators on  $\cc^{r_m}$, bringing us back to the setting of \Cref{prop:SU(k)}. 
	
	Mirroring  for $\hat x \in \pcr$ the  construction from \Cref{prop:SU(k)},  we will obtain a family of measures $m_{\hat x, \J}$  on $\pcr$  invariant under the unitary dynamics from $G_\J$. Out of these we can then construct a family of $\Pi$-invariant measures on $\pcd$ by taking  the product measures
	$\chi_{\rm inv} \otimes    m_{\hat x, \J}$
	on $ \mathcal D_m \times \pcr$  and pushing them forward to $\pcd$ via the map  \[\Psi_\J \colon    (D,\hat z) \mapsto J_D\cdot \hat z,\] see also Figure~\ref{fig:diagram2}. 
	
	Our ultimate aim, which we will reach in \Cref{sec:invmeasures}, is to show that the set of pushforward measures 
	\[(\Psi_\J)_\star(\chi_{\rm inv} \otimes    m_{\hat x, \J})\]	 over $\hat x \in \pcr$ 
	coincides with the set of all   $\Pi$-ergodic measures on $\pcd$, provided that the family of isometries $\J$ is chosen in an optimal way, i.e., it is a so-called minimal family. 
	In the following subsections  we first formalize  the approach sketched above  and    then discuss  what makes a family of isometries optimal for our construction.
	
	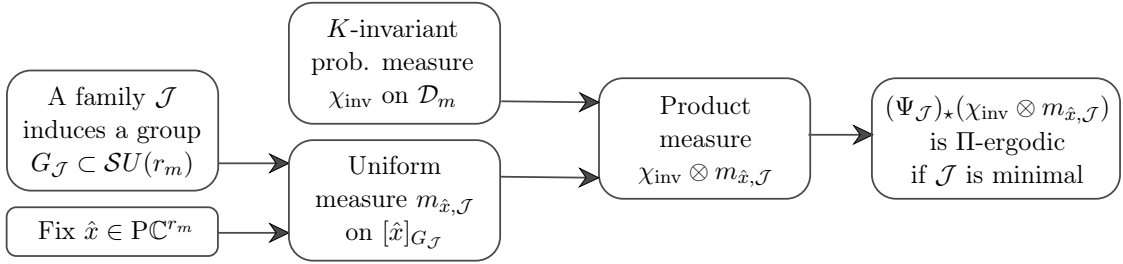
\begin{figure}[ h]
		\begin{center}
			\scalebox{0.9}{

				\tikzset{every picture/.style={line width=0.75pt}} 
				
				\begin{tikzpicture}[x=0.75pt,y=0.75pt,yscale=-1,xscale=1]
					
					\draw [color={rgb, 255:red, 74; green, 74; blue, 74 }  ,draw opacity=1 ]   (160.13,132.33) -- (181.62,132.33) -- (196.48,132.33) ;
					\draw [shift={(199.48,132.33)}, rotate = 180] [fill={rgb, 255:red, 74; green, 74; blue, 74 }  ,fill opacity=1 ][line width=0.08]  [draw opacity=0] (10.72,-5.15) -- (0,0) -- (10.72,5.15) -- (7.12,0) -- cycle    ;
					
					\draw [color={rgb, 255:red, 74; green, 74; blue, 74 }  ,draw opacity=1 ]   (159.22,171.32) -- (181.56,171.32) -- (196.48,171.32) ;
					\draw [shift={(199.48,171.32)}, rotate = 180] [fill={rgb, 255:red, 74; green, 74; blue, 74 }  ,fill opacity=1 ][line width=0.08]  [draw opacity=0] (10.72,-5.15) -- (0,0) -- (10.72,5.15) -- (7.12,0) -- cycle    ;
					
					\draw [color={rgb, 255:red, 74; green, 74; blue, 74 }  ,draw opacity=1 ]   (490.63,116.33) -- (513.88,116.33) -- (523.72,116.33) ;
					\draw [shift={(526.72,116.33)}, rotate = 180] [fill={rgb, 255:red, 74; green, 74; blue, 74 }  ,fill opacity=1 ][line width=0.08]  [draw opacity=0] (10.72,-5.15) -- (0,0) -- (10.72,5.15) -- (7.12,0) -- cycle    ;
					
					\draw  [color={rgb, 255:red, 74; green, 74; blue, 74 }  ,draw opacity=1 ] (199.48,55.84) .. controls (199.48,48.34) and (205.55,42.27) .. (213.05,42.27) -- (304.29,42.27) .. controls (311.79,42.27) and (317.87,48.34) .. (317.87,55.84) -- (317.87,96.57) .. controls (317.87,104.07) and (311.79,110.15) .. (304.29,110.15) -- (213.05,110.15) .. controls (205.55,110.15) and (199.48,104.07) .. (199.48,96.57) -- cycle ;
					
					\draw  [color={rgb, 255:red, 74; green, 74; blue, 74 }  ,draw opacity=1 ] (41.31,162.45) .. controls (41.31,159.43) and (43.76,156.98) .. (46.78,156.98) -- (154.24,156.98) .. controls (157.26,156.98) and (159.7,159.43) .. (159.7,162.45) -- (159.7,178.85) .. controls (159.7,181.87) and (157.26,184.31) .. (154.24,184.31) -- (46.78,184.31) .. controls (43.76,184.31) and (41.31,181.87) .. (41.31,178.85) -- cycle ;
					
					\draw  [color={rgb, 255:red, 74; green, 74; blue, 74 }  ,draw opacity=1 ] (41.48,93.74) .. controls (41.48,86.3) and (47.51,80.27) .. (54.95,80.27) -- (146.39,80.27) .. controls (153.83,80.27) and (159.87,86.3) .. (159.87,93.74) -- (159.87,134.17) .. controls (159.87,141.61) and (153.83,147.65) .. (146.39,147.65) -- (54.95,147.65) .. controls (47.51,147.65) and (41.48,141.61) .. (41.48,134.17) -- cycle ;
					
					\draw  [color={rgb, 255:red, 74; green, 74; blue, 74 }  ,draw opacity=1 ] (199.48,132.85) .. controls (199.48,125.39) and (205.53,119.33) .. (213,119.33) -- (304.35,119.33) .. controls (311.81,119.33) and (317.87,125.39) .. (317.87,132.85) -- (317.87,173.41) .. controls (317.87,180.88) and (311.81,186.93) .. (304.35,186.93) -- (213,186.93) .. controls (205.53,186.93) and (199.48,180.88) .. (199.48,173.41) -- cycle ;
					
					\draw  [color={rgb, 255:red, 74; green, 74; blue, 74 }  ,draw opacity=1 ] (373.24,98.14) .. controls (373.24,90.67) and (379.29,84.62) .. (386.76,84.62) -- (478.11,84.62) .. controls (485.58,84.62) and (491.63,90.67) .. (491.63,98.14) -- (491.63,138.7) .. controls (491.63,146.17) and (485.58,152.22) .. (478.11,152.22) -- (386.76,152.22) .. controls (379.29,152.22) and (373.24,146.17) .. (373.24,138.7) -- cycle ;
					
					\draw  [color={rgb, 255:red, 74; green, 74; blue, 74 }  ,draw opacity=1 ] (526.19,98.14) .. controls (526.19,90.67) and (532.24,84.62) .. (539.71,84.62) -- (651.72,84.62) .. controls (659.18,84.62) and (665.24,90.67) .. (665.24,98.14) -- (665.24,138.7) .. controls (665.24,146.17) and (659.18,152.22) .. (651.72,152.22) -- (539.71,152.22) .. controls (532.24,152.22) and (526.19,146.17) .. (526.19,138.7) -- cycle ;
					
					\draw [color={rgb, 255:red, 74; green, 74; blue, 74 }  ,draw opacity=1 ]   (318.37,139.85) -- (370.1,140.42) ;
					\draw [shift={(373.1,140.45)}, rotate = 180.62] [fill={rgb, 255:red, 74; green, 74; blue, 74 }  ,fill opacity=1 ][line width=0.08]  [draw opacity=0] (10.72,-5.15) -- (0,0) -- (10.72,5.15) -- (7.12,0) -- cycle    ;
					
					\draw [color={rgb, 255:red, 74; green, 74; blue, 74 }  ,draw opacity=1 ]   (318.51,97.54) -- (370.24,98.11) ;
					\draw [shift={(373.24,98.14)}, rotate = 180.62] [fill={rgb, 255:red, 74; green, 74; blue, 74 }  ,fill opacity=1 ][line width=0.08]  [draw opacity=0] (10.72,-5.15) -- (0,0) -- (10.72,5.15) -- (7.12,0) -- cycle    ;
					
					\draw (43,88) node [anchor=north west][inner sep=0.75pt]  [font=\normalsize] [align=left] {\begin{minipage}[lt]{84.09pt}\setlength\topsep{0pt}
							\begin{center}
								A family  $\mathcal{J}$   \\   induces   a group   \\ 
								$ G_{\mathcal{J}}  \subset \mathcal{SU} ( r_{m})$
							\end{center}
							
					\end{minipage}}; 
					
					\draw (38,163) node [anchor=north west][inner sep=0.75pt]  [font=\normalsize] [align=left] {\begin{minipage}[lt]{94.77pt}\setlength\topsep{0pt}
							\begin{center}
								$\hat{x}  \in \mathrm{P}\mathbb{C}^{r_{m}}$
							\end{center}
					\end{minipage}};

					\draw (210,127) node [anchor=north west][inner sep=0.75pt]  [font=\normalsize] [align=left] {\begin{minipage}[lt]{71.04pt}\setlength\topsep{0pt}
							\begin{center}
								Uniform \\[-0.1em] measure $m_{\hat{x} ,\mathcal{J}}$  \\[-0.1em] on  $[\hat{x}]_{G_{\J}}$ 
							\end{center}  
					\end{minipage}};

					\draw (202,50) node [anchor=north west][inner sep=0.75pt]  [font=\normalsize] [align=left] {\begin{minipage}[lt]{81.25pt}\setlength\topsep{0pt}
							\begin{center}
								$K$-invariant \\ prob. measure \\ $\chi_{\mathrm{inv}}$ on $\D$ 
							\end{center}
					\end{minipage}};

					\draw (385,93) node [anchor=north west][inner sep=0.75pt]  [font=\normalsize] [align=left] {\begin{minipage}[lt]{69.83pt}\setlength\topsep{0pt}
							\begin{center}
								Product \\ measure \\ 
								$\chi_{\mathrm{inv}} \otimes m_{\hat{x},\mathcal{J}}$ 
							\end{center} 
					\end{minipage}};

					\draw (531,93) node [anchor=north west][inner sep=0.75pt]  [font=\normalsize] [align=left] {\begin{minipage}[lt]{96.22pt}\setlength\topsep{0pt}
							\begin{center} 
								$(\Psi_{\J})_{\star}(\chi_{\mathrm{inv}} \otimes m_{\hat{x},\J})$
								\\[0.1em]
								is $\Pi$-ergodic 
								\\
								if $\mathcal{J}$ is minimal
							\end{center}
					\end{minipage}}; 
				\end{tikzpicture}
			}
		\end{center}
		\caption{Diagram summarizing the construction of $\Pi$-ergodic measures 
		}\label{fig:diagram2}
	\end{figure}

	\subsection{Minimal families of isometries}
	Fix $ \J = \{J_D\}_{ D \in \mathcal D_m}$, where $J_D \colon \cc^{r_m} \to D$ is a linear isometry. 
	On $(\mathcal D_m \times \Omega,\,\mathcal Q \otimes\mathcal O,\,\mathbb P_{\chi_{\rm inv}})$ we consider the random variable
	\[U =   \frac{J^{-1}_{V_1D_0} V_1 J^{\phantom{{-1}}}_{D_{0}}}{\sqrt{\tr(V_1 \frac{\pi_{D_{0}}}{r_m} V_1^*)}} \colon \mathbb C^{r_m}\rightarrow D_0\rightarrow V_1D_0\rightarrow\mathbb C^{r_m}.\]
	Since  \[ \mathbb P_{\chi_{{\rm inv}}}\Big(\operatorname{tr}\!\big(V_1 \pi_{D_0} V_1^*\big) = 0 \Big) = \mathbb P_{\chi_{{\rm inv}}}\big(V_1D_0  = \{0\} \big) =0,\] we see that $U$ is $\mathbb P_{\chi_{{\rm inv}}}$-almost surely well defined and takes values in the unitary group $\Urm$ by definition of dark subspaces. 
	Here and henceforth, if $\omega = (v,\ldots)$ and $\operatorname{tr}(v \pi_D v^*)>0$, we let $u_{v,D}$ stand for $\exp(\mathrm{i} \gamma) U(D, \omega)$, where $\gamma$ is so chosen that $\det u_{v,D} = 1$. That is, \[u_{v,D} \propto J_{v D}^{-1}v J_D^{\vphantom{-1}}\] is a special unitary operator induced on the reference space $\cc^{r_m}$ when the map $v$ acts on the system occupying the dark subspace $D$.
	By $G_\J$ we denote the smallest closed subgroup of $\mathcal {SU}(r_m)$ containing the set 
	\begin{equation}
		\label{eq:defS}
		S_\J =\big\{u_{v,D} \in {\mathcal {SU}(r_m)} \,|\, v \in \supp\mu,\: D \in \supp\chi_{{\rm inv}},\: \operatorname{tr}(v \pi_D v^*)>0\big\}.
	\end{equation}
	Let $\mathrm{H}_{G_\J}$ be
	the normalized Haar measure on $G_\J$. For  $\hat x \in \pcr$   consider the map
	\[\begin{array}{rcl} g_{\hat x}\colon \  G_\J&\longrightarrow&[\hat x]_{G_\J}  \subset \pcr \\[0.33em]
		u &\longmapsto& u\cdot \hat x     \end{array}\]
	and define $m_{\hat x,\J}$ as the push-forward measure of $\mathrm{H}_{G_\J}$ by $g_{\hat x}$, i.e., 
	\[m_{\hat x,\J}=({g_{\hat x}})_\star \mathrm{H}_{G_\J}.\]
	
	Now, sending a copy of $m_{\hat x,\J}$ to each maximal dark subspace with a weight prescribed by $\chi_{\rm inv}$, we expect to obtain a $\Pi$-invariant measure on $\pcd$. Indeed,  let
	\[\begin{array}{rcl}\Psi_\J\colon \   \mathcal D_m \times \P(\cc^{r_m}) & \longrightarrow & \pcd\\[0.33em]
		(D,\hat z) &\longmapsto& J_D\cdot \hat z  \end{array}\]
	and define 
	\begin{equation}\label{eq:proc}
		\nu_{\hat x,\J} = (\Psi_\J)_\star(\chi_{{\rm inv}} \otimes m_{\hat x,\J}).
	\end{equation}
	Explicitly,
	\begin{align*}
		\int_{ \pcd }f(\hat y)\d \nu_{\hat x, \J}  (\hat y) 
		=
		\int_{\mathcal D_m}\int_{G_\J} f(J_D u  \cdot \hat x) \d\mathrm{H}_{G_\J}  (u) \d\chi_{\rm inv}(D),
	\end{align*}
	where $f\in \mathcal C(\pcd)$.  As expected, we have the following

	\begin{proposition}\label{prop:productmeas}
		$\nu_{\hat x,\J}$ is $\Pi$-invariant.
	\end{proposition}
	\begin{proof} 
		Recall that for every  $u \in G_\J$ and $D \in \D$ we have
		\begin{equation}\label{darkprob}
			\|v J_D u x \|^2 = \tr(v \tfrac{\pi_D}{r_m}v^*).
		\end{equation} 
		Also, recall that the invariance of Haar measure guarantees that,  for every $\tilde u \in G_\J$ and every $h\in \mathcal C(\pcr)$,
		\begin{equation}\label{haarinvv}
			\int_{G_\J} h(\tilde u  u \cdot \hat x) \,\mathrm{dH}_{G_\J}(u)  = \int_{G_\J} h(   u \cdot \hat x) \,\mathrm{dH}_{G_\J}( u).
		\end{equation}
		Fix $f\in \mathcal C(\pcd)$.
		From \Cref{darkprob,haarinvv} and the  $K$-invariance of $\chi_{{\rm inv}}$,
		\begin{align*}
			\int_{ \pcd }\Pi f( \hat y) \,  \mathrm{d} \nu_{\hat x,\J} (\hat y)   
			& = \int_{\Ld}\int_{\mathcal D_m} \left[  \int_{G_\J} f( v J_{D}  u \cdot \hat x) \,\mathrm{dH}_{G_\J}(u) \right] \tr(v\tfrac{\pi_D}{r_m}v^*)\, \d\chi_{\rm inv}(D) \d\mu(v)
			\\ &  =
			\int_{\Ld} \int_{\mathcal D_m} \left[\int_{G_\J} (f\circ J_{vD}) (u_{v,D} u \cdot \hat x) \,\mathrm{dH}_{G_\J}(u)\right] \tr(v\tfrac{\pi_D}{r_m}v^*)\d\chi_{\rm inv}(D)\d\mu(v)
			\\   & =
			\int_{\Ld} \int_{\mathcal D_m} \left[\int_{G_\J} (f\circ J_{vD}) (u \cdot \hat x) \,\mathrm{dH}_{G_\J}(u)\right] \tr(v\tfrac{\pi_D}{r_m}v^*)\d\chi_{\rm inv}(D)\d\mu(v)
			\\  & =
			\int_{\mathcal D_m}\int_{G_\J} (f\circ J_{D}) (u \cdot \hat x) \, \mathrm{dH}_{G_\J}(u)  \mathrm{d}[\chi_{\rm inv}K](D)
			\\  & =   \int_{ \pcd }  f( \hat y) \, \mathrm{d} \nu_{\hat x,\J} (\hat y),
		\end{align*}
		as desired.
	\end{proof}

	\begin{remark}
		If $G_\J = \mathcal{SU}(r_m)$, then   $m_{\hat x, \J} = \operatorname{Unif}_{\pcr}$ and $\nu_{\hat x,\J} = \nu_{\rm unif}$ for all $\hat x \in \pcr$.
	\end{remark}

	Let us stress that all  the objects related to the reference space $\pcr$, in particular the group
	$G_\J$ and thus also the  family of $\Pi$-invariant measures
	\[\Upsilon_\J = \{\nu_{\hat x,\J} \colon \hat x\in \pcr\},\]
	depend on the fixed family of isometries $\mathcal J$. It is no surprise that the choice of $\J$ can heavily influence what $\Pi$-invariant measures we obtain in $\Upsilon_\J$. Namely, 
	for two families $\J$ and $\tJ$ it is possible that:
	\begin{itemize}
		\item    $G_{\mathcal J}$ and $G_{\widetilde {\mathcal J}}$   are different and neither of them is a subgroup of the other (even up to unitary equivalence), see   Example \hyperref[ex:16a]{1(6a)} in Section 7.   
		\item    $G_{\J}$ is uncountably infinite  and $G_{\tJ}$ is finite, see   Example \hyperref[ex:16b]{1(6b)}.
		\item  $G_{\J}$ is the full special unitary group (so $\Upsilon_\J = \{\nu_{\rm unif}\}$), while $G_{\tJ}$ is its proper subgroup (and $\Upsilon_{\tJ}$ is infinite), Example \hyperref[ex:15b]{1(5b)}. 
	\end{itemize}  
	As a consequence, there may exist a $\Pi$-invariant measure of the form $\nu_{\hat x,\tJ}$ which does not belong to the convex hull of $\Upsilon_\J$.  This means that the set $\Upsilon_\J$ does not contain all $\Pi$-ergodic measures.  
	Intuitively, this happens when the family $\J$ is poorly chosen in the sense that the group $G_\J$ is too big compared to the dynamics actually taking place in $\pcd$. As a result, we end up with too few measures of the form $m_{\hat x, \J}$ on the reference space $\pcr$, and so transforming them into $\nu_{\hat x, \J}$ does not allow us to recover all $\Pi$-invariant measures. We address this issue by introducing the notion of \emph{minimal families}: 
	
	\begin{definition}
		A family of linear isometries  $\mathcal J =\{J_D\}_{D\in \mathcal D_m}$  is called a \emph{minimal family}  if the group $G_\J$ is, up to a unitary equivalence, minimal as a subgroup of the special unitary group, i.e., for any other family $\tJ$  there exists $Q \in \Urm$ such that $G_\J \subset Q G_{\tJ}Q^{-1}$.
	\end{definition}
	
	We will see  that the minimality of $\J$ guarantees that $\Upsilon_\J$ is exactly the set of $\Pi$-ergodic measures.  
	More precisely, as mentioned already in \Cref{thm:intro-ergodic}, if $\J$ is a minimal family, then a probability measure $\nu$ on $\pcd$ is $\Pi$-ergodic if and only if $\nu   = \nu_{\hat x, \J}$ for some $\hat x \in \pcr$.  As a first step towards proving this claim, in the next subsection we  explore the properties of minimal families. In particular, we provide a characterization of minimal families in terms of the so-called smart families, which will play a key role in the results leading to the characterization of $\Pi$-invariant measures.

	\subsection{Properties of minimal families}\label{subs:minimality}
	
	First, we introduce and explore the notion of \emph{smart families}, which we will then prove to be equivalent to minimal families.
	To this end, we need to  introduce some notation. For $n \in \nnone$ we denote \[\V_n = \{v_n \cdots v_1 \colon v_1, \ldots, v_n \in \supp \mu\}\]  and put $\mathcal V_0 = \{\id_{\cc^d}\}$ for consistency. Then we define 
	\[\V_{\rm fin} = \bigcup_{n \in \nnzero}\V_n \quad \textrm{ and } \quad \mathcal W_{\rm fin} = \left\{v/\|v\| \colon v \in \mathcal V_{\rm fin}, v \neq 0 \right\},\quad \mathcal W = \overline{\mathcal W_{\rm fin}}\] and finally \[\mathcal W_{r_m} = \{w\in \mathcal W \colon \rank w=r_m\}.\]
	Note that, by construction, if $w_1, w_2 \in \mathcal W$ are such that $w_1w_2 \neq 0$, then $w_1w_2/\|w_1w_2\| \in \mathcal W$.
	
	As the construction of smart families relies on $\mathcal W_{r_m}$, the next lemma establishes some of its important properties.
	\begin{lemma}
		\label{lem:properties cal W}
		Recall that $\supp\mu$ is assumed irreducible. Then, $r_m=\min \{\rank w: w\in \mathcal W\}$ and $\mathcal W_{r_m}$ is non-empty and compact. Moreover, for any $(v,w)\in\mathcal V_{\rm fin}\times\mathcal W_{r_m}$, either $vw=0$ or $\frac{vw}{\|vw\|}\in\mathcal W_{r_m}$, and either $wv=0$ or $\frac{wv}{\|wv\|}\in\mathcal W_{r_m}$.
	\end{lemma}
	\begin{proof}
		Let $D$ be a dark subspace of dimension $r_m$. Assume for contradiction that there exists
		$w\in\mathcal W$ such that $\operatorname{rank}w<r_m$. First, we show that if
		$w\pi_D=0$, then  there exists $v\in\mathcal V_{\mathrm{fin}}$ such that
		$wv\pi_D\neq0$. Indeed, otherwise we have 
		\[
		E:=\operatorname{linspan}\{vx \,|\, v\in\mathcal V_{\mathrm{fin}},\ x\in D\}
		\subset\ker w.
		\]
		The space $E$ is non-zero and invariant under every
		$v\in\operatorname{supp}\mu$, hence $E=\mathbb C^d$ by irreducibility,
		which would imply $w=0$, a contradiction. If $w\pi_D\neq0$, we set
		$v=\id$. The 
		definition of dark subspaces implies
		\[
		\pi_Dv^*w^*wv\pi_D=\|wv\pi_D\|^2\pi_D.
		\]
		Since $wv\pi_D\neq0$, the right-hand side has rank $r_m$, whereas   the product rank inequality gives  
		$
		\operatorname{rank}(\pi_Dv^*w^*wv\pi_D)
		\leq \operatorname{rank}w  <r_m,
		$
		a contradiction. In consequence,  $r_m\leq\min\{\operatorname{rank}w \,|\, w\in\mathcal W\}$.
		
		Now, assume that 
		$\operatorname{rank}w>r_m$ for every $w\in\mathcal W$. 
		Since ${M_\infty}/{\|M_\infty\|}\in \{w^*w \,|\,  w\in \mathcal W\}$ by definition, that would imply $\rank M_\infty>r_m$, which contradicts \Cref{prop:rankM}. Therefore, $\min\{\rank w \,|\,  w\in \mathcal W\}=r_m$ and $\mathcal W_{r_m}$ is non-empty since, almost surely, $M_\infty/\|M_\infty\|=w^*w$ for some $w\in \mathcal W_{r_m}$.
		
		The fact that $vw=0$ or $vw/\|vw\|\in \mathcal W_{r_m}$ is directly implied by the rank product inequality and the fact that for any $(v,w)\in \mathcal V\times \mathcal W$, either $vw=0$ or $vw/\|vw\|\in \mathcal W$ by definition. The same argument leads to $wv=0$ or $wv/\|wv\|\in \mathcal W_{r_m}$.
		
		It remains to prove that $\mathcal W_{r_m}$ is compact. The set $L_{r_m}$ of matrices of rank at most $r_m$ is closed as the intersection of the zero sets of determinants of submatrices. Since $r_m=\min\{\rank w \,|\,  w\in \mathcal W\}$, we have $\mathcal W_{r_m}=\mathcal W\cap L_{r_m}$. Hence, $\mathcal W_{r_m}$ is closed and therefore compact as a subset of the matrix unit ball.
	\end{proof}
	
	We can now  define  smart families.
	\begin{definition}\label{def:Dcsmart}
		Let  $\J = \{J_D\}_{D\in \mathcal D_m}$ be a family of linear isometries such that $J_D\cc^{r_m}=D$ and let $D_c\in\mathcal D_m$ be fixed. We say that $\J$ is $D_c-$\emph{smart} if  for every $D  \in \mathcal \supp{\chi_{{\rm inv}}}$  there exists $w \in \mathcal W$ such that $J_{D} \propto  w J_{D_c}$.  
		We say that $\J$ is \emph{smart} if it is $D_c-$smart for all $D_c\in \supp \chi_{\rm inv}$.
	\end{definition}

	The idea behind $D_c-$smartness is the presence of the `central' dark subspace $D_c$ from which every other dark subspace in $\supp \chi_{{\rm inv}}$ can be reached with the help of some $w \in \mathcal W$ 
	without inducing any non-trivial transformation on the reference space $\cc^{r_m}$. This is because the property $J_{D} \propto w J^{\vphantom{-1}}_{D_c}$ is equivalent to $J_{D}^{{-1}} w J^{\vphantom{-1}}_{D_c} \propto \id$, and the following diagram commutes

	\tikzset{every picture/.style={line width=0.75pt}} 
	\begin{center}
		\scalebox{0.5}{
			\begin{tikzpicture}[x=0.75pt,y=0.75pt,yscale=-1,xscale=1]
				
				\draw    (295.37,43.89) -- (410.14,43.89) ;
				\draw [shift={(409.14,43.89)}, rotate = 180] [fill={rgb, 255:red, 0; green, 0; blue, 0 }  ][line width=0.08]  [draw opacity=0] (12,-3) -- (0,0) -- (12,3) -- cycle    ;
				
				\draw    (300.37,192.89) -- (410.14,192.89) ;
				\draw [shift={(413.14,192.89)}, rotate = 180] [fill={rgb, 255:red, 0; green, 0; blue, 0 }  ][line width=0.08]  [draw opacity=0] (12,-3) -- (0,0) -- (12,3) -- cycle    ;
				
				\draw    (252.85,164.64) -- (252.85,128.64) -- (252.85,70.22) ;
				\draw [shift={(252.85,68.22)}, rotate = 90] [fill={rgb, 255:red, 0; green, 0; blue, 0 }  ][line width=0.08]  [draw opacity=0] (12,-3) -- (0,0) -- (12,3) -- cycle    ;
				
				\draw    (444.85,165.74) -- (444.85,129.73) -- (444.85,71.32) ;
				\draw [shift={(444.85,69.32)}, rotate = 90] [fill={rgb, 255:red, 0; green, 0; blue, 0 }  ][line width=0.08]  [draw opacity=0] (12,-3) -- (0,0) -- (12,3) -- cycle    ;

				\draw (240,30) node [anchor=north west][inner sep=0.75pt]  [font=\huge]  {$D_{c}$};
				
				\draw (430,30) node [anchor=north west][inner sep=0.75pt]  [font=\huge]  {$D$};
				
				\draw (235,178) node [anchor=north west][inner sep=0.75pt]  [font=\huge]  {$\mathbb{C}^{r_{m}}$};
				
				\draw (425,178) node [anchor=north west][inner sep=0.75pt]  [font=\huge]  {$\mathbb{C}^{r_{m}}$};

				\draw (210,105) node [anchor=north west][inner sep=0.75pt]  [font=\LARGE]  {$J_{D_{c}}$};
				
				\draw (410,105) node [anchor=north west][inner sep=0.75pt]  [font=\LARGE]  {$J_{D}$};

				\draw (342,17) node [anchor=north west][inner sep=0.75pt]  [font=\LARGE]  {$w$};
				
				\draw (342,172) node [anchor=north west][inner sep=0.75pt]  [font=\LARGE]  {$\Id$};

			\end{tikzpicture}
		}
	\end{center}

	Our most immediate aim is to establish 
	the existence of a $D_c-$smart family for an arbitrary $D_c \in \D$ and then to prove that $D_c-$smartness implies smartness whenever $D_c\in \supp \chi_{\rm inv}$. These results will be crucial in making the link between smart and minimal families, which is the main result of this section. To achieve these goals, we need a series of auxiliary results. Let us start by collecting a list of useful basic facts on isometries and dark subspaces (the proof is not provided as these properties  follow directly from the definitions).
	\begin{proposition}Let $D \in \D$, $\hat y \in \P(D)$ and $w \in \W$. Then
		\begin{enumerate}[label={\rm (\roman*)}] 
			
			\item   $wy \neq 0$ iff  $\operatorname{tr}(w\pi_Dw^*) >0$ iff $wD \neq \{0\}$.   In particular, $w \cdot \hat y$ and $u_{w,D}$ are well defined if $\operatorname{tr}(w\pi_Dw^*) >0$.
			
			\item $J_{D} \propto  w J_{D_c} $ implies that $w D_c = D$; in particular, $w D_c \neq \{0\}$.
			
			\item Since all $J_D$'s are isometric, the absolute value of the proportionality constant in $J_{D} \propto  w J_{D_c} $ is   equal to 
			$\tr(w \frac{\pi_{D_c}}{r_m}w^*)^{-\frac 12}$.

		\end{enumerate}
	\end{proposition}
	
	Next, we establish an auxiliary result that concerns the support of $\chi_{{\rm inv}}$. It will allow us to prove the existence of a $D_c-$smart family for an arbitrary dark subspace $D_c$. 
	\begin{proposition}\label{prop:suppchi} For any $D  \in \mathcal D_m$ we have
		\[\operatorname{supp}\chi_{\mathrm{inv}}\subset\{wD\mid w\in\mathcal W_{r_m},\ wD\neq\{0\}\}.\]
	\end{proposition}
	\begin{proof}
		First, we prove that for any $w\in \mathcal W_{r_m}$ and $\hat x\in \pcd$,
		\begin{equation}\label{eq:non zero WD}
			\exists v\in \mathcal V_{\rm fin} \quad \text{such that}\quad w vx\neq 0.
		\end{equation}
		If $w vx=0$ for any $v\in \mathcal V_{\rm fin}$, $E=\operatorname{linspan}\{vx \,|\,  v\in \mathcal V_{\rm fin}\}$ is $\mathcal V_1$-invariant and a subset of $\ker w$. 
		Then $E\neq\{0\}$ since $\id \in \mathcal V_{\rm fin}$, and $\ker w \neq\cc^d$ since $\|w\|=1$. Hence, the irreducibility assumption on $\mu$ implies \Cref{eq:non zero WD}.
		
		The first consequence of \Cref{eq:non zero WD} is that $\mathcal W_{r_m}D\neq \{\{0\}\}$ since $w v/\|w  v\|\in \mathcal W_{r_m}$ as a consequence of rank inequality on products and \Cref{lem:properties cal W}.
		
		Second, we prove $\mathcal W_{r_m}D\setminus\{\{0\}\}$ is closed. Let $(D_n)_n$ be a sequence in $\mathcal W_{r_m}D\setminus\{\{0\}\}$ and let $D_\infty$ be its limit in $\mathcal D_m$. Let $D_n=w_nD$ with $w_n\in \mathcal W_{r_m}$. \Cref{lem:properties cal W} ensures $(w_n)_n$ is relatively compact in $\mathcal W_{r_m}$. Let $w_\infty$ be an accumulation point of $(w_n)_n$. Either $w_\infty D=D_\infty$ or $w_\infty D=\{0\}$. 
		
		In the first alternative, we immediately deduce $D_\infty\in \mathcal W_{r_m}D$. In the second alternative, \Cref{eq:non zero WD} implies there exists $v\in \mathcal V_{\rm fin}$ such that $w_\infty vD\neq \{0\}$. Hence, for $n$ large enough, by definition of dark subspaces and since $\rank w_n=r_m$, $D_n=\range w_n$ and  $w_nvD=D_n$. Using $w_\infty v D\neq \{0\}$, it follows that $w_\infty vD=D_\infty$. \Cref{lem:properties cal W} implies $w_\infty v/\|w_\infty v\|\in \mathcal W_{r_m}$. Thus, $D_\infty\in \mathcal W_{r_m}D$. We conclude that $\mathcal W_{r_m}D\setminus\{\{0\}\}$ is closed.
		
		Third, we prove that if $\mathcal W_{r_m}D$ is $K$-invariant, then $\supp\chi_{\mathrm{inv}}\subset\mathcal W_{r_m} D\setminus\{\{0\}\}$. Assume that $\mathcal W_{r_m} D$ is $K$-invariant, setting $K(\{0\},\cdot)=\delta_{\{0\}}$. Since for any $D'\in \mathcal D_m$, $vD'=\{0\}$ implies $\operatorname{tr}(v\pi_{D'}v^*)=0$,
		\[K(D',\{0\})=0.\] Hence, for any $D'\in \mathcal D_m$, $K(D',\mathcal W_{r_m}D)=K(D',\mathcal W_{r_m}D\setminus\{\{0\}\})$ and $\mathcal W_{r_m}D\setminus\{\{0\}\}$ is $K$-invariant. Since $\supp\chi_{\mathrm{inv}}$ is the smallest $K$-invariant closed set and $\mathcal W_{r_m}D\setminus\{\{0\}\}$ is closed as proved in the previous paragraph, it follows that $\supp\chi_{\mathrm{inv}}\subset\mathcal W_{r_m} D\setminus\{\{0\}\}$.
		
		Finally, it remains to prove that $\mathcal W_{r_m}D$ is $K$-invariant. For $D'\in \mathcal W_{r_m}D$, if $D'=\{0\}$, then  $K(\{0\},\{\{0\}\})=1$ and $K(\{0\},\mathcal W_{r_m}D)=1$. If $D'\neq \{0\}$, by definition of dark subspaces $D'\in \mathcal D_m$ and
		\[K(D',\mathcal W_{r_m}D)=\int_{\Ld}\one_{\mathcal W_{r_m}D}(vD')\tr(v\tfrac{\pi_{D'}}{r_m}v^*)\d\mu(v).\]
		\Cref{lem:properties cal W} ensures that $vD'\in \mathcal W_{r_m}D$ for any $v\in \mathcal V_1$ such that $vD'\neq\{0\}$, while
		$vD'=\{0\}$ implies $\operatorname{tr}(v\pi_{D'}v^*)=0$. Hence,  $K(D',\mathcal W_{r_m}D)=1$. So $\mathcal W_{r_m}D$ is $K$-invariant and the proposition is proved.
	\end{proof}

	\begin{proposition}\label{prop:smart-construction}
		Let $D_c \in \mathcal D_m$ and let $J \colon \cc^{r_m} \to D_c$ be a linear isometry. There exists a $D_c-$smart family $(J_D)_{D \in \mathcal D_m}$  such that $J_{D_c} = J$.
	\end{proposition}
	\begin{proof}
		Set $J_{D_c}= J$.
		From \Cref{prop:suppchi}, for every   $D \in \supp{\chi_{{\rm inv}}}$ we have $D \in {\mathcal{W}} D_c$, so there exists $w \in {\mathcal{W}}$ such that $D = w D_c$. This  guarantees that $\operatorname{tr}\!\big(w \pi_{D_c} w^*\big)>0$ and we  set 
		\[J_{D} = \tr(w \tfrac{\pi_{D_c}}{r_m} w^*)^{-\frac 12} wJ.\]
		For every   $D \notin \supp{\chi_{{\rm inv}}}$ we set $J_{D}$ to be an arbitrary isometry between $\cc^{r_m}$ and $D$. By construction, $(J_D)_{D \in \mathcal D_m}$ satisfies the desired properties.
	\end{proof}

	Next,  we show that  $D_c-$smartness is equivalent to smartness, provided that $D_c \in \supchi$.   
	Although the following simple lemma can be considered folklore, its proof is included for the sake of completeness as we will invoke this result repeatedly.
	
	\begin{lemma}\label{lem:unitarypowers}\label{lem:unitaryinclusion}
		Let $u \in \Urm$ and $A,B \subset \pcr$ be closed. Then,
		\begin{enumerate}[label={\rm (\roman*)}]
			\item  $(u^n)_{n=1}^\infty$ contains a subsequence convergent to $\id$.
			\item   $uA \subset B \subset A$ implies that $A = B= uA$.
		\end{enumerate}
	\end{lemma}
	\begin{proof}
		(i)	As a corollary to a more general theorem, this result can be found in \cite[Corollary~7]{Benitez2006}. There is also a very simple proof that exploits the unitarity of $u$. Namely, since $\Urm$ is compact,   $(u^n)_{n=1}^\infty$ has a convergent subsequence $(u^{n_\ell})_{\ell=1}^\infty$, which necessarily is also a Cauchy sequence, i.e., for every $\varepsilon > 0$ there exists $N \in \nn$ such that \[\forall \ell,\tilde \ell\geq N \ \ \|u^{n_{\tilde \ell}}-u^{n_{\ell}}\| = \|u^{ n_{\tilde \ell}-n_{\ell}} - \id \| <\varepsilon.\]
		For $q \in \nnone$ we put $m_q = n_{\tilde \ell}-n_{\ell}$, where $\ell, \tilde \ell$ are so chosen that $(m_q)_{q=1}^\infty$ is strictly increasing. Then $u^{m_q} \to \id$ as ${q \to \infty}$, as desired.
		
		\smallskip
		
		\noindent
		(ii)	From $uA \subset B \subset A$ it follows that
		$u^n A \subset B \subset A$ for every $n \in \nnone$.
		Extracting from $(u^n)_{n=1}^\infty$ a subsequence $(u^{n_k})_{k=1}^\infty$ convergent to $\id$, we obtain
		\[A = \lim_{k \to \infty} u^{n_k} A \subset B \subset A,\]
		which gives $A = B$.
		To prove that $A = uA$, it suffices to show  $A \subset   uA$. Observe that $u^{n_k-1}A \subset A$ for every $k \in \nnone$, thus also \[ u^{-1}A = \lim_{k \to \infty} u^{n_k-1}A \subset A,\] which concludes the proof.
	\end{proof}

	\begin{lemma}\label{lem:new_lemma_fix}
		Let \(D\in\operatorname{supp}\chi_{\mathrm{inv}}\), \(D'\in\mathcal D_m\), and let \((w_n)_n\subset\mathcal W\) and \((u_n)_n\subset\mathcal U(r_m)\) be such that for every \(n\) 
		\[ w_nD=D' \quad \text{and} \quad J_{D'}^{-1}w_nJ_D^{\vphantom{-1}}\propto u_n,\] 
		and assume that 
		$u_n\to u\in\mathcal U(r_m)$. 
		Then there exists \(w_\infty\in\mathcal W\) such that
		\[
		J_{D'}^{-1}w_\infty J_D^{\vphantom{-1}}\propto u.\]
	\end{lemma}
	\begin{proof}
		Since $D\in \supp\chi_{\rm inv}$, \Cref{prop:suppchi} implies there exists $w\in \mathcal W_{r_m}$ such that $D=wD$. Since $\rank w=r_m=\dim D$ and $D\subset \range w$, we have $D=\range w$. Then, the definition of dark subspaces implies that there exists $U\in \mathcal{U}(D)$ such that
		\[w|_D=\|w\pi_D\|U.\]
		Hence,
		\[w^m=\|w\pi_D\|^{m-1}U^{m-1}w \quad \text{and} \quad \frac{w^m}{\|w^m\|}=U^{m-1}w.\]
		Then \Cref{lem:unitarypowers} implies that there exists an accumulation point $w_D$ of $(w^m/\|w^m\|)_m$ such that
		\[w_D=U^{-1}w.\]
		By definition of $U$, $w_D|_D=\|w\pi_D\|\id_D$.
		
		Since  $w_nD=D'$ for any $n$, $w_n|_D\neq 0$ and \Cref{lem:properties cal W} implies $\tilde{w}_n=\frac{w_nw_D}{\|w_nw_D\|}\in \mathcal W_{r_m}$.
		Then, since $D=\range w_D$, 
		\[\tilde{w}_n=\frac{w_n\pi_Dw_D}{\|w_n\pi_Dw_D\|}.\]
		Since  $w_n|_D=\|w_n\pi_D\|\beta_n J_{D'}^{\vphantom{-1}}u_nJ_D^{-1} 
		$  for some $(\beta_n)_n\in \mathbb T^\nn$,  it follows that 
		\[
		\|w_n\pi_Dw_D\|
		=\|w_n\pi_D\|\,\|\pi_Dw_D\|
		=\|w_n\pi_D\|,
		\]
		where we used $\pi_Dw_D=w_D$ and $\|w_D\|=1$. Therefore,
		\[
		\widetilde w_n|_D
		=\|w\pi_D\|\beta_nJ_{D'}^{\vphantom{-1}}u_nJ_D^{-1}.
		\]
		\Cref{lem:properties cal W} ensures $(\tilde{w}_n)_n$ is relatively compact in $\mathcal W_{r_m}$. Let $w_\infty$ be one of its accumulation points. Then, using $u_n \to u$, 
		\[w_\infty|_D=\|w\pi_D\|\beta J_{D'}^{\vphantom{-1}}u J_D^{-1}\]
		for some $\beta\in \mathbb T$ and the lemma is proved.
	\end{proof}

	\begin{proposition}\label{prop:smart-propagates} 
		Let $D_c \in \supchi$.   If $\J = \{J_D\}_{D \in \mathcal D_m}$ is $D_c-$smart, then $\mathcal J$ is smart.  
	\end{proposition}
	\begin{proof} 
		Fix $D,D' \in \supchi$.  We need to show that   $J_{D'} \propto w J_{D}$ for some $w \in \mathcal W$. Observe that the $D_c-$smartness of $\J$  guarantees the existence of $w'  \in \mathcal W$ such that   $J_{D'}  \propto w' J_{D_c}$, so it is sufficient to construct  $\widetilde w \in \mathcal W$ such that   $ J_{D_c}  \propto \widetilde w J_{D}$ as we will then have $J_{D'} \propto w'  \widetilde w J^{\vphantom{-1}}_{D}$.  
		
		Using again the $D_c-$smartness of $\J$, we deduce the existence of $w_1 \in \mathcal W$ such that $J_{D}^{-1} w_1 J_{D_c}^{\vphantom{-1}} \propto \id$. Also, via \Cref{prop:suppchi}, there exists $w_2$ such that $D_c = w_2 D$, and for the corresponding unitary we put $u \propto J_{D_c}^{-1} w_2 J_{D}^{\vphantom{-1}}$. 
		The following diagram summarizes the dark subspaces and operators introduced so far.  
		
		\begin{center}
			\scalebox{0.6}{
				
				\tikzset{every picture/.style={line width=0.75pt}} 
				
				\begin{tikzpicture}[x=0.75pt,y=0.75pt,yscale=-1,xscale=1]
					
					\draw    (102.45,168.55) -- (102.45,138.03) -- (102.45,88.82) ;
					\draw [shift={(102.45,86.82)}, rotate = 90] [fill={rgb, 255:red, 0; green, 0; blue, 0 }  ][line width=0.08]  [draw opacity=0] (12,-3) -- (0,0) -- (12,3) -- cycle    ;
					
					\draw    (291.71,169.48) -- (291.71,138.96) -- (291.71,89.74) ;
					\draw [shift={(291.71,87.74)}, rotate = 90] [fill={rgb, 255:red, 0; green, 0; blue, 0 }  ][line width=0.08]  [draw opacity=0] (12,-3) -- (0,0) -- (12,3) -- cycle    ;
					
					\draw    (128,59.67) -- (266.67,59.67) ;
					\draw [shift={(126,59.67)}, rotate = 0] [fill={rgb, 255:red, 0; green, 0; blue, 0 }  ][line width=0.08]  [draw opacity=0] (12,-3) -- (0,0) -- (12,3) -- cycle    ;
					
					\draw    (314.77,59.67) -- (451.33,59.67) ;
					\draw [shift={(453.33,59.67)}, rotate = 180] [fill={rgb, 255:red, 0; green, 0; blue, 0 }  ][line width=0.08]  [draw opacity=0] (12,-3) -- (0,0) -- (12,3) -- cycle    ;
					
					\draw    (472.04,169.33) -- (472.04,138.81) -- (472.04,89.6) ;
					\draw [shift={(472.04,87.6)}, rotate = 90] [fill={rgb, 255:red, 0; green, 0; blue, 0 }  ][line width=0.08]  [draw opacity=0] (12,-3) -- (0,0) -- (12,3) -- cycle    ;
					
					\draw    (332.89,192.05) -- (446.67,192.05) ;
					\draw [shift={(448.67,192.05)}, rotate = 180] [fill={rgb, 255:red, 0; green, 0; blue, 0 }  ][line width=0.08]  [draw opacity=0] (12,-3) -- (0,0) -- (12,3) -- cycle    ;
					
					\draw    (124.67,50.67) .. controls (172.85,20.85) and (214.49,19.9) .. (272.24,50.28) ;
					\draw [shift={(274,51.21)}, rotate = 208.11] [fill={rgb, 255:red, 0; green, 0; blue, 0 }  ][line width=0.08]  [draw opacity=0] (12,-3) -- (0,0) -- (12,3) -- cycle    ;
					
					\draw    (270.48,75.11) .. controls (222.74,103.91) and (181.12,104.09) .. (123.33,73.23) ;
					\draw [shift={(272.67,73.77)}, rotate = 148.25] [fill={rgb, 255:red, 0; green, 0; blue, 0 }  ][line width=0.08]  [draw opacity=0] (12,-3) -- (0,0) -- (12,3) -- cycle    ;
					
					\draw    (274.48,207.11) .. controls (226.74,235.91) and (185.12,236.09) .. (127.33,205.23) ;
					\draw [shift={(276.67,205.77)}, rotate = 148.25] [fill={rgb, 255:red, 0; green, 0; blue, 0 }  ][line width=0.08]  [draw opacity=0] (12,-3) -- (0,0) -- (12,3) -- cycle    ;
					
					\draw    (141.33,191.67) -- (259.33,191.67) ;
					\draw [shift={(139.33,191.67)}, rotate = 0] [fill={rgb, 255:red, 0; green, 0; blue, 0 }  ][line width=0.08]  [draw opacity=0] (12,-3) -- (0,0) -- (12,3) -- cycle    ;
					
					\draw    (128,176.6) .. controls (176.18,146.78) and (217.83,145.83) .. (275.58,176.22) ;
					\draw [shift={(277.33,177.15)}, rotate = 208.11] [fill={rgb, 255:red, 0; green, 0; blue, 0 }  ][line width=0.08]  [draw opacity=0] (12,-3) -- (0,0) -- (12,3) -- cycle    ;
					
					\draw (195,210) node [anchor=north west][inner sep=0.75pt]  [font=\Large]  
					{$u$};
					
					\draw (192,170) node [anchor=north west][inner sep=0.75pt]  [font=\Large]  {$\Id$};
					
					\draw (192,132) node [anchor=north west][inner sep=0.75pt]  [font=\Large]  {$\Id$};
					
					\draw (378,170) node [anchor=north west][inner sep=0.75pt]  [font=\Large]  {$\Id$};
					
					\draw (188,77) node [anchor=north west][inner sep=0.75pt]  [font=\Large]  {$w_{2}$};
					
					\draw (188,42) node [anchor=north west][inner sep=0.75pt]  [font=\Large]  {$w_{1}$};
					
					\draw (190,7) node [anchor=north west][inner sep=0.75pt]  [font=\Large]  {$\tilde{w}$};
					
					\draw (375,37) node [anchor=north west][inner sep=0.75pt]  [font=\Large]  {$w'$};
					
					\draw (73,120) node [anchor=north west][inner sep=0.75pt]  [font=\Large]  {$J_{D}$};
					
					\draw (257,120) node [anchor=north west][inner sep=0.75pt]  [font=\Large]  {$J_{D_{c}}$};
					
					\draw (440,120) node [anchor=north west][inner sep=0.75pt]  [font=\Large]  {$J_{D'}$};
					
					\draw (90,50) node [anchor=north west][inner sep=0.75pt]  [font=\LARGE]  {$D$};
					
					\draw (277,50) node [anchor=north west][inner sep=0.75pt]  [font=\LARGE]  {$D_{c}$};
					
					\draw (462,50) node [anchor=north west][inner sep=0.75pt]  [font=\LARGE]  {$D'$};
					
					\draw (83,180) node [anchor=north west][inner sep=0.75pt]  [font=\LARGE]  {$\mathbb{C}^{r_{m}}$};
					
					\draw (280,180) node [anchor=north west][inner sep=0.75pt]  [font=\LARGE]  {$\mathbb{C}^{r_{m}}$};
					
					\draw (458,180) node [anchor=north west][inner sep=0.75pt]  [font=\LARGE]  {$\mathbb{C}^{r_{m}}$};

				\end{tikzpicture}
				
			}
		\end{center}

		It follows that $w_2w_1 D_c = D_c$ and
		\[J_{D_c}^{-1} w_2   w_1 J_{D_c}^{\vphantom{-1}} =  \big(J_{D_c}^{-1} w_2 J_{D}^{\vphantom{-1}}\big)\big( J_{D}^{-1} w_1 J_{D_c}^{\vphantom{-1}}) \propto u;\]
		thus, for every $k \in \nnone$ we have
		\[J_{D_c}^{-1} \big(w_2   w_1\big)^{k-1} w_2 J_{D}^{\vphantom{-1}} \propto   u^k.\]
		By \Cref{lem:unitarypowers}, we can extract a subsequence $(k_n)_{n \in \nn}$ such that $u^{k_n} \to \id$ as $n \to \infty$. Hence, putting
		\[\widetilde w_k = \frac{(w_2   w_1)^{k-1} w_2}{\|(w_2   w_1)^{k-1} w_2\|} \in \W\] 
		for $k \in \nnone$, we obtain
		\[J_{D_c}^{ {-1}}  \widetilde w_{k_n} J_{D}^{\vphantom{-1}} \propto u^{k_n}.\]

		Since $\widetilde w_{k_n}D=D_c$ for every $n$, we may apply
		\Cref{lem:new_lemma_fix} with $D$, $D_c$, $\widetilde w_{k_n}$ and
		$u^{k_n}$. Hence, there exists $\widetilde w\in\mathcal W$ such that
		\[
		J_{D_c}^{-1}\widetilde wJ_D\propto\Id.
		\]
		Equivalently,
		\[
		J_{D_c}\propto\widetilde wJ_D,
		\]
		as desired.
	\end{proof}
	
	\Cref{prop:smart-propagates} states that if the central space $D_c$ is in $\supchi$, then the
	defining property of a $D_c-$smart family, which we recall is the possibility of going from $D_c$ to any other dark subspace in $\supp \chi_{{\rm inv}}$ via some map from $\W$ without inducing any non-trivial dynamics on $\cc^{r_m}$, propagates to the whole $\supp \chi_{{\rm inv}}$, i.e., for any $D,D' \in \supp\chi_{\rm inv}$ there exists $w \in \W$ such that 
	\[J_{D'}^{{-1}} w J^{\vphantom{-1}}_{D} \propto \id.\]
	Equivalently, going from $D$ to $D'$ via the reference space, i.e., $D  \xrightarrow{J^{{-1}}_{D}} \cc^{r_m} \xrightarrow{J_{D'}} D'$, 
	does not induce any extra dynamics besides what can be achieved with the help of maps from $\W$: 
	\[J_{D'}^{\vphantom{-1}} J^{{-1}}_{D} \propto w|_{D}.\]
	This leads us to another auxiliary result about smart families. Namely, we identify a dense subset of $G_\J$  whose elements have a particularly nice form. To proceed, we need to introduce some  notation. 
	
	Here and henceforth, for a family $\J = \{J_D\}_{D\in \D}$ we denote by $u_{w, D}$ a special unitary operator on $\cc^{r_m}$ proportional to $J_{wD}^{-1} w J_{D}^{\vphantom{-1}}$, provided that $\operatorname{tr}(w\pi_Dw^*) >0$, and we put 
	\[R_\J = \{u_{w,D} \in \mathcal{SU}(r_m) \,|\, w \in \W,\: D \in \supp\chi_{{\rm inv}},\: \operatorname{tr}(w \pi_D w^*)>0\}.\]
	Also, we denote the set of finite products of generators of $G_\J$ as
	\[S^{\rm fin}_\J =\{s_1\cdots s_n \:|\: s_i \in S_\J, \text{ for } i=1, \ldots, n \ \text{ and } n\in \nnone\},\]
	where we recall that $S_\J$ is defined as
	\begin{equation*}
		S_\J=\{u_{v,D} \in \mathcal{SU}(r_m) \,|\, v \in \supp\mu,\: D \in \supp\chi_{{\rm inv}},\: \operatorname{tr}(v \pi_D v^*)>0\}.
	\end{equation*} 
	
	The following proposition states that in the case of smart families, the  elements of the form $u_{w,D}$ are dense in $G_\J$. Note, however, that we need not have $R_\J \subset G_\J$.  
	
	\begin{proposition}\label{prop:smartgroup}    
		Let $\J$ be a family of isometries. Then
		\begin{enumerate}[label={\rm (\roman*)}] 
			\item   $S^{\rm fin}_\J$ is a dense subset of $G_\J$;
			\item  If $\J$ is   smart, then $S^{\rm fin}_\J  \subset   R_\J$. 
		\end{enumerate}  
		As a consequence, if $\J$ is   smart, then  $G_\J \cap R_\J$ is dense in $G_\J$. 
	\end{proposition} 
	\begin{proof}
		
		\noindent  (i)   Recall that, by definition, $G_\J$ is the closure of \[{\braket{S_\J}} =\{s_1^{\epsilon_1}\cdots s_\ell^{\epsilon_\ell} \:|\: s_i \in S_\J,\, {\epsilon_i} = \pm 1 \text{ for } i=1, \ldots, \ell \ \text{ and } \ell\in \nnone\}.\]  We will show that \[S^{\rm fin}_\J \subset \braket{S_\J} \subset \overline {S^{\rm fin}_\J}.\]
		The first inclusion holds trivially. As for the second one, consider $s = s_1^{\epsilon_1}\cdots s_\ell^{\epsilon_\ell}\in \langle S_\J | S_\J \rangle$. By  \Cref{lem:unitaryinclusion}(i) applied to the block-diagonal unitary
		$\operatorname{diag}(s_1,\ldots,s_\ell)$, there exists a subsequence
		$(n_k)_k\subset\mathbb N$ such that  for every $i\in\{1,\ldots,\ell\}$ we have 
		$s_i^{n_k-1}\to s_i^{-1}$. 
		For $k \in \nn$ define 
		\[\tilde s_{i,k} = \left\lbrace \begin{array}{ll}  s^{n_k -1}_i & \text{if} \   \epsilon_i =-1 \\ 
			s_i & \text{if} \ \epsilon_i =1
		\end{array}\right.\]
		and   $\tilde s_k = \tilde s_{1,k} \cdots \tilde s_{\ell,k}$. 
		
		Then $\tilde s_k\in S^{\rm fin}_\J $  and 
		$\tilde s_k \to s$  as $k \to 
		\infty$; 
		hence, $s \in \overline {S^{\rm fin}_\J}$, as claimed. 
		It follows, by taking closures, that 
		\[\overline {S^{\rm fin}_\J} = \overline {\braket{S_\J}} = G_\J,\]
		which concludes the proof of (i).

		\noindent  (ii)      
		We show that every finite product  of elements from $S_\J$ is of the form $u_{w,D} \propto J_{wD}^{-1}wJ_D^{\vphantom{-1}}$ with some $w \in \W$ and $D \in \supp \chi_{{\rm inv}}$.
		Consider
		\begin{align*}
			\prod_{i=1}^n u_{v_i,D_i}  & = \big(  J_{v_1D_1}^{-1}  v_1 J_{D_1}^{\vphantom{-1}}\big) \big(J_{v_2D_2}^{-1} v_2 J_{D_2}^{\vphantom{-1}}\big) \cdots \big(J_{v_{n-1}D_{n-1}}^{-1} v_{n-1} J_{D_{n-1}}^{\vphantom{-1}}\big) \big(J_{v_nD_n}^{-1} v_n J_{D_n}^{\vphantom{-1}}\big)
			\\
			&= \  \: J_{v_1D_1}^{-1}  v_1 \big(J_{D_1}^{\vphantom{-1}}  J_{v_2D_2}^{-1}\big) v_2  J_{D_2}^{\vphantom{-1}}\: \cdots \:\: J_{v_{n-1}D_{n-1}}^{-1}  v_{n-1} \big(J_{D_{n-1}}^{\vphantom{-1}} J_{v_nD_n}^{-1}\big) v_n J_{D_n}^{\vphantom{-1}}.
		\end{align*}
		From \Cref{prop:smart-propagates} we deduce that 
		\[\forall {i\in\{2,\ldots, n\}}   \ \ \exists w_i \in \mathcal W \quad  J_{D_{i-1}}^{\vphantom{-1}} J_{v_iD_i}^{-1} \propto  w_i|_{v_iD_i}\]
		and so 
		\[\prod_{i=1}^n u_{v_i,D_i} = J_{v_1D_1}^{-1} v_1  w_2 v_2 w_3 v_3 \cdots w_n v_n J_{D_n}^{\vphantom{-1}}.\]
		Denoting $\tilde w = v_1  w_2 v_2 w_3 v_3 \cdots w_n v_n$ and $w = \tilde w / \|\tilde w\|$, we have $v_1D_1 =   w D_n$ and
		\[\prod_{i=1}^n u_{v_i,D_i} =    J_{ w D_n}^{-1}  w J_{D_n}^{\vphantom{-1}} \propto u_{w, D_n},\]  	 as desired. 
	\end{proof}

	From the next lemma we will conclude the invariance of $\supp\chi_{\rm inv}$ under $\mathcal W$. The lemma itself is formulated in a slightly more general setting, covering also the case of $\pcd$, which we will exploit later on. Recall that $F \subset \Sd$ is called  $\Pi$-\emph{invariant} if $\Pi(\rho, F)= 1$ for every $\rho \in F$. 
	\begin{lemma}\label{lem:Piinvclosed}
		Let $F \subset \Sd$ be a closed and $\Pi$-invariant subset of $\Sd$. Then for every  $\rho \in F$ and $w \in \W$ such that $\tr(w\rho w^*)>0$ we have $\frac{w\rho w^*}{\operatorname{tr}(w\rho w^*)}\in F$.
	\end{lemma}
	\begin{proof}Let $\rho \in F$. First we show that the claim holds for maps from     $\supp\mu$.  Suppose, for contradiction, that $\tfrac{v_0\rho v_0^*}{\operatorname{tr}(v_0\rho v_0^*)} \notin F$ for some   $v_0\in \supp \mu$ such that $\operatorname{tr}(v_0 \rho v_0^*)>0$. 
		By continuity and since $F$ is closed, there exists an open neighbourhood  $V \subset \Ld$  of $v_0$ such that $\tr(v\rho v^*) > 0$ and $\tfrac{v\rho v^*}{\tr(v\rho v^*)} \notin   F$ for every $v \in  V$.  Therefore,
		\[\int_{V}\mathbf 1_F\Big(\tfrac{v\rho v^*}{\tr(v\rho v^*)}\Big) \tr(v\rho v^*) \d\mu(v) = 0 \ \ \text{ and } \ \ \int_{V}\tr(v\rho v^*) \d\mu(v) >0.\] 
		Then,
		\begin{align*}
			1 = \Pi(\rho, F) =  \int_{\Ld}\mathbf 1_F\Big(\tfrac{v\rho v^*}{\tr(v\rho v^*)}\Big) \tr(v\rho v^*) \d\mu(v)
			& =
			\int_{\Ld\setminus V }\mathbf 1_F\Big(\tfrac{v\rho v^*}{\tr(v\rho v^*)}\Big) \tr(v\rho v^*) \d\mu(v)
			\\ & \leq \int_{\Ld\setminus V } \tr(v\rho v^*) \d\mu(v)<1, 
		\end{align*} 
		which, by contradiction, implies that $\tfrac{v_0\rho v_0^*}{\operatorname{tr}(v_0\rho v_0^*)} \in F$.
		By repeated application of this result,   we deduce that $\tfrac{v\rho v^*}{\tr(v\rho v^*)} \in F$ for every
		$v \in \W_{\rm fin}$   such that $\tr(v\rho v^*)>0$.
		
		To conclude, let us consider $w \in \W$ such that $\tr(w\rho w^*)>0$. Since $\W$ is the closure of $\W_{\rm fin}$, $F$ is closed and $\tr(w\rho w^*)>0$, by continuity $\tfrac{w\rho w^*}{\tr(w\rho w^*)} \in F$, as desired. 
	\end{proof}

	To derive the $\W$-invariance of $\supp\chi_{\rm inv} \subset \D$ by means of this lemma, we consider $\nu_{\rm ch}$, the image measure of $\chi_{{\rm inv}}$ by the map 
	\[\begin{array}{rcl} q\colon\ \D&\longrightarrow&\Sd\\[0.33em] D &\longmapsto& \frac{\pi_D}{r_m}. \end{array}\]
	\begin{proposition}\label{prop:suppchi-invariance}
		If $D \in \supp\chi_{\rm inv}$ and $w \in \mathcal W$ satisfy $wD \neq \{0\}$, then   $wD \in \supp\chi_{\rm inv}$.
	\end{proposition}
	\begin{proof}  
		Since $\D$ is compact and $K$ is Feller with a unique invariant probability measure, \cite[Proposition~5.12]{benaim2022markov} implies that  $\supp\chi_{\rm inv}$ is a $K$-invariant set. Then, since $q$ is continuous and one to one,
		$\supp\nu_{\rm ch} = \{ {\pi_D}/{r_m} \,|\, D \in \supp \chi_{{\rm inv}} \}$ is 
		a $\Pi$-invariant subset of $\Sd$. Thus we can apply  \Cref{lem:Piinvclosed}, which guarantees that  
		if $D \in \supp \chi_{{\rm inv}}$ and $w \in \W$ satisfy $\operatorname{tr}(w\pi_D w^*)>0$, then 
		\[\frac{\pi_{wD}}{r_m} = \frac{w\pi_D w^*}{\tr(w\pi_D w^*)}   \in \supp\nu_{\rm ch};\] hence, $wD \in \supp \chi_{{\rm inv}}$. To conclude the proof, it remains to recall  that the condition $\operatorname{tr}(w \pi_D w^*) > 0$ is equivalent to $wD \neq \{0\}$.
	\end{proof}

	We are now in a position to prove the main result of this section, namely that the notions of minimality  and smartness of families are equivalent. In what follows, we consider two families $\J= \{J_D\}_{D\in \mathcal D_m}$ and $\tJ= \{ \widetilde  J_D \}_{D\in \mathcal D_m}$ and adopt the following notation. The (special) unitaries on the reference space are, respectively, \[u_{w,D} \propto J_{w D}^{-1}w J^{\vphantom{-1}}_{D} \quad \textrm{ and }\quad \widetilde u_{w,D} \propto \widetilde J_{w D}^{-1}w \widetilde J^{\vphantom{-1}}_{D},\] where
	$w \in \mathcal W$ and $D \in \supp \chi_{{\rm inv}}$ satisfy $\operatorname{tr}(w \pi_D w^*)>0$. Also, we put \[Q_D = J^{{-1}}_{D} \widetilde J_{D}^{\vphantom{-1}}\] for $D \in \mathcal D_m$.
	Note that $Q_D \in \Urm$ and  $
	\widetilde J_{w D}^{-1}w \widetilde J^{\vphantom{-1}}_{D}  = Q_{wD}^{-1}\big( J_{w D}^{-1}w J^{\vphantom{-1}}_D \big)  Q_D$,
	that is
	\begin{equation}\label{eq:utilde}
		\widetilde u_{w,D}  =  Q_{wD}^{-1}   u_{w,D}Q_D.
	\end{equation}

	\begin{lemma}\label{lem:QDinG}
		Let $D_c \in \supchi$. If $\J= \{J_D\}_{D\in \mathcal D_m}$ is  $D_c-$smart and   $\tJ= \{ \widetilde  J_D \}_{D\in \mathcal D_m}$ is an arbitrary family, then $Q_D^{-1} Q_{D'}^{\vphantom{-1}} \in G_{\tJ}$ for all $D,D' \in \supchi$.
	\end{lemma}
	\begin{proof}
		Let $D,D' \in \supchi$. First, we show that $Q_{D}^{-1}Q_{D_c}^{\vphantom{-1}} \in G_{\tJ}$.
		By the $D_c-$smartness of $\J$,   there exists $  w \in \mathcal W$ such that $J_D \propto w J_{D_c}$, i.e., $u_{w,D_c} = \id$, so \Cref{eq:utilde} leads to
		$ Q_{D}^{-1}Q_{D_c}^{\vphantom{-1}} =   \widetilde u_{w,D_c}  \in G_{\tJ}$,
		as claimed.
		Analogously, $Q_{D'}^{-1}Q_{D_c}^{\vphantom{-1}} \in G_{\tJ}$. It remains to observe that  \[Q_{D}^{-1}Q_{D'}^{\vphantom{-1}} = \big(Q_{D}^{-1}Q_{D_c}^{\vphantom{-1}}\big) \big(Q_{D'}^{-1}Q_{D_c}^{\vphantom{-1}}\big)^{-1}  \in G_{\tJ},\] which concludes the proof.
	\end{proof}

	\begin{theorem}\label{thm:smartiffminimal}
		A family of isometries is minimal iff it is smart.
	\end{theorem}
	\begin{proof}
		\noindent
		($\Leftarrow$) Let  $\J$ be a smart  family   and  $\tJ$ an arbitrary family. We show
		\begin{equation}\label{eq:GJGJtilde}
			G_{\J} \subset Q_{D'}^{\vphantom{-1}} G_{\tJ} Q_{D'}^{-1} \qquad  \forall D' \in \supchi\end{equation}
		from which minimality follows immediately.
		Fix $D' \in \supchi$.  From \Cref{eq:utilde},
		\begin{equation*}
			u_{v,D} =  Q_{vD} \widetilde u_{v,D}  Q_D^{-1}  =    Q^{\vphantom{-1}}_{D'} \big( Q^{{-1}}_{D'}  Q_{vD}^{\vphantom{-1}} \big)  \widetilde u_{v,D}  \big(Q_D^{-1} Q^{\vphantom{-1}}_{D'}\big) Q_{D'}^{-1}
		\end{equation*}
		for   every $v \in \supp \mu$ and $D \in \supchi$ such that  $\operatorname{tr}(v \pi_D v^*)>0$.
		\Cref{prop:suppchi-invariance} guarantees $vD \in \supchi$ and from \Cref{lem:QDinG} it follows that both  $Q_D^{-1} Q^{\vphantom{-1}}_{D'}$  and $Q^{{-1}}_{D'}  Q_{vD}^{\vphantom{-1}}$ are in $G_{\tJ}$; hence, $  u_{v,D} \in Q_{D'}^{\vphantom{-1}} G_{\tJ}  Q_{D'}^{-1}$. This means
		$G_{\J} \subset Q_{D'}^{\vphantom{-1}} G_{\tJ} Q_{D'}^{-1}$, as desired.
		
		\smallskip
		
		\noindent
		($\Rightarrow$)  Let $ \tJ =\{\widetilde J_D\}_{D\in \mathcal D_m}$ be   minimal and let   $D_c \in \supchi$. We show that $\tJ$ is $D_c-$smart. It follows from \Cref{prop:smart-construction} that there exists  a $D_c-$smart family $ {\mathcal J}=\{   J_D \}_{D\in \mathcal D_m}$   such that   $J_{D_c} = \widetilde J_{D_c}$.
		First, we show that $G_{\J} = G_{\tJ}$.
		From the ($\Leftarrow$) proof above, $\J$ is minimal and, since $Q_{D_c} = \id$, from \Cref{eq:GJGJtilde} it follows that $G_{\J} \subset  G_{\tJ}$. On the other hand, the minimality of $\tJ$ implies that there exists $Q \in \Urm$ such that
		$G_{\tJ} \subset QG_\J Q^{-1}$. Therefore,
		\[G_\J \subset G_{\tJ} \subset QG_\J Q^{-1}.\] It now suffices to proceed as in the proof of \Cref{lem:unitaryinclusion}(ii); that is, for every $n \in \nnone$, 
		$G_\J \subset Q^n G_\J Q^{-n}$ and there exists a subsequence $(Q^{n_k})_{k=1}^\infty$ convergent to $\id$, see \Cref{lem:unitarypowers}(i). Hence,
		\[G_\J \subset G_{\tJ} \subset Q^{n_k} G_\J Q^{-n_k}  \xrightarrow{k \to \infty} G_\J,\]
		and so $G_\J = G_{\tJ}$, as claimed.
		In what follows, we denote this common group by $G$.
		
		We now fix $D \in \supchi$ and prove the $D_c-$smartness of $\tJ$ by constructing $\widetilde w \in \mathcal W$ such that $ \widetilde J_D = \widetilde  w   \widetilde J_{D_c}$.
		Since $Q_{D_c}=\id$, \Cref{lem:QDinG} implies $Q_D \in G$. 
		Assume first that $Q_D \in R_\J$, i.e.,  \[Q_D = u_{w,D'} \propto J_{wD'}^{-1} w J_{D'}^{\vphantom{-1}}\] for some $D' \in \supp\chi_{\rm inv}$ and $w \in \W$ such that $\operatorname{tr}(w \pi_{D'} w^*)>0$. Since $\J$ is $D_c-$smart, there exist $w_1, w_2 \in \mathcal W$ such that $J_D \propto w_1  J_{D_c}$ and $J_{D'} \propto w_2 J_{D_c}$, and  from  \Cref{prop:smart-propagates} there exists  $w_3  \in \mathcal W$ such that $J_{D_c}^{\vphantom{-1}} \propto   w_3 J_{wD'}$. It follows that
		\[\widetilde J_D =  J_D  Q_D = J_D  u_{w,D'} \propto  w_1  J_{D_c}^{\vphantom{-1}} J_{wD'}^{-1} w J_{D'}^{\vphantom{-1}}  \propto  w_1  w_3 w w_2 J_{D_c}^{\vphantom{-1}};\]
		thus, since $J_{D_c} = \widetilde J_{D_c}$, for $\widetilde w = \frac{w_1  w_3 w w_2}{\|w_1  w_3 w w_2\|} \in \mathcal W$,
		$\widetilde J_D   \propto  \widetilde w \widetilde J_{D_c} $, as desired.

		It remains to consider the case $Q_D \notin R_\J$. From \Cref{prop:smartgroup}   it follows that $Q_D = \lim_{n \to \infty} u_{n}$ with $(u_n)_n \subset R_\J$. 
		Applying the preceding case to each $u_n$, we obtain $w_n\in\mathcal W$
		such that
		\[
		J_D^{-1}w_nJ_{D_c}\propto u_n.
		\]
		Hence, by
		\Cref{lem:new_lemma_fix}, there exists
		$\widetilde w\in\mathcal W$ such that
		$
		J_D^{-1}\widetilde wJ_{D_c}^{\vphantom{-1}}\propto Q_D
		$. 
		Consequently, using also $J_{D_c}=\widetilde J_{D_c}$, 
		\[\widetilde J_D = J_DQ_D \propto   \widetilde wJ_{D_c} = \widetilde w\widetilde J_{D_c},\]
		which concludes the proof.
	\end{proof}

	We can now also show that all minimal families produce the same set of $\Pi$-invariant measures on $\pcd$:
	\begin{theorem}\label{thm:minimalfammeasures} 
		Let $\J, \tJ$ be two minimal families. Then $\Upsilon_{\J \vphantom{\tJ}}  = \Upsilon_{\tJ}$. 
	\end{theorem}
	\begin{proof}Let $\hat x \in \pcr$. By permutation symmetry between $\J$ and $\tJ$, it is sufficient to show that 
		$\nu_{\hat x,\tJ} \in \Upsilon_{\J}$.  
		Since $\J = \{J_D\}_{D \in \mathcal D_m}$ and $\tJ = \{\tilde J_D\}_{D \in \mathcal D_m}$ are both minimal, the definitions of minimality and $Q_{D'}$ imply
		\[G_{\J  \vphantom{\tJ}} = Q_{D'}^{\vphantom{-1}} G_{\tJ} Q_{D'}^{-1} \qquad   \forall D'  \in \supp \chi_{\rm inv}.\]
		Also, via \Cref{lem:QDinG}, the minimality of $\J$ guarantees that  
		\[Q_D^{-1}Q_{D'}^{\vphantom{-1}} \in G_{\tJ}\qquad   \forall D, D'  \in \supp \chi_{\rm inv}.\] 
		Let $f \in \mathcal C (\pcd)$. Using the unitary invariance of $\operatorname{H}_{G_{\tJ}}$ and the substitution $u = Q_{D'}^{\vphantom{-1}}\tilde u\, Q_{D'}^{-1} $, it follows that
		\begin{align*}
			\int_{\pcd} f(\hat z) \d \nu_{\hat x,\,\tJ}(\hat z)
			& =  
			\int_{\mathcal D_m} \int_{G_{\tJ}} f\big(\tilde J_D^{\vphantom{-1}} Q_D^{-1}Q_{D'}^{\vphantom{-1}} \tilde u \cdot \hat x\big) \d \mathrm{H}_{G_{\tJ}}(\tilde u) \d \chi_{\rm inv}(D) 
			\\ & =  
			\int_{\mathcal D_m} \int_{G_{\tJ}} f\big(  J_D^{\vphantom{-1}} Q_{D'}^{\vphantom{-1}}\tilde u\, Q_{D'}^{-1} Q_{D'}^{\vphantom{-1}} \cdot \hat x\big) \d \mathrm{H}_{G_{\tJ}}(\tilde u) \d \chi_{\rm inv}(D) 
			\\ & =  
			\int_{\mathcal D_m} \int_{G_{\J}} f\big(  J_D^{\vphantom{-1}}    u \,  Q_{D'}^{\vphantom{-1}} \cdot \hat x\big) \d \mathrm{H}_{G_{\J}}(u) \d \chi_{\rm inv}(D) 
			\\ & = 
			\int_{\pcd} f(\hat z) \d \nu_{Q_{D'} \cdot \hat x,\,\J}(\hat z).
		\end{align*}
		Hence, $\nu_{\hat x,\,\tJ} = \nu_{Q_{D'} \cdot \hat x,\,\J \vphantom{\tJ}} \in \Upsilon_{\J}$, as desired. 
	\end{proof}

	\section{Invariant measures on quantum trajectories, proof of \Cref{thm:intro-ergodic,thm:introuniq}}\label{sec:invmeasures}
	Recall that $\Pi$-ergodic measures are the extreme points of the convex set of $\Pi$-invariant Borel probability measures.
	The main result of this section is the characterization of $\Pi$-ergodic measures. It is obtained via a result of Raugi \cite[Proposition~3.2]{Raugi1992}, see also \cite[Proposition~2.9]{BenoistQuint2014}, which, adapted to our context, reads:
	\begin{theorem}\label{thm:raugi}
		If the kernel $\Pi$ is equicontinuous, i.e., if  $(\Pi^nf)_n$ is equicontinuous for any $f\in \mathcal C(\pcd)$, then:
		\begin{enumerate}[label={\rm (\roman*)}]
			\item the support of any $\Pi$-ergodic measure on $\pcd$ is a $\Pi$-minimal subset of $\pcd$;
			\item  any $\Pi$-minimal subset of $\pcd$  carries a unique $\Pi$-invariant probability measure.
		\end{enumerate}
	\end{theorem} 
	The two ingredients needed to obtain the complete characterization of $\Pi$-ergodic measures via Raugi's theorem are the characterization of $\Pi$-minimal subsets of $\pcd$ and the proof of $\Pi$ equicontinuity. We provide them in \Cref{subs:piinvariant,subs:equicont} respectively. In the final \Cref{subs:ergo} we apply \Cref{thm:raugi} to conclude that the set of $\Pi$-ergodic measures coincides with $\Upsilon_{\J}$, provided that $\J$ is minimal. 
	
	\subsection{Characterization of minimal sets}\label{subs:piinvariant}
	Here we prove that the $\Pi$-minimal subsets of $\pcd$ are exactly those of the form  $\supp \nu_{\hat x,\J}$, provided that  $\J$ is a  minimal family.  Recall that $F \subset \pcd$ is called $\Pi$-\emph{invariant} if    $\Pi(\hat y, F) = 1$ for every $\hat y \in F$, and that a $\Pi$-invariant, closed and non-empty set is called  $\Pi$-\emph{minimal} if it contains no proper subset enjoying these three properties.
	Since quantum trajectories converge (exponentially fast) towards dark subspaces, see \Cref{thm:conv to Dark}, we can restrict our search for $\Pi$-minimal subsets of $\pcd$ to (projective) dark subspaces, as expressed in the following proposition, which is a corollary to \Cref{thm:conv to Dark}. 
	\begin{proposition}\label{prop:invariantindark}
		If $F \subset \pcd$ is  $\Pi$-minimal, then  $F \subset \bigcup_{D \in \mathcal D_m}\P(D)$. 
	\end{proposition}
	In the context of pure trajectories, \Cref{lem:Piinvclosed} can be restated as follows:
	\begin{proposition} \label{prop:invariance_pure}  Let $F$ be a closed and $\Pi$-invariant   subset of $\pcd$.
		If $\hat y \in F$ and $w \in \mathcal W$ are such that $w y \neq 0$, then $w \cdot \hat y \in F$.
	\end{proposition}
	
	Before we state and prove the main theorem of this section, we state and prove a standard lemma. Although the result is standard, the formulation we require is not quite so. For example, we do not require that the Markov kernel admit a unique invariant measure. That distinguishes it from \cite[Proposition~5.12]{benaim2022markov} for example. Hence, we provide a short proof.
	\begin{lemma} \label{lem:suppisinv}
		Let $\mathcal T$ be a  metrizable topological space equipped with its Borel $\sigma$-algebra. Let $T$ be a Feller Markov kernel on $\mathcal T$ and $\nu$ a $T$-invariant probability measure over $\mathcal T$. Then $\supp\nu$ is a $T$-invariant subset of $\mathcal T$. That is, $T(x,\supp\nu)=1$ for any $x\in \supp\nu$. 
	\end{lemma} 
	\begin{proof}  
		Assume there exists $x\in \supp\nu$ such that $T(x,\supp\nu)<1$. Then $Tf(x)>0$ for any continuous function $f$ such that $f(y)=0$ for every $y\in \supp\nu$ and $f(y)>0$ for every $y\not\in\supp\nu$. Since $T$ is Feller, $x\mapsto Tf(x)$ is continuous. Therefore, there exists an open neighborhood $A$ of $x$ such that $Tf(y)>0$ for any $y\in A$. By definition of $\supp\nu$ we have $\nu(A)>0$; hence, $\int_{\mathcal T} Tf(y)\d\nu(y)>0$. However, since $\nu$ is $T$-invariant, $\int_{\mathcal T} Tf(y)\d\nu(y)=\int_{\mathcal T} f(y)\d\nu(y)=0$. Thus, by contradiction, $T(x,\supp\nu)=1$ for any $x\in \supp\nu$.  
	\end{proof}
	
	We are now equipped to prove the main result of this section.

	\begin{theorem}\label{thm:piminimalsets}
		Let $\J$ be a minimal family. Then every $\Pi$-minimal subset of $\pcd$ is of the form  $\supp \nu_{\hat x,\J}$ for some $\hat x \in \pcr$.   
	\end{theorem} 
	
	\begin{proof}Fix a minimal family $\J = (J_D)_{D \in \mathcal D_m}$ and a set $F \subset \pcd$. We will show that   $F$ is $\Pi$-minimal iff $F = \supp \nu_{\hat x,\J}$ for some $\hat x \in \pcr$. To alleviate notation, we drop the subscript $\J$ when referring to the measure $\nu_{\hat x,\J}$. 
		
		\smallskip
		\noindent 
		$(\Leftarrow)$     Fix  $\hat x \in \pcr$, we aim to show that $\supp \nu_{\hat x} = \{ {J_Du \cdot \hat x} \,|\,  D \in \supp\chi_{{\rm inv}} , u \in G_\J\}$ is $\Pi$-minimal. Since $\Pi$ is Feller and $\nu_{\hat x}$ is $\Pi$-invariant, \Cref{lem:suppisinv} implies that  $\supp\nu_{\hat x}$ is $\Pi$-invariant. It only remains to show that it has no non-trivial $\Pi$-invariant subsets. Let $F$ be a non-empty, closed, and  $\Pi$-invariant subset of $\supp \nu_{\hat x}$. To conclude that $\supp \nu_{\hat x} = F$, it suffices to show that $\supp \nu_{\hat x} \subset F$.
		To this end, we fix \[\hat y_0 = J_{D_0}u_0 \cdot \hat x \in F \subset \supp \nu_{\hat x} \quad \textrm{ and }\quad \hat y_1 = J_{D_1}u_1 \cdot \hat x \in \supp \nu_{\hat x}.\] In order to conclude that $\hat y_1 \in F$, we show that $\hat y_1$ can be reached from $\hat y_0$ with the help of a map from $\W$.

		First, assume that $u_1^{\vphantom{-1}} u_0^{-1} \in   R_\J$.  By the definition of $R_{\mathcal J}$, \[u_1^{\vphantom{-1}} u_0^{-1} =    u_{w,D} \propto J_{wD}^{-1} w  J_{D}^{\vphantom{-1}}\] for some $D \in \supp\chi_{{\rm inv}}$ and $w \in \W$ such that $\operatorname{tr}(w \pi_{D} w^*)>0$.   
		By \Cref{prop:suppchi-invariance}, $w D \in\supchi$. Also, \Cref{prop:smart-propagates} and smartness imply that there exist $w_0, w_1 \in \mathcal W$ such that  $J_{D} \propto w_0J_{D_0}$ and $J_{D_1}  \propto w_1 J_{wD}$; that is,
		\[J_{D}^{-1} w_0 J_{D_0}^{\vphantom{-1}} \propto \id \ \ \ \textrm{ and } \ \ \ \ J_{D_1}^{-1} w_1 J_{wD}^{\vphantom{-1}} \propto \id.\]
		The following diagram summarizes the dark subspaces and operations discussed so far. 
		\begin{center}
			\scalebox{0.6}{
				\tikzset{every picture/.style={line width=0.75pt}} 
				
				\begin{tikzpicture}[x=0.75pt,y=0.75pt,yscale=-1,xscale=1]
					
					\draw    (90.89,160.1) -- (185.89,160.1) ;
					\draw [shift={(187.89,160.1)}, rotate = 180] [fill={rgb, 255:red, 0; green, 0; blue, 0 }  ][line width=0.08]  [draw opacity=0] (12,-3) -- (0,0) -- (12,3) -- cycle    ;
					
					\draw    (47.45,133.55) -- (47.45,103.03) -- (47.45,53.82) ;
					\draw [shift={(47.45,51.82)}, rotate = 90] [fill={rgb, 255:red, 0; green, 0; blue, 0 }  ][line width=0.08]  [draw opacity=0] (12,-3) -- (0,0) -- (12,3) -- cycle    ;
					
					\draw    (277.89,160.1) -- (372.89,160.1) ;
					\draw [shift={(374.89,160.1)}, rotate = 180] [fill={rgb, 255:red, 0; green, 0; blue, 0 }  ][line width=0.08]  [draw opacity=0] (12,-3) -- (0,0) -- (12,3) -- cycle    ;
					
					\draw    (472.89,160.1) -- (567.89,160.1) ;
					\draw [shift={(569.89,160.1)}, rotate = 180] [fill={rgb, 255:red, 0; green, 0; blue, 0 }  ][line width=0.08]  [draw opacity=0] (12,-3) -- (0,0) -- (12,3) -- cycle    ;
					
					\draw    (92.89,33.1) -- (187.89,33.1) ;
					\draw [shift={(189.89,33.1)}, rotate = 180] [fill={rgb, 255:red, 0; green, 0; blue, 0 }  ][line width=0.08]  [draw opacity=0] (12,-3) -- (0,0) -- (12,3) -- cycle    ;
					
					\draw    (275.89,34.1) -- (370.89,34.1) ;
					\draw [shift={(372.89,34.1)}, rotate = 180] [fill={rgb, 255:red, 0; green, 0; blue, 0 }  ][line width=0.08]  [draw opacity=0] (12,-3) -- (0,0) -- (12,3) -- cycle    ;
					
					\draw    (474.89,33.1) -- (569.89,33.1) ;
					\draw [shift={(571.89,33.1)}, rotate = 180] [fill={rgb, 255:red, 0; green, 0; blue, 0 }  ][line width=0.08]  [draw opacity=0] (12,-3) -- (0,0) -- (12,3) -- cycle    ;
					
					\draw    (230.36,133.55) -- (230.36,103.03) -- (230.36,53.82) ;
					\draw [shift={(230.36,51.82)}, rotate = 90] [fill={rgb, 255:red, 0; green, 0; blue, 0 }  ][line width=0.08]  [draw opacity=0] (12,-3) -- (0,0) -- (12,3) -- cycle    ;
					
					\draw    (420.36,131.55) -- (420.36,101.03) -- (420.36,51.82) ;
					\draw [shift={(420.36,49.82)}, rotate = 90] [fill={rgb, 255:red, 0; green, 0; blue, 0 }  ][line width=0.08]  [draw opacity=0] (12,-3) -- (0,0) -- (12,3) -- cycle    ;
					
					\draw    (615.36,130.55) -- (615.36,100.03) -- (615.36,50.82) ;
					\draw [shift={(615.36,48.82)}, rotate = 90] [fill={rgb, 255:red, 0; green, 0; blue, 0 }  ][line width=0.08]  [draw opacity=0] (12,-3) -- (0,0) -- (12,3) -- cycle    ;

					\draw (30,15) node [anchor=north west][inner sep=0.75pt]  [font=\LARGE]  {$D_{0}$};
					
					\draw (220,15) node [anchor=north west][inner sep=0.75pt]  [font=\LARGE]  {$D$};
					
					\draw (400,15) node [anchor=north west][inner sep=0.75pt]  [font=\LARGE]  {$wD$};
					
					\draw (600,15) node [anchor=north west][inner sep=0.75pt]  [font=\LARGE]  {$D_{1}$};
					
					\draw (30,145) node [anchor=north west][inner sep=0.75pt]  [font=\LARGE]  {$\mathbb{C}^{r_{m}}$};
					
					\draw (220,145) node [anchor=north west][inner sep=0.75pt]  [font=\LARGE]  {$\mathbb{C}^{r_{m}}$};
					
					\draw (400,145) node [anchor=north west][inner sep=0.75pt]  [font=\LARGE]  {$\mathbb{C}^{r_{m}}$};
					
					\draw (600,145) node [anchor=north west][inner sep=0.75pt]  [font=\LARGE]  {$\mathbb{C}^{r_{m}}$};
					
					\draw (125,15) node [anchor=north west][inner sep=0.75pt]  [font=\Large]  {$w_{0}$};
					
					\draw (312,15) node [anchor=north west][inner sep=0.75pt]  [font=\Large]  {$w$};
					
					\draw (510,15) node [anchor=north west][inner sep=0.75pt]  [font=\Large]  {$w_{1}$};
					
					\draw (125,135) node [anchor=north west][inner sep=0.75pt]  [font=\Large]  {$\Id$};
					
					\draw (300,130) node [anchor=north west][inner sep=0.75pt]  [font=\Large]  {$u_{1} u_{0}^{-1}$};
					
					\draw (510,135) node [anchor=north west][inner sep=0.75pt]  [font=\Large]  {$\Id$};
					
					\draw (8,80) node [anchor=north west][inner sep=0.75pt]  [font=\Large]  {$J_{D_{0}}$};
					
					\draw (197,80) node [anchor=north west][inner sep=0.75pt]  [font=\Large]  {$J_{D}$};
					
					\draw (376,80) node [anchor=north west][inner sep=0.75pt]  [font=\Large]  {$J_{wD}$};
					
					\draw (578,80) node [anchor=north west][inner sep=0.75pt]  [font=\Large]  {$J_{D_{1}}$};
					
				\end{tikzpicture}
			}
		\end{center} 
		
		It follows that  $w_1 w w_0D_0 = D_1$, so from  \Cref{prop:invariance_pure} we get $w_1 w w_0 \cdot \hat y_0 \in F$. Moreover, by direct calculation we now show that $w_1 w w_0 \cdot \hat y_0 = \hat y_1$. Indeed, $w_0$ maps $D_0$ to $D$: 
		\[ w_0  \cdot  \hat y_0 =    w_0 J_{D_0}^{\vphantom{-1}} u_0 \cdot  \hat x  =   J_{D}^{\vphantom{-1}}   J_{D}^{-1} w_0 J_{D_0}^{\vphantom{-1}} u_0 \cdot  \hat x =   J_{D}^{\vphantom{-1}}  u_0 \cdot  \hat x.\]
		Next, under the action of $w$ on $D$, the unitary $u_{w,D}^{\vphantom{-1}} = u_1^{\vphantom{-1}} u_0^{-1}$ arises on $\pcr$: 
		\[ w J_{D}^{\vphantom{-1}}  u_0 \cdot  \hat x =  J_{wD}^{\vphantom{-1}}   J_{wD}^{-1} w  J_{D}^{\vphantom{-1}} u_0 \cdot  \hat x =  J_{wD}^{\vphantom{-1}} u_{w,D}   u_0 \cdot  \hat x =  J_{wD}^{\vphantom{-1}}  u_1 \cdot  \hat x.\]
		Finally, $w_1$ maps $wD$ to the desired dark subspace $D_1$ while inducing $\Id$ on the reference space:
		\[ w_1 J_{wD}^{\vphantom{-1}}  u_1 \cdot  \hat x = J_{D_1}^{\vphantom{-1}}   J_{D_1}^{-1} w_1 J_{wD}^{\vphantom{-1}}  u_1 \cdot  \hat x =  J_{D_1}^{\vphantom{-1}} u_1 \cdot  \hat x = \hat y_1,\]
		as claimed.

		It remains to consider the case $u_1u_0^{-1}\notin R_{\mathcal J}$. Let
		$(u_n)_n\subset R_{\mathcal J}$ be such that
		$u_n\to  u_1u_0^{-1}$.
		Applying the preceding construction to each $u_n$, we obtain
		$w_n\in\mathcal W$ such that 
		\[
		w_nD_0=D_1,
		\quad \text{and} \quad 
		J_{D_1}^{-1}w_nJ_{D_0}^{\vphantom{-1}}\propto u_n.
		\]
		By \Cref{lem:new_lemma_fix}, there exists $w\in\mathcal W$ such that
		\[
		J_{D_1}^{-1}wJ_{D_0}^{\vphantom{-1}}\propto u_1u_0^{-1}.
		\]
		Therefore,
		\[
		w\cdot\hat y_0
		=
		wJ_{D_0}u_0\cdot\hat x
		=
		J_{D_1}u_1\cdot\hat x
		=
		\hat y_1.
		\] 
		By \Cref{prop:invariance_pure} we have $\hat y_1\in F$,  
		and the proof of the first implication is concluded. 
		
		\medskip
		\noindent 
		$(\Rightarrow)$ Fix a $\Pi$-minimal set $F \subset \pcd$. We show there exists $\hat x\in\pcr$ such that $\mathrm{supp}\, \nu_{\hat x}\subset F$. Then, by minimality, \Cref{lem:suppisinv} gives  $F=\mathrm{supp}\, \nu_{\hat x}$. We need to investigate how the minimality of $F$ is expressed in the dark subspaces and, more precisely, how it is expressed in $\pcr$. To this end, for $D \in \mathcal D_m$ we consider \[X_D = J_D^{-1}(F \cap \P(D))\subset \pcr.\] We break the proof into three main steps. 
		
		\medskip 
		\noindent
		\textbf{First step.} We show that  there exists a non-empty subset  $X \subset \pcr$ such that  $ X_D = X$ for every $D \in   \supchi$.

		We start by showing that if $D_0 \in   \supp{\chi_{{\rm inv}}}$, then  $X_{D_0} \neq \varnothing$. Indeed, since  $F$ is a non-empty subset of $\bigcup_{D \in \mathcal D_m}\P(D)$, see  \Cref{prop:invariantindark}, we obtain  \[ \bigcup_{D\in \mathcal D_m} J_DX_D = F \cap \bigcup_{D\in \mathcal D_m} \P(D) =  F \neq \varnothing,\] and so there exists $  D \in \mathcal D_m$ such that $X_{{  D}} \neq \varnothing$. By \Cref{prop:suppchi}, there   exists $w \in {\mathcal W} $ such that $D_0 = w   D$. Fix $\hat x \in X_{  D}$. Then, $J_{D} \cdot \hat x \in F \cap   \P(D)$ and  $\operatorname{tr}(w\pi_Dw^*)>0$, i.e., $\|wJ_Dx\|>0$, which
		via \Cref{prop:invariance_pure} implies that $w   J_{D} \cdot \hat x \in F \cap \P(D_0)$. That is, $J_{D_0}^{-1} w J_{{D}}^{\vphantom{-1}} \cdot \hat x \in X_{D_0}$; hence $X_{D_0} \neq \varnothing$, as claimed.
		
		Next, we observe that if $D \in \supp{\chi_{{\rm inv}}}$ and $w \in \mathcal W$ satisfies  $w  D  \neq \{0\}$, then 
		$u_{w,D} X_{D} \subset X_{wD}$.
		Indeed, let $\hat x \in X_{D}$. We have $ J_{ D} \cdot \hat x \in F \cap \P(D)$ and
		$wD \neq \{0\}$,  so again via \Cref{prop:invariance_pure} we get $w J_{D} \cdot \hat x \in F \cap \P(wD)$; thus, \[ J_{wD}^{-1}w J_{D}^{\vphantom{-1}} \cdot \hat x =  u_{w,D} \cdot \hat x  \in X_{wD},\] as claimed.

		We are now ready to conclude the first step.  Fix $D, D' \in \supp{\chi_{{\rm inv}}}$. \Cref{prop:smart-propagates}  guarantees that   there exists  $w'  \in \mathcal W$ such that $J_D \propto w'J_{D'}$, i.e., $u_{w', D'} = \id$. This  implies  $X_{D'} \subset X_D$.
		On the other hand, by \Cref{prop:suppchi}, there exists $ w    \in \mathcal W$ such that $w  D = D'$, so
		$u_{w,D}  X_D \subset  X_{D'} $.
		Therefore, \[u_{w,D}   X_D \subset  X_{D'} \subset X_D \] and    \Cref{lem:unitaryinclusion}(ii)   gives  $X_D =  X_{D'}$. 
		For the remainder of the proof we denote this common set by $X$.
		
		\medskip 
		\noindent
		\textbf{Second step.} It consists in showing $X=uX$ for every $u \in G_\J$.
		
		Since $X$ is closed as the continuous pre-image of the intersection of two closed sets, by \Cref{prop:smartgroup}, it suffices to consider $u \in G_\J \cap R_\J$. Let  $u = u_{w,D}$ for some $D \in \supchi$ and $w \in \W$ such that $\operatorname{tr}(w\pi_Dw^*)>0$. Since $wD \in \supchi$, from the previous step we immediately get the inclusion  $u_{w,D}  X \subset X$, and so   \Cref{lem:unitaryinclusion}(ii) gives the desired equality $u_{w,D}  X = X$. Hence,  $X=uX$ for all $u \in G_\J \cap R_\J$, as desired. Note that, as a consequence, $[\hat x]_{G_\J} \subset X$ for any $\hat x \in X$.
		
		\medskip 
		\noindent
		\textbf{Third step, conclusion.} Fix any $\hat x \in X$ and  observe that \[F = \bigcup_{D\in \mathcal D_m} J_D X_D \  \supset    \bigcup_{D\in \supp{\chi_{{\rm inv}}}} \hspace{-5mm} J_D X
		\ \supset \bigcup_{D\in \supp{\chi_{{\rm inv}}}} \hspace{-5mm} J_D  [\hat  x]_{G_\J} = \supp \nu_{\hat x}.\]
		Then, $\supp\nu_{\hat x}\subset F$ and the theorem holds.
	\end{proof}

	\subsection{Equicontinuity of the kernel}\label{subs:equicont}

	In what follows, we show that $\Pi$ is equicontinuous, i.e.,  that the family  $(\Pi^nf)_{n\in \nnone}$ is (uniformly) equicontinuous on $\pcd$ for every $f \in  \mathcal C(\pcd)$.
	Our approach consists in showing that this is the case when $f$ is a Lipschitz function and then using a density argument. 
	
	On $\pcd$ we use the metric   
	\[\delta(\hat x, \hat y) = \sqrt{1 - \vert\langle x,y\rangle \vert^2},\] which coincides with the gap metric 
	$d_{\rm G}(\cc x,\cc y) =\|\pi_{\hat x} - \pi_{\hat y}\|  = \|x \wedge y\| $ (recall that $x,y$ stand for unit norm representatives of $\hat x, \hat y$ respectively). Also, recall the notation  $\V_n = \{v_n\cdots v_1 \:|\: (v_1, \ldots, v_n)\in \supp \mutensorn \}$ and observe that it can be rewritten as $\V_n = \{W_n(\omega) \:|\: \omega \in \Omega,\: \pi_n(\omega) \in \supp\mutensorn \}$. We repeatedly make use of the stochasticity condition, which in the current context is expressed as  $\int_{\Omega}  W_n^*W_n      \d\mutensorn = \id_{\cc^d}$. 
	
	The following three lemmas establish the equicontinuity property for Lipschitz functions. Similar estimates are derived in  \cite[Lemma 4.5]{BHP24} under the assumption of purification (i.e., $r_m =1$). Since we work without this assumption, some bounds need to be adapted.
	
	\begin{lemma}\label{lem:eqM}
		Let  $f \in \mathcal C(\pcd)$ be  Lipschitz. Let  $\hat x, \hat y \in \pcd$ and $n \in \nnone$. Then
		\[\int\limits_{W_n \in A} \left| f(W_n \cdot \hat x) - f(W_n \cdot \hat y)\right| \cdot  \|W_n x\|^2 \d\mutensorn
		\leq  L  \delta(\hat x, \hat y),\]
		where $A = \{w \in \V_n \:|\: \|w y\|\geq  \|w x\|\}$ and   $L$ is the Lipschitz constant of $f$.
	\end{lemma}
	\begin{proof}
		Fix $\hat x,\hat y\in \pcd$. If $\hat x=\hat y$, there is nothing to prove. Assume $\hat x\neq \hat y$. Let $y^\perp=y -  \langle x,y\rangle x$ and denote its normalization by $\tilde{y}^\perp$. Then, $\|y^\perp\|=\delta(\hat x,\hat y)$. Using the linearity and $x \wedge x  =0$, for any $w \in A$ we have 
		\begin{align*}
			\|  w  x \wedge  w    y \|=\|wx\wedge w y^\perp\|=\delta(\hat x,\hat y)\| w x\wedge w\tilde y^\perp\|.
		\end{align*}
		Assume $ w x \neq 0$. Then  $\|w y\| \geq \|wx\| >0$,    so both $w \cdot \hat x$ and $w \cdot \hat y$ are well defined. It follows that
		\begin{align*}
			\delta(w \cdot \hat x, w \cdot \hat y)
			=   \frac{ \|  w   x\wedge w y\|}{ \|w  x\|\|w  y\| } 
			\leq
			\frac{ \|  w   x\wedge w \tilde y^\perp \|}{ \|w  x\|^2 }\delta({\hat x, \hat y}).
		\end{align*}
		Thus, using the Lipschitz property of $f$,
		\begin{align}\label{eq:equicont11}
			| f(w \cdot \hat x) - f(w \cdot \hat y)| \,\|w  x\|^2  \leq L \delta(w \cdot \hat x, w \cdot \hat y) \|w  x\|^2  \leq
			L \delta(\hat x, \hat y)     \|  w   x\wedge w \tilde y^\perp\|.
		\end{align}
		Then, Cauchy--Schwarz inequality and the stochasticity condition imply  
		\begin{align}  \nonumber
			\int_{\Omega}\|  W_n  x \wedge  W_n  y^\perp\| \d\mutensorn
			&  \leq  \int_{\Omega} \|  W_n   x\| \,\| W_n y^\perp\| \d\mutensorn
			\\  & \nonumber \leq
			\left(\ \int_{\Omega} \| W_n   x\|^2  \d\mutensorn\right)^{ \!\frac 12}   \left(\ \int_{\Omega} \|  W_n y^\perp \|^2  \d\mutensorn \right)^{\!   \frac 12}
			\\  &  =  \label{eq:equicont2222} 
			\bigg(\int_{\Omega} \langle x,W_n^*W_nx\rangle   \d\mutensorn \bigg)^{   \frac 12}  \bigg(\int_{\Omega} \langle \tilde y^\perp, W_n^*W_n\tilde y^\perp\rangle   \d\mutensorn   \bigg)^{   \frac 12} 
			=  1.
		\end{align}
		It remains to note that both $W_n \cdot \hat x$ and $W_n \cdot \hat y$ are almost surely well defined with respect to the measure $\|W_nx\|^2\d\mutensorn$ on $A$, so \Cref{eq:equicont11,eq:equicont2222} imply
		\begin{align*}\int\limits_{W_n \in A} \left| f(W_n \cdot \hat x) - f(W_n \cdot \hat y)\right|\, \|W_n x&\|^2 \d\mutensorn 
			\leq     L \delta(\hat x, \hat y),
		\end{align*}
		which concludes the proof.
	\end{proof}

	\begin{lemma}\label{lem:eqVn} Let $\hat x, \hat y \in \pcd$ and $n \in \nnone$. Then
		\begin{align*} \int_{\Omega}  \big|\|W_n x\|^2 - \|W_n y\|^2 \big|   \d\mutensorn
			\leq   2 \delta(\hat x, \hat y) .\end{align*}
	\end{lemma}
	\begin{proof} 
		Note that, for all $w \in \V_n$, 
		\[ \left| \|w x\|^2 - \|w y\|^2 \right|     =  \left| \tr(w^* w (\pi_{\hat x} - \pi_{\hat y})) \right|   \leq  \tr(w^* w  |  \pi_{\hat x} - \pi_{\hat y}|).\]  
		Then, using the stochasticity condition, 
		\begin{align*} \int_{\Omega}   \big|\|W_n x\|^2 - \|W_n y\|^2 \big|   \d\mutensorn 
			\leq \tr\left( \int_{\Omega}  W_n^*W_n      \d\mutensorn |  \pi_{\hat x} - \pi_{\hat y} | \right)  =   \|\pi_{\hat x} - \pi_{\hat y}\|_{\rm tr}.\end{align*}
		Since $\operatorname{rank}(\pi_{\hat x} - \pi_{\hat y}) \leq 2$,     the inequality   $\|A\|_{\rm tr} \leq \operatorname{rank}(A) \|A\|$  concludes the proof. 
	\end{proof}

	\begin{lemma}\label{lem:eqLip}
		Let  $f \in \mathcal C(\pcd)$ be  Lipschitz. Let $\hat x, \hat y \in \pcd$ and $n \in \nnone$. Then
		\[|\Pi^n f(\hat x) - \Pi^n f(\hat y) |\leq C \delta(\hat x, \hat y),\]
		where $C$ is a positive constant that depends only on   $f$.
	\end{lemma}
	\begin{proof}
		Let $A = \{w \in \V_n \:|\: \|w y\|\geq \|w x\|\}$   and consider the disjoint decomposition  $\V_n = A \cup A^c$. Then
		\begin{align*}
			|\Pi^nf(\hat x)-\Pi^nf(\hat y)|  \leq
			\int\limits_{W_n \in A^{\phantom{c}}} & \left|f(W_n \cdot \hat x)\|W_n x\|^2 - f(W_n \cdot \hat y)\|W_n y\|^2 \right|  \d\mutensorn
			\\
			+    \int\limits_{W_n \in A^c} & \left|f(W_n \cdot \hat x)\|W_n x\|^2 - f(W_n \cdot \hat y)\|W_n y\|^2 \right|  \d\mutensorn.
		\end{align*}
		The first term on the right-hand side above can be bounded as follows
		\begin{align*} \nonumber
			\int\limits_{W_n \in A}  \big|f( &W_n  \cdot  \hat x)\|W_n x\|^2 -    f( W_n \cdot \hat y) \|W_n y\|^2 \big|  \d\mutensorn 
			\\ &  \leq
			\int\limits_{W_n \in A} \left| f(W_n \cdot \hat x) - f(W_n\cdot \hat  y)\right| \cdot  \|W_n x\|^2  \d\mutensorn      + \int\limits_{W_n \in A}\left|f(W_n\cdot \hat  y)\right| \cdot  \big|\| W_n x\|^2 - \|W_n y\|^2 \big|  \d\mutensorn
			\\ &  \leq
			\int\limits_{W_n \in A} \left| f(W_n \cdot \hat x) - f(W_n\cdot \hat  y)\right| \cdot  \|W_n x\|^2  \d\mutensorn       + \|f\|_\infty \int\limits_{W_n \in A}  \big|\|W_n x\|^2 - \|W_n y\|^2 \big|  \d\mutensorn.
		\end{align*}
		Bounding  the other term in an analogous manner, we obtain 
		\begin{align}
			|\Pi^nf(\hat x)-\Pi^nf(\hat y)|  \leq &
			\int\limits_{W_n \in A^{\phantom{c}}}  \left| f(W_n \cdot \hat x) - f(W_n\cdot \hat  y)\right| \cdot  \|W_n x\|^2  \d\mutensorn
			\label{eq:equicont:sum111}  \\  + & \int\limits_{W_n \in A^c} \left| f(W_n \cdot \hat y) - f(W_n\cdot \hat  x)\right| \cdot  \|W_n y\|^2  \d\mutensorn
			\label{eq:equicont:sum222} \\  + & \,\|f\|_\infty  \int\limits_{W_n \in \V_n}  \big|\|W_n x\|^2 - \|W_n y\|^2 \big|  \d\mutensorn. \label{eq:equicont:sum333}
		\end{align}
		
		Let $L$ denote  the Lipschitz constant of $f$.  Then \Cref{lem:eqM} implies each of \Cref{eq:equicont:sum111} and \Cref{eq:equicont:sum222} is bounded by $L  \delta(\hat x, \hat y)$. Then, \Cref{lem:eqVn} implies \Cref{eq:equicont:sum333} is upper bounded by $2\, \|f\|_\infty \,\delta(\hat x, \hat y)$.
		Therefore,
		\[|\Pi^nf(\hat x)-\Pi^nf(\hat y)|  \leq
		2 ( L    +   \|f\|_\infty )\, \delta(\hat x, \hat y),\] 
		which concludes the proof.
	\end{proof}
	
	\begin{theorem}\label{thm:Piequicont}
		Let $f \in \mathcal C(\pcd)$. Then $(\Pi^n f)_{n \in \nnone}$ is equicontinuous.
	\end{theorem}
	\begin{proof}
		
		Fix   $ \varepsilon > 0$.
		Let $f_L$ be a Lipschitz function satisfying $\|f-f_L\|_\infty < \varepsilon/3.$
		Let $\hat x, \hat y \in \pcd$ and $n \in \nn$. We have
		\begin{align*}
			|\Pi^n f(\hat x) - \Pi^n f(\hat y)|
			&  \leq
			|\Pi^n (f-f_L)(\hat x)|  + |\Pi^n (f -f_L)(\hat y)| + |\Pi^n f_L (\hat x) - \Pi^n f_L(\hat y)|.
		\end{align*}
		Since Markov kernels are contractions on the Banach space of bounded continuous functions,
		$\|\Pi\|_\infty = 1$,
		so
		\[|\Pi^n (f-f_L)(\hat x)|  \leq \|\Pi^n (f-f_L)\|_\infty \leq  \|f-f_L\|_\infty < \varepsilon/3,\]    and, analogously,  $|\Pi^n (f-f_L)(\hat y)| < \varepsilon/3  $.
		\Cref{lem:eqLip} implies the last term is bounded as
		\[|\Pi^n f_L (\hat x) - \Pi^n f_L(\hat y)| \leq C \delta(\hat x, \hat y),\]
		where the constant $C$ depends only on  $f_L$.
		Therefore, for   $\delta=\varepsilon/(3 C)$,
		\[\delta(\hat x, \hat y) <  \delta \ \ \Longrightarrow\ \
		|\Pi^n f(\hat x) - \Pi^n f(\hat y)| < \varepsilon\]
		for any $n \in \nn$ and the theorem holds.
	\end{proof}
	
	\subsection{Characterization of ergodic measures}\label{subs:ergo}
	
	Finally, we are in a position to prove the main result of this section, which is the content of \Cref{thm:intro-ergodic}. Recall that for $\hat x \in \pcr$ and a family of isometries $\J$ we denote $\nu_{\hat x,\J} = (\Psi_\J)_{\star}(\chi_{\rm inv} \otimes m_{\hat x, \J}  )$ and $\Upsilon_\J = \{\nu_{\hat x,\J}\:|\: \hat x \in \pcr\}$.

	\begin{theorem}\label{thm:charPiinvmeas} 	Let $\J$ be a minimal family. Then $\Upsilon_\J$ is the set of   $\Pi$-ergodic measures.  
	\end{theorem}
	\begin{proof}
		We  show that a probability measure $\nu$ on $\pcd$ is $\Pi$-ergodic if and only if $\nu   = \nu_{\hat x, \J}$ for some $\hat x \in \pcr$. 
		The~result follows from Raugi's \Cref{thm:raugi}, which can be applied thanks to the equicontinuity of $\Pi$ (\Cref{thm:Piequicont}), and relies on the characterization of $\Pi$-minimal subsets of $\pcd$ (\Cref{thm:piminimalsets}).  
		
		\begin{itemize}\itemsep=1mm   
			\item[($\Rightarrow$)]
			Let $\nu$ be $\Pi$-ergodic. By \Cref{thm:raugi}, $\supp \nu$ is a $\Pi$-minimal set that carries a unique $\Pi$-invariant probability measure. By \Cref{thm:piminimalsets}, there exists $\hat x \in \pcr$ such that $\supp \nu = \supp \nu_{\hat x, \J}$, and \Cref{prop:productmeas} implies that  $\nu_{\hat x, \J}$ is $\Pi$-invariant. Therefore, $\nu = \nu_{\hat x, \J}$, as desired.
			
			\item[($\Leftarrow$)]
			Let $\hat x \in \pcr$. By \Cref{thm:piminimalsets}, $\supp \nu_{\hat x, \J}$ is a $\Pi$-minimal subset of $\pcd$. Again by \Cref{thm:raugi}, this set carries a unique $\Pi$-invariant probability measure, and this measure is $\nu_{\hat x, \J}$, see   \Cref{prop:productmeas}. As a consequence, $\nu_{\hat x, \J}$ is uniquely invariant on its support, hence it is ergodic on $\pcd$, as desired. \qedhere
		\end{itemize}   
	\end{proof}

	From the characterization of $\Pi$-ergodic measures we can immediately deduce the following corollary.

	\begin{corollary}\label{cor:Piinv}Let  $\J$ be a minimal family.
		A   probability measure   $\nu$ on $\pcd$ is  $\Pi$-invariant if and only if 
		\[\nu =   \int_{ \P(\cc^{r_m})} \nu_{\hat x,\J}  \   \mathrm{d}\lambda(\hat x)    = (\Psi_{\J})_\star\left(\chi_{{\rm inv}} \otimes \int_{ \P(\cc^{r_m})} m_{\hat x,\J}  \   \mathrm{d}\lambda(\hat x)\right) \]   
		for some probability measure $\lambda$ on $ \P(\cc^{r_m})$.  
	\end{corollary}
	
	In what follows, we denote by $G_{\rm min}$ the subgroup of $\mathcal{SU}(r_m)$  induced by some minimal family. Let us recall that this group is unique up to unitary equivalence.
	\Cref{cor:Piinv} implies that  $\Pi$ admits a unique invariant probability measure if and only if  $[\hat x]_{G_{{\rm min}}} = \pcr$ for some $\hat x \in \pcr$, which is equivalent to the transitivity of the action of $G_{{\rm min}}$ on $\pcr$.

	By definition, $\mathcal{SU}(r_m)$ acts transitively on $\pcr$. There may exist proper subgroups of $\mathcal{SU}(r_m)$ that act transitively on $\pcr$. 
	Since $\pcr$
	is connected, $G_{\min}$ acts transitively on
	$\pcr$ if and only if its identity component
	$G_{\min}^{\circ}$ does \cite[Ch.~1, \S5, Proposition~7]{Onishchik1994}, see also \cite[Proposition~4.10]{GalazGarciaZarei2018}.
	It was shown by Onishchik \cite[Theorem~6]{Oni}
	(see also \cite{Besse,Shankar}) that, up to unitary equivalence, the
	only proper  connected subgroup of  $\mathcal{SU}(r_m)$ acting transitively
	on $\pcr$ is the symplectic group
	${\rm Sp}(r_m/2)$, known also as the hyperunitary group, provided
	that $r_m$ is even.  
	Let us recall that ${\rm Sp}(r_m/2)$    consists of $r_m \times r_m$ special unitary complex matrices $U$ satisfying $U^TJU=J$, where   
	\[J = \begin{bsmallmatrix}
		0 & \id_{\cc^{r_m/2}} \\ 
		- \id_{\cc^{r_m/2}} & 0    
	\end{bsmallmatrix}\] 
	The following corollary is \Cref{thm:introuniq}. 
	\begin{corollary}\label{cor:uniquePiinvmeas}
		The measure $\nu_{\rm unif}$ is the unique $\Pi$-invariant probability measure on $\pcd$ if and only if   
		\begin{itemize}
			\item     $G_{\rm min} = \mathcal{SU}(r_m)$  if $r_m$ is odd,
			\item    $G_{\rm min} = \mathcal{SU}(r_m)$ or $G^{\circ}_{\rm min} = Q{\rm Sp}(r_m/2)Q^*$ for some $Q \in \mathcal U(r_m)$ if $r_m$ is even,
			where  $G^{\circ}_{\rm min}$ denotes the identity component of $G_{\min}$.
		\end{itemize} 
	\end{corollary}

	\begin{remark}
		In principle, it is possible to obtain $\Upsilon_\J$ equal to the set of all $\Pi$-ergodic measures even if the family $\J= \{J_D\}_{D \in \mathcal D_m}$ used for the construction is not minimal. However, if  the trajectories on $\pcr$ allow us to distinguish between    $G_\J$ and  $G_{\J_{\rm min}}$, where $\J_{\rm min}$  is a minimal family such that $G_{\J_{\rm min}} \subset G_\J$, then $\Upsilon_\J$ contains a non-ergodic measure. More precisely, if   $\hat x_0 \in \pcr$ is such that $[\hat 
		x_0]_{G_{\J_{\rm min}}} \subsetneq [\hat  x_0]_{G_\J}$, then $\nu_{\hat  x_0 , \J} $ is not $\Pi$-ergodic.  
		To see this, consider
		\[ \nu = \int_{G_\J} \nu_{u \cdot \hat x_0,\, \J_{\rm min}} \d \mathrm{H}_{G_\J}(u).\]
		By direct calculation we can check that $\nu   = \nu_{\hat x_0, \J}$. Indeed, using the invariance of Haar measure, for any $f \in \mathcal C(\pcd)$ we obtain  
		\begin{align*}
			\int_{\pcd} f(\hat y) \d  \nu (\hat y)  
			& = \int_{\D} \int_{G_{\J_{\rm min}}}  \int_{G_\J} f( J_D \tilde u  u \cdot \hat  x_0) \d \mathrm{H}_{G_\J}( u) \d \mathrm{H}_{G_{\J_{\rm min}}}( \tilde u) \d \chi_{\rm inv}(D)
			\\ & = \int_{\D} \int_{G_\J}  f(J_D    u \cdot \hat  x_0) \d \mathrm{H}_{G_\J}( u)   \d \chi_{\rm inv}(D)
			\\ & =   \int_{\pcd} f(\hat y) \d \nu_{\hat x_0,\J} (\hat y). 
		\end{align*} 
		Therefore, $\nu_{\hat x_0,\J}$ can be non-trivially  decomposed into a convex combination of $\Pi$-ergodic measures of the form $\nu_{u \cdot \hat x_0,\J_{\rm min}}$, which means that $\nu_{\hat x_0,\J}$ is not $\Pi$-ergodic, as claimed.  
	\end{remark}

	\section{Examples}\label{sec:examples}
	\addtocontents{toc}{\protect\setcounter{tocdepth}{1}}

	\begin{remark}
		The projective space $\pc{2}$ can be represented on the unit sphere of $\rr^3$ using the mapping  \[\hat x\mapsto \big(\operatorname{tr}(\pi_{\hat x}\sigma_x),\operatorname{tr}(\pi_{\hat x}\sigma_y),\operatorname{tr}(\pi_{\hat x}\sigma_z)\big)\] with the three Pauli matrices
		\[\sigma_x=\begin{pmatrix}
			0&1\\1&0
		\end{pmatrix},\quad \sigma_y=\begin{pmatrix}
			0&-\i\\ {\i} &0
		\end{pmatrix}\quad\mbox{and}\quad \sigma_z=\begin{pmatrix}
			1&0\\0&-1
		\end{pmatrix}.\]
		In this context, the unit sphere in $\rr^3$ is called the Bloch sphere. We will use it to illustrate our examples.
	\end{remark}

	\subsubsection*{Example 1. (Two disjoint Bloch spheres)}\label{example1}   
	\noindent
	Let us fix the unitaries  $u_1, u_2, u_3, u_4 \in \mathcal{SU}(2)$ and consider the two matrices 
	\[v_1 =   \begin{bmatrix}
		0 & \sqrt{\frac{1}{ 3}} u_1  \\
		\sqrt{\frac{1}{ 4}} u_2 &0
	\end{bmatrix} \ \ \textrm{and} \ \ v_2 =  \begin{bmatrix}
		0 & \sqrt{\frac{2}{ 3}} u_3  \\
		\sqrt{\frac{3}{ 4}} u_4 & 0
	\end{bmatrix}.\]
	
	\noindent
	\textbf{(1) Irreducibility.}  We show that $\mu = \delta_{v_1} + \delta_{v_2}$ is irreducible  if   $u_1u_2$ and $u_3u_4$  do not share an eigenbasis.   
	To this end, observe that $\id/4$ is a fixed point of the quantum channel $\phi(\rho) = v_1\rho v_1^* + v_2\rho v_2^*$, so from \cite[Theorem~6.13]{Wolf} we know that  $\mu$ is irreducible iff the only linear operators on $\cc^4$ that commute  with both Kraus operators $v_1$ and $v_2$ and also with their adjoints are the multiples of identity, i.e., 
	\[\big\{M \in \mathcal M_4\: | \:  [M,v_1]=[M, v_1^*] = [M,v_2]=[M, v_2^*] = 0\big\} =\big\{\alpha \id_4 \:|\: \alpha \in \cc\big\}.\]
	To see that this condition is satisfied if  $u_1u_2$ and $u_3u_4$  do  not share an eigenbasis, consider $M = \begin{bsmallmatrix}
		A & B \\ C & D
	\end{bsmallmatrix}$ and assume that $M$ commutes with $v_1$, $v_2$, $v_1^*$, $v_2^*$. Then 
	\[\sqrt{\tfrac{1}{4} } u_2B  =   \sqrt{\tfrac{1}{3} }Cu_1\ \  \textrm{  and }\ \ \sqrt{\tfrac{1}{3} } Bu_1^*  =   \sqrt{\tfrac{1}{4} } u_2^*C,\]
	which implies that $u_2B = \sqrt{\tfrac{4}{3} }Cu_1 =   \sqrt{\tfrac{3}{4} }Cu_1$, thus also $B = C = 0$.
	Moreover,
	\begin{equation}\label{AD} 
		u_1D = Au_1  \ \ \textrm{ and }
		\ \    u_2A = Du_2 
	\end{equation}
	which gives 
	$ u_1u_2A = u_1Du_2 = Au_1u_2$, 
	that is
	$
	[A, u_1u_2]  =  0.$
	Analogously,  $[A, u_3u_4]   = 0$.  
	
	Let us now assume that $u_1u_2$ and $u_3u_4$ do not have a common eigenbasis.
	Then both $u_1u_2$ and $u_3u_4$ have distinct eigenvalues (i.e.,  neither $u_1u_2$ nor $u_3u_4$ is a multiple of identity). From the commutativity relations
	$[A, u_1u_2]  =  [A, u_3u_4]   = 0$, it follows that  $A$ shares an eigenbasis with both  $u_1u_2$ and $u_3u_4$, which implies that $A = \alpha \id_2$ for some $\alpha \in \cc$. From \Cref{AD} it then follows that $D = \alpha \id_2$,
	thus also
	$M = \alpha \id_4$, as desired.

	\smallskip 
	
	For the remainder of this example we assume that $u_1, u_2, u_3, u_4$ are so chosen that the irreducibility of $\mu$ holds. 
	
	\medskip
	
	\noindent   \textbf{(2) Dark subspaces.}   One can immediately identify the following two dark subspaces:
	\[D_a = \left\{ 
	[x,y,0,0]^T
	\:|\: x,y \in \cc\right\}\ \  \textrm{and} \ \  D_b = \left\{ 
	[0,0,x,y]^T
	\:|\: x,y \in \cc\right\}.\]
	We see that $D_a, D_b$ are orthogonal, so in particular they intersect trivially. The projective dark subspaces $\P(D_a), \P(D_b)$ are isomorphic to $\pc{2}$.

	To see that $D_a, D_b$ are maximal dark subspaces, assume that $D \subset \cc^4$ is a dark subspace of dimension $3$. Then for any $A\in \operatorname{span}\{ v_1^*v_1, v_2^*v_2\}=\operatorname{span}\{\pi_{D_a},\pi_{D_b}\}$ we have  $\pi_DA\pi_D=\frac13\operatorname{tr}(\pi_DA)\pi_D$. Taking $A=\pi_{D_a}$ or $A=\pi_{D_b}$, it follows from $\rank(\pi_D\pi_{D_{a/b}}\pi_D)\leq 2$ and $\rank \pi_D=3$ that $\operatorname{tr}(\pi_D\pi_{D_{a/b}})=0$, hence $\operatorname{tr}(\pi_D)=0$, which contradicts $\pi_D\neq 0$, so $D$ is not dark.

	\medskip
	
	\noindent 
	\textbf{(3) Invariant measure for dark spaces.}
	{The Markov chain on $ \{D_a, D_b\}$ is deterministic: its transition matrix reads $ \begin{bsmallmatrix}0 & 1 \\ 1 & 0	\end{bsmallmatrix} $. Note that the period is $2$.
		Hence,  the  uniform measure  \[ \tfrac 12 \delta_{D_a} + \tfrac 12 \delta_{D_b}\] is $K$-invariant. 
		We have thus identified an invariant measure  on $\mathcal D_m$ and \Cref{thm:invmeasDark} implies it is unique. We can therefore restrict our attention to its support, i.e., to the subspaces $D_a$, $D_b$. However, in principle,   this system may admit additional maximal dark subspaces.
	}  
	
	\medskip
	
	\noindent  \textbf{(4) Minimal family.}      The family of isometries $\J_{\rm emb}$ consisting of natural embeddings
	\[J_a^{\rm emb} \colon \begin{bsmallmatrix}
		x \\ y
	\end{bsmallmatrix}
	\mapsto \begin{bsmallmatrix}
		x \\ y \\ 0 \\ 0
	\end{bsmallmatrix}
	\ \ \ \ \ \ \
	J_b^{\rm emb} \colon \begin{bsmallmatrix}
		x \\ y
	\end{bsmallmatrix}
	\mapsto \begin{bsmallmatrix}
		0 \\ 0 \\ x \\ y
	\end{bsmallmatrix}
	\]
	leads to the group generated by the four unitaries: $G_{\J_{\rm emb}} = \overline{ \langle u_1, u_2, u_3, u_4 \rangle}$. This group need not be minimal. 
	To obtain a minimal group, let us construct a $D_a-$smart family $\J = \{J_a, J_b\}$ such that $J_a = J_a^{\rm emb}$. Taking $J_b = 2 v_1 J_a =   J_b^{\rm emb} u_2$, we obtain 
	\[ u_{1, D_a}
	=  u_2^{-1}u_2 = \id ,\ \   u_{2, D_a}
	=    u_2^{-1}u_4 , \ \  u_{1, D_b}
	=    u_1u_2 , \ \   u_{2, D_b}
	=    u_3u_2,\]
	and so $G_{\J} = G_{\rm min} = \overline{ \langle u_2^{-1}u_4, u_1u_2, u_3u_2 \rangle}$.

	\medskip
	
	\noindent   \textbf{(5)}
	Let us take a closer look at a more concrete case. For $\theta\in \rr$ let $R_{x}(\theta) = \exp(\i\theta\sigma_x/2)$  and $R_{z}(\theta)   = \exp(\i\theta\sigma_z/2)$. On the Bloch sphere, these correspond to rotations of angle $\theta$ about the axes $x$ and $z$, respectively.  
	Set
	\[u_1 = \id ,\ u_2 = R_{x}(\theta_x) ,\   u_3  = R_{z}(\theta_z) ,\ u_4 =\id\]
	with arbitrary rotation angles $\theta_x, \theta_z \in (0, 2\pi)$. The family $\J_{\rm emb}$  of  canonical embeddings   is  now minimal (because it is smart)  and \[G_{\J_{\rm emb}} = G_{{\rm min}} = \overline{\langle R_{x}(\theta_x), R_{z}(\theta_z)  \rangle}.\] Let us recall that we do not adjust the phases explicitly but all unitary operators on the reference space are in fact special unitary matrices. 
	There are three qualitatively different cases:
	
	\begin{enumerate}[label={\bf (\alph*)}]\itemsep=1mm
		\item   If $ \theta_x$, $\theta_z \notin \pi\mathbb Q$, then, by a compactness argument, the minimal family $\J_{\rm emb}$ (thus also every other family) induces the full group: $G_{{\rm min}} = \mathcal {SU}(2)$, and so $\nu_{{\rm inv}} =  \frac 12 \mathrm{Unif}_{\P(D_a)} + \frac 12 \mathrm{Unif}_{\P(D_b)}$ is the unique $\Pi$-invariant measure on $\pcd$.
		
		\item \label{ex:15b} If $ \theta_x = \pi$ and $\theta_z \notin \pi\mathbb Q$, then $G_{{\rm min}} = \overline{ \langle \i\sigma_x, R_z(\theta_z) \rangle   } =
		\{
		R_z(t), \i\sigma_x R_z(t)\, | \, t \in \rr
		\}$.  
		Hence, there are infinitely many mutually singular $\Pi$-ergodic measures. Orbits in $\P(\cc^2)$ are in bijection with $[0,1/\sqrt{2}]$, where for a given $c\in[0,1/\sqrt{2}]$ the orbit is given by $\{\hat x_{t}\}_{t\in \rr}\cup\{\sigma_x\cdot \hat x_{t}\}_{t\in \rr}$ with $x_{t}=[c\mathrm{e}^{it},\sqrt{1-c^2}]^T$.
		
		On the other hand, a non-minimal family $\tJ  = \{J_a^{\rm emb}, J_b^{\rm emb}R_x(\psi_x)\}$ with $\psi_x \notin \pi\mathbb Q$ induces the full group: $G_{\tJ} = \mathcal {SU}(2)$; thus, the only $\Pi$-invariant measure that $\tJ$ can induce is $\nu_{\rm unif}$ (which is non-ergodic).

		\item   If $ \theta_x=\pi$ and $\theta_z \in \pi\mathbb{Q}$, then setting $q\in \nnone$ such that $q\theta_z\in \pi\mathbb Z$ and using $\mathcal{SU}(2)$ commutation relations, any product of $R_{x}(\theta_x)$ and $R_{z}(\theta_z)$ to the power $2q$ is proportional to the identity. 
		Thus $G_{\rm min}$ is finite. There are infinitely many mutually singular $\Pi$-ergodic measures and each of them is finitely supported.  Note that both $\W_{\rm fin}$ and $\supchi$ are finite; hence, all the sets considered in \Cref{prop:smartgroup} coincide: $ S^{\rm fin}_{\J_{\rm emb}}   =  R_{\J_{\rm emb}}  = G_{\rm min}$.
		
		\medskip 
		
		\noindent 
		In particular, if  we set $\theta_x = \theta_z = \pi$, i.e., 
		\[u_1 = \id ,\ u_2 = \i\sigma_x ,\   u_3  = \i\sigma_z ,\ u_4 =\id,\]  then
		$G_{\rm min} = \{\pm\id, \pm\i\sigma_x, \pm\i\sigma_y,  \pm\i\sigma_z\}$  and a generic $\Pi$-ergodic measure 
		is supported on 8 points, which can degenerate to 4 points if an eigenvector of $\sigma_x$ or $\sigma_z$ is in the support. See also  \Cref{fig:ex1pic1}.
	\end{enumerate}
		\noindent   \textbf{(6)}\label{ex:16}
	As in (5c), we set $u_1 = \id$, $u_2 = \mathrm{i} \sigma_x$, $u_3  = \mathrm{i}\sigma_z$, $u_4 =\id$, but  
	consider now the following (non-minimal) families:  \[\hspace{10mm} \tJ_1 = \big\{J_a^{\rm emb}, J^{\rm emb}_bR_z(\tfrac \pi 3)\big\}, \ \ \ \tJ_2 = \big\{J_a^{\rm emb}, J_b^{\rm emb}R_z(\tfrac \pi 2)\big\}, \    \ \ \tJ_3 = \big\{J_a^{\rm emb}, J_b^{\rm emb}R_z(\psi)\big\}\] with $\psi \notin \pi\mathbb Q$. They induce the following groups: 
	\begin{align*}
		G_{\tJ_1} & =  \langle \i\sigma_x, R_{z}(\tfrac \pi 3)\rangle \:=\big\{ {\pm}R_{z}(\tfrac {j \pi} 3 ),  \pm  \i\sigma_x R_{z}(\tfrac {j \pi} 3)    \:|\: j= 0, \ldots, 5\big\},
		\\ G_{\tJ_2} &=  \langle \i\sigma_x, R_{z}(\tfrac \pi 2)\rangle \, =\big\{ {\pm}R_{z}(\tfrac {j \pi} 2 ), \pm \i\sigma_x R_{z}(\tfrac {j \pi} 2)    \:|\: j= 0, \ldots, 3\big\},
		\\ G_{\tJ_3} &= \overline{\langle \i\sigma_x, R_{z}(\psi) \rangle} =\big\{R_{z}(t), \i\sigma_xR_{z}(t)  \:|\: t\in \rr\big\};
	\end{align*}
	see also \Cref{fig:ex1pic2} for a generic trajectory under each of these groups. Then: 
	\begin{enumerate}[label={\bf (\alph*)}]
		\item \label{ex:16a} Considering $G_{\tJ_1}$ and $G_{\tJ_2}$, we see that neither of them is a subgroup of the other. The minimal group $G_{\rm min} = \langle \i\sigma_x, \i\sigma_z \rangle $ is a subgroup of both.
		
		\item \label{ex:16b} The group $G_{\tJ_3}$ is uncountably infinite, which contrasts with $G_{\rm min}$ being finite.  
	\end{enumerate}
	This illustrates how the choice of a family of isometries can influence what groups on $\cc^{r_m}$, thus also what $\Pi$-invariant measures on $\pcd$, can be generated by the procedure described in \Cref{sec:proc}. 
	
	\begin{figure}[ h]
		\begin{center}
			\scalebox{0.7987}{

				\tikzset {_i1wczxb9v/.code = {\pgfsetadditionalshadetransform{ \pgftransformshift{\pgfpoint{89.1 bp } { -108.9 bp }  }  \pgftransformscale{1.32 }  }}}
				\pgfdeclareradialshading{_7jw66logx}{\pgfpoint{-72bp}{88bp}}{rgb(0bp)=(1,1,1);
					rgb(3.9186504908970425bp)=(1,1,1);
					rgb(25bp)=(0.67,0.64,0.55);
					rgb(25bp)=(0.78,0.75,0.66);
					rgb(400bp)=(0.78,0.75,0.66)}

				\tikzset {_8oz7q2gyn/.code = {\pgfsetadditionalshadetransform{ \pgftransformshift{\pgfpoint{89.1 bp } { -108.9 bp }  }  \pgftransformscale{1.32 }  }}}
				\pgfdeclareradialshading{_ykz2w2dfa}{\pgfpoint{-72bp}{88bp}}{rgb(0bp)=(1,1,1);
					rgb(3.9186504908970425bp)=(1,1,1);
					rgb(25bp)=(0.67,0.64,0.55);
					rgb(25bp)=(0.78,0.75,0.66);
					rgb(400bp)=(0.78,0.75,0.66)}

				\tikzset {_ngdxw703z/.code = {\pgfsetadditionalshadetransform{ \pgftransformshift{\pgfpoint{89.1 bp } { -108.9 bp }  }  \pgftransformscale{1.32 }  }}}
				\pgfdeclareradialshading{_3sec6s6o5}{\pgfpoint{-72bp}{88bp}}{rgb(0bp)=(1,1,1);
					rgb(3.9186504908970425bp)=(1,1,1);
					rgb(25bp)=(0.67,0.64,0.55);
					rgb(25bp)=(0.78,0.75,0.66);
					rgb(400bp)=(0.78,0.75,0.66)}
				\tikzset{every picture/.style={line width=0.75pt}} 
				
				\begin{tikzpicture}[x=0.75pt,y=0.75pt,yscale=-1,xscale=1]
					
					\path  [shading=_7jw66logx,_i1wczxb9v] (244.94,334.8) .. controls (244.94,286.5) and (284.1,247.35) .. (332.39,247.35) .. controls (380.69,247.35) and (419.84,286.5) .. (419.84,334.8) .. controls (419.84,383.1) and (380.69,422.25) .. (332.39,422.25) .. controls (284.1,422.25) and (244.94,383.1) .. (244.94,334.8) -- cycle ;  
					\draw   (244.94,334.8) .. controls (244.94,286.5) and (284.1,247.35) .. (332.39,247.35) .. controls (380.69,247.35) and (419.84,286.5) .. (419.84,334.8) .. controls (419.84,383.1) and (380.69,422.25) .. (332.39,422.25) .. controls (284.1,422.25) and (244.94,383.1) .. (244.94,334.8) -- cycle ;  
					
					\path  [shading=_ykz2w2dfa,_8oz7q2gyn] (83.3,106.78) .. controls (83.3,58.49) and (122.45,19.33) .. (170.75,19.33) .. controls (219.05,19.33) and (258.2,58.49) .. (258.2,106.78) .. controls (258.2,155.08) and (219.05,194.23) .. (170.75,194.23) .. controls (122.45,194.23) and (83.3,155.08) .. (83.3,106.78) -- cycle ; 
					\draw   (83.3,106.78) .. controls (83.3,58.49) and (122.45,19.33) .. (170.75,19.33) .. controls (219.05,19.33) and (258.2,58.49) .. (258.2,106.78) .. controls (258.2,155.08) and (219.05,194.23) .. (170.75,194.23) .. controls (122.45,194.23) and (83.3,155.08) .. (83.3,106.78) -- cycle ;

					\draw   (409.3,107.68) .. controls (409.3,59.39) and (448.45,20.23) .. (496.75,20.23) .. controls (545.05,20.23) and (584.2,59.39) .. (584.2,107.68) .. controls (584.2,155.98) and (545.05,195.13) .. (496.75,195.13) .. controls (448.45,195.13) and (409.3,155.98) .. (409.3,107.68) -- cycle ;
					
					\draw    (250.2,69.23) .. controls (306.92,21.47) and (360.66,21.23) .. (416.36,68.52) ;
					\draw [shift={(417.2,69.23)}, rotate = 220.6] [fill={rgb, 255:red, 0; green, 0; blue, 0 }  ][line width=0.08]  [draw opacity=0] (12,-3) -- (0,0) -- (12,3) -- cycle    ;
					
					\draw    (416.6,144.68) .. controls (357.9,194.43) and (306.71,196.22) .. (250.05,145.99) ;
					\draw [shift={(249.2,145.23)}, rotate = 41.82] [fill={rgb, 255:red, 0; green, 0; blue, 0 }  ][line width=0.08]  [draw opacity=0] (12,-3) -- (0,0) -- (12,3) -- cycle    ;
					
					\draw    (409.6,120.18) .. controls (360.1,156.81) and (309.23,158.21) .. (259.7,120.39) ;
					\draw [shift={(258.2,119.23)}, rotate = 37.95] [fill={rgb, 255:red, 0; green, 0; blue, 0 }  ][line width=0.08]  [draw opacity=0] (12,-3) -- (0,0) -- (12,3) -- cycle    ;
					
					\draw    (258.2,94.23) .. controls (308.69,58.59) and (356.24,57.26) .. (408.61,92.16) ;
					\draw [shift={(410.2,93.23)}, rotate = 214.19] [fill={rgb, 255:red, 0; green, 0; blue, 0 }  ][line width=0.08]  [draw opacity=0] (12,-3) -- (0,0) -- (12,3) -- cycle    ;
					
					\draw  [color={rgb, 255:red, 155; green, 155; blue, 155 }  ,draw opacity=1 ][dash pattern={on 4.5pt off 4.5pt}][line width=0.75]  (83.3,106.78) .. controls (83.3,97.12) and (122.45,89.28) .. (170.75,89.28) .. controls (219.05,89.28) and (258.2,97.12) .. (258.2,106.78) .. controls (258.2,116.45) and (219.05,124.28) .. (170.75,124.28) .. controls (122.45,124.28) and (83.3,116.45) .. (83.3,106.78) -- cycle ;
					
					\path  [shading=_3sec6s6o5,_ngdxw703z] (409.3,107.68) .. controls (409.3,59.39) and (448.45,20.23) .. (496.75,20.23) .. controls (545.05,20.23) and (584.2,59.39) .. (584.2,107.68) .. controls (584.2,155.98) and (545.05,195.13) .. (496.75,195.13) .. controls (448.45,195.13) and (409.3,155.98) .. (409.3,107.68) -- cycle ;  
					\draw   (409.3,107.68) .. controls (409.3,59.39) and (448.45,20.23) .. (496.75,20.23) .. controls (545.05,20.23) and (584.2,59.39) .. (584.2,107.68) .. controls (584.2,155.98) and (545.05,195.13) .. (496.75,195.13) .. controls (448.45,195.13) and (409.3,155.98) .. (409.3,107.68) -- cycle ;

					\draw  [color={rgb, 255:red, 155; green, 155; blue, 155 }  ,draw opacity=1 ][dash pattern={on 4.5pt off 4.5pt}][line width=0.75]  (409.3,107.68) .. controls (409.3,98.02) and (448.45,90.18) .. (496.75,90.18) .. controls (545.05,90.18) and (584.2,98.02) .. (584.2,107.68) .. controls (584.2,117.35) and (545.05,125.18) .. (496.75,125.18) .. controls (448.45,125.18) and (409.3,117.35) .. (409.3,107.68) -- cycle ;
					
					\draw [color={rgb, 255:red, 74; green, 74; blue, 74 }  ,draw opacity=1 ]   (284.89,242.11) -- (220.64,207.07) ;
					\draw [shift={(218.89,206.11)}, rotate = 28.61] [fill={rgb, 255:red, 74; green, 74; blue, 74 }  ,fill opacity=1 ][line width=0.08]  [draw opacity=0] (12,-3) -- (0,0) -- (12,3) -- cycle    ;
					
					\draw  [color={rgb, 255:red, 0; green, 0; blue, 0 }  ,draw opacity=1 ][fill={rgb, 255:red, 208; green, 2; blue, 27 }  ,fill opacity=1 ][line width=1.5]  (94.1,54.72) .. controls (94.1,51.49) and (96.72,48.87) .. (99.95,48.87) .. controls (103.18,48.87) and (105.8,51.49) .. (105.8,54.72) .. controls (105.8,57.95) and (103.18,60.57) .. (99.95,60.57) .. controls (96.72,60.57) and (94.1,57.95) .. (94.1,54.72) -- cycle ;
					
					\draw  [color={rgb, 255:red, 0; green, 0; blue, 0 }  ,draw opacity=1 ][fill={rgb, 255:red, 208; green, 2; blue, 27 }  ,fill opacity=1 ][line width=1.5]  (94.1,157.33) .. controls (94.1,154.1) and (96.72,151.48) .. (99.95,151.48) .. controls (103.18,151.48) and (105.8,154.1) .. (105.8,157.33) .. controls (105.8,160.56) and (103.18,163.17) .. (99.95,163.17) .. controls (96.72,163.17) and (94.1,160.56) .. (94.1,157.33) -- cycle ;
					
					\draw  [color={rgb, 255:red, 0; green, 0; blue, 0 }  ,draw opacity=1 ][fill={rgb, 255:red, 208; green, 2; blue, 27 }  ,fill opacity=1 ][line width=1.5]  (235.1,54.72) .. controls (235.1,51.49) and (237.72,48.87) .. (240.95,48.87) .. controls (244.18,48.87) and (246.8,51.49) .. (246.8,54.72) .. controls (246.8,57.95) and (244.18,60.57) .. (240.95,60.57) .. controls (237.72,60.57) and (235.1,57.95) .. (235.1,54.72) -- cycle ;
					
					\draw  [color={rgb, 255:red, 0; green, 0; blue, 0 }  ,draw opacity=1 ][fill={rgb, 255:red, 208; green, 2; blue, 27 }  ,fill opacity=1 ][line width=1.5]  (234.95,157.33) .. controls (234.95,154.1) and (237.57,151.48) .. (240.8,151.48) .. controls (244.03,151.48) and (246.65,154.1) .. (246.65,157.33) .. controls (246.65,160.56) and (244.03,163.17) .. (240.8,163.17) .. controls (237.57,163.17) and (234.95,160.56) .. (234.95,157.33) -- cycle ;
					
					\draw  [color={rgb, 255:red, 155; green, 155; blue, 155 }  ,draw opacity=1 ][dash pattern={on 4.5pt off 4.5pt}][line width=0.75]  (243.81,334.8) .. controls (243.81,325.01) and (283.47,317.07) .. (332.39,317.07) .. controls (381.32,317.07) and (420.98,325.01) .. (420.98,334.8) .. controls (420.98,344.59) and (381.32,352.53) .. (332.39,352.53) .. controls (283.47,352.53) and (243.81,344.59) .. (243.81,334.8) -- cycle ;
					
					\draw  [color={rgb, 255:red, 0; green, 0; blue, 0 }  ,draw opacity=1 ][fill={rgb, 255:red, 208; green, 2; blue, 27 }  ,fill opacity=1 ][line width=1.5]  (422.1,58.72) .. controls (422.1,55.49) and (424.72,52.87) .. (427.95,52.87) .. controls (431.18,52.87) and (433.8,55.49) .. (433.8,58.72) .. controls (433.8,61.95) and (431.18,64.57) .. (427.95,64.57) .. controls (424.72,64.57) and (422.1,61.95) .. (422.1,58.72) -- cycle ;
					
					\draw  [color={rgb, 255:red, 0; green, 0; blue, 0 }  ,draw opacity=1 ][fill={rgb, 255:red, 208; green, 2; blue, 27 }  ,fill opacity=1 ][line width=1.5]  (422.1,161.33) .. controls (422.1,158.1) and (424.72,155.48) .. (427.95,155.48) .. controls (431.18,155.48) and (433.8,158.1) .. (433.8,161.33) .. controls (433.8,164.56) and (431.18,167.17) .. (427.95,167.17) .. controls (424.72,167.17) and (422.1,164.56) .. (422.1,161.33) -- cycle ;
					
					\draw  [color={rgb, 255:red, 0; green, 0; blue, 0 }  ,draw opacity=1 ][fill={rgb, 255:red, 208; green, 2; blue, 27 }  ,fill opacity=1 ][line width=1.5]  (563.1,58.72) .. controls (563.1,55.49) and (565.72,52.87) .. (568.95,52.87) .. controls (572.18,52.87) and (574.8,55.49) .. (574.8,58.72) .. controls (574.8,61.95) and (572.18,64.57) .. (568.95,64.57) .. controls (565.72,64.57) and (563.1,61.95) .. (563.1,58.72) -- cycle ;
					
					\draw  [color={rgb, 255:red, 0; green, 0; blue, 0 }  ,draw opacity=1 ][fill={rgb, 255:red, 208; green, 2; blue, 27 }  ,fill opacity=1 ][line width=1.5]  (562.95,161.33) .. controls (562.95,158.1) and (565.57,155.48) .. (568.8,155.48) .. controls (572.03,155.48) and (574.65,158.1) .. (574.65,161.33) .. controls (574.65,164.56) and (572.03,167.17) .. (568.8,167.17) .. controls (565.57,167.17) and (562.95,164.56) .. (562.95,161.33) -- cycle ;
					
					\draw  [color={rgb, 255:red, 0; green, 0; blue, 0 }  ,draw opacity=1 ][fill={rgb, 255:red, 208; green, 2; blue, 27 }  ,fill opacity=1 ][line width=1.5]  (256.1,281.72) .. controls (256.1,278.49) and (258.72,275.87) .. (261.95,275.87) .. controls (265.18,275.87) and (267.8,278.49) .. (267.8,281.72) .. controls (267.8,284.95) and (265.18,287.57) .. (261.95,287.57) .. controls (258.72,287.57) and (256.1,284.95) .. (256.1,281.72) -- cycle ;
					
					\draw  [color={rgb, 255:red, 0; green, 0; blue, 0 }  ,draw opacity=1 ][fill={rgb, 255:red, 208; green, 2; blue, 27 }  ,fill opacity=1 ][line width=1.5]  (256.1,384.33) .. controls (256.1,381.1) and (258.72,378.48) .. (261.95,378.48) .. controls (265.18,378.48) and (267.8,381.1) .. (267.8,384.33) .. controls (267.8,387.56) and (265.18,390.17) .. (261.95,390.17) .. controls (258.72,390.17) and (256.1,387.56) .. (256.1,384.33) -- cycle ;
					
					\draw  [color={rgb, 255:red, 0; green, 0; blue, 0 }  ,draw opacity=1 ][fill={rgb, 255:red, 208; green, 2; blue, 27 }  ,fill opacity=1 ][line width=1.5]  (397.1,281.72) .. controls (397.1,278.49) and (399.72,275.87) .. (402.95,275.87) .. controls (406.18,275.87) and (408.8,278.49) .. (408.8,281.72) .. controls (408.8,284.95) and (406.18,287.57) .. (402.95,287.57) .. controls (399.72,287.57) and (397.1,284.95) .. (397.1,281.72) -- cycle ;
					
					\draw  [color={rgb, 255:red, 0; green, 0; blue, 0 }  ,draw opacity=1 ][fill={rgb, 255:red, 208; green, 2; blue, 27 }  ,fill opacity=1 ][line width=1.5]  (396.95,384.33) .. controls (396.95,381.1) and (399.57,378.48) .. (402.8,378.48) .. controls (406.03,378.48) and (408.65,381.1) .. (408.65,384.33) .. controls (408.65,387.56) and (406.03,390.17) .. (402.8,390.17) .. controls (399.57,390.17) and (396.95,387.56) .. (396.95,384.33) -- cycle ;
					
					\draw [color={rgb, 255:red, 74; green, 74; blue, 74 }  ,draw opacity=1 ]   (389.89,244.11) -- (454.13,209.07) ;
					\draw [shift={(455.89,208.11)}, rotate = 151.39] [fill={rgb, 255:red, 74; green, 74; blue, 74 }  ,fill opacity=1 ][line width=0.08]  [draw opacity=0] (12,-3) -- (0,0) -- (12,3) -- cycle    ;
					
					\draw (249.1,28.23) node [anchor=north west][inner sep=0.75pt]  [font=\normalsize]  {$\frac{1}4$};
					
					\draw (268.1,87.23) node [anchor=north west][inner sep=0.75pt]  [font=\normalsize]  {$\frac{3}4$};
					
					\draw (375,109) node [anchor=north west][inner sep=0.75pt]  [font=\normalsize]  {$\frac{1}3$};
					
					\draw (393.1,163.23) node [anchor=north west][inner sep=0.75pt]  [font=\normalsize]  {$\frac{2}3$};
					
					\draw (306,50) node [anchor=north west][inner sep=0.75pt]  [font=\large]  {$\Id$};
					
					\draw (351,148) node [anchor=north west][inner sep=0.75pt]  [font=\large]  {$\Id$};
					
					\draw (306,17) node [anchor=north west][inner sep=0.75pt]  [font=\large]  {$\sigma_{x}$};
					
					\draw (352,183) node [anchor=north west][inner sep=0.75pt]  [font=\large]  {$\sigma_{z}$};
					
					\draw (209,220.4) node [anchor=north west][inner sep=0.75pt]  [font=\large]  {$J_{a}^{\mathrm{emb}}$};
					
					\draw (432,220.4) node [anchor=north west][inner sep=0.75pt]  [font=\large]  {$J_{b}^{\mathrm{emb}}$};
					
					\draw (152,203.4) node [anchor=north west][inner sep=0.75pt]  [font=\large]  {$D_{a}$};
					
					\draw (491,202.4) node [anchor=north west][inner sep=0.75pt]  [font=\large]  {$D_{b}$};
					
					\draw (314.26,430.45) node [anchor=north west][inner sep=0.75pt]  [font=\large]  {$\mathsf{P}\mathbb{C}^{2}$};
					
					\draw (417,269.4) node [anchor=north west][inner sep=0.75pt]   [font=\normalsize] {$\hat{x}$};
					
					\draw (208,272.4) node [anchor=north west][inner sep=0.75pt]  [font=\normalsize]  {$\sigma_{z} \cdot \hat{x}$};
					
					\draw (207,374.4) node [anchor=north west][inner sep=0.75pt]  [font=\normalsize]  {$\sigma_{y} \cdot \hat{x}$};
					
					\draw (416,373.4) node [anchor=north west][inner sep=0.75pt] [font=\normalsize]   {$\sigma_{x} \cdot \hat{x}$};  
					
				\end{tikzpicture}
			}
		\end{center} 
		\caption{This figure refers to Example 1. A generic trajectory  under the action of a minimal group  on the reference space $\pc{2}$  consists of four points. A generic $\Pi$-ergodic measure on $\pc{4}$ is uniform on eight atoms.} \label{fig:ex1pic1}
	\end{figure}
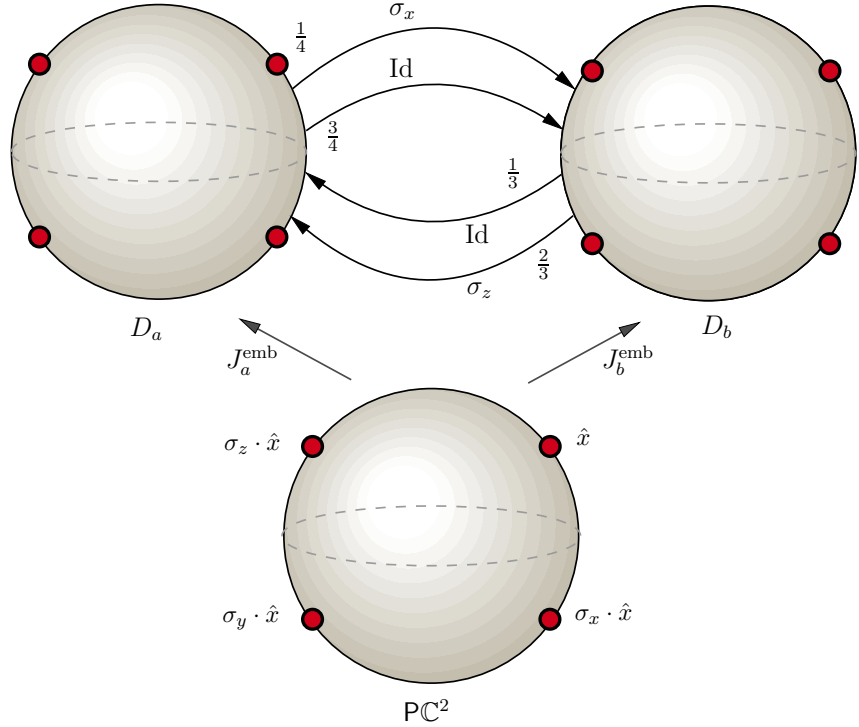

	\begin{figure}[ h]
		\begin{center}
			\scalebox{0.6907}{

				\tikzset {_rmpfxf0nq/.code = {\pgfsetadditionalshadetransform{ \pgftransformshift{\pgfpoint{89.1 bp } { -108.9 bp }  }  \pgftransformscale{1.32 }  }}}
				\pgfdeclareradialshading{_if97qjtp0}{\pgfpoint{-72bp}{88bp}}{rgb(0bp)=(1,1,1);
					rgb(3.9186504908970425bp)=(1,1,1);
					rgb(25bp)=(0.67,0.64,0.55);
					rgb(25bp)=(0.78,0.75,0.66);
					rgb(400bp)=(0.78,0.75,0.66)}
				\tikzset{_2edlmti46/.code = {\pgfsetadditionalshadetransform{\pgftransformshift{\pgfpoint{89.1 bp } { -108.9 bp }  }  \pgftransformscale{1.32 } }}}
				\pgfdeclareradialshading{_p9l1cu52f} { \pgfpoint{-72bp} {88bp}} {color(0bp)=(transparent!0);
					color(3.9186504908970425bp)=(transparent!0);
					color(25bp)=(transparent!0);
					color(25bp)=(transparent!21);
					color(400bp)=(transparent!21)} 
				\pgfdeclarefading{_hyw8g3tth}{\tikz \fill[shading=_p9l1cu52f,_2edlmti46] (0,0) rectangle (50bp,50bp); } 
				\tikzset{every picture/.style={line width=0.75pt}}         
				
				\begin{tikzpicture}[x=0.75pt,y=0.75pt,yscale=-1,xscale=1] 
					\path (221.5,0);  
					
					\path  [shading=_if97qjtp0,_rmpfxf0nq,path fading= _hyw8g3tth ,fading transform={xshift=2}] (236.81,123.57) .. controls (236.81,64.33) and (284.84,16.3) .. (344.08,16.3) .. controls (403.32,16.3) and (451.34,64.33) .. (451.34,123.57) .. controls (451.34,182.81) and (403.32,230.83) .. (344.08,230.83) .. controls (284.84,230.83) and (236.81,182.81) .. (236.81,123.57) -- cycle ;  
					\draw   (236.81,123.57) .. controls (236.81,64.33) and (284.84,16.3) .. (344.08,16.3) .. controls (403.32,16.3) and (451.34,64.33) .. (451.34,123.57) .. controls (451.34,182.81) and (403.32,230.83) .. (344.08,230.83) .. controls (284.84,230.83) and (236.81,182.81) .. (236.81,123.57) -- cycle ;

					\draw  [color={rgb, 255:red, 139; green, 87; blue, 42 }  ,draw opacity=1 ][line width=3]  (249.81,68.73) .. controls (249.81,58.42) and (291.59,50.06) .. (343.13,50.06) .. controls (394.67,50.06) and (436.44,58.42) .. (436.44,68.73) .. controls (436.44,79.04) and (394.67,87.4) .. (343.13,87.4) .. controls (291.59,87.4) and (249.81,79.04) .. (249.81,68.73) -- cycle ;
					
					\draw  [color={rgb, 255:red, 208; green, 2; blue, 27 }  ,draw opacity=1 ][fill={rgb, 255:red, 208; green, 2; blue, 27 }  ,fill opacity=1 ] (360.71,87.56) .. controls (360.71,85.04) and (362.75,82.99) .. (365.27,82.99) .. controls (367.79,82.99) and (369.83,85.04) .. (369.83,87.56) .. controls (369.83,90.08) and (367.79,92.12) .. (365.27,92.12) .. controls (362.75,92.12) and (360.71,90.08) .. (360.71,87.56) -- cycle ;
					
					\draw  [color={rgb, 255:red, 208; green, 2; blue, 27 }  ,draw opacity=1 ][fill={rgb, 255:red, 208; green, 2; blue, 27 }  ,fill opacity=1 ] (376.67,52.06) .. controls (376.67,49.54) and (378.71,47.5) .. (381.23,47.5) .. controls (383.75,47.5) and (385.79,49.54) .. (385.79,52.06) .. controls (385.79,54.58) and (383.75,56.62) .. (381.23,56.62) .. controls (378.71,56.62) and (376.67,54.58) .. (376.67,52.06) -- cycle ;
					
					\draw  [color={rgb, 255:red, 208; green, 2; blue, 27 }  ,draw opacity=1 ][fill={rgb, 255:red, 208; green, 2; blue, 27 }  ,fill opacity=1 ] (307.21,51.06) .. controls (307.21,48.54) and (309.25,46.49) .. (311.77,46.49) .. controls (314.29,46.49) and (316.33,48.54) .. (316.33,51.06) .. controls (316.33,53.58) and (314.29,55.62) .. (311.77,55.62) .. controls (309.25,55.62) and (307.21,53.58) .. (307.21,51.06) -- cycle ;
					
					\draw  [color={rgb, 255:red, 208; green, 2; blue, 27 }  ,draw opacity=1 ][fill={rgb, 255:red, 208; green, 2; blue, 27 }  ,fill opacity=1 ] (290.21,85.06) .. controls (290.21,82.54) and (292.25,80.49) .. (294.77,80.49) .. controls (297.29,80.49) and (299.33,82.54) .. (299.33,85.06) .. controls (299.33,87.58) and (297.29,89.62) .. (294.77,89.62) .. controls (292.25,89.62) and (290.21,87.58) .. (290.21,85.06) -- cycle ;
					
					\draw  [color={rgb, 255:red, 0; green, 76; blue, 255 }  ,draw opacity=1 ][fill={rgb, 255:red, 0; green, 76; blue, 255 }  ,fill opacity=1 ] (325.71,81.49) -- (336.83,81.49) -- (336.83,92.62) -- (325.71,92.62) -- cycle ;
					
					\draw  [color={rgb, 255:red, 155; green, 155; blue, 155 }  ,draw opacity=1 ][dash pattern={on 4.5pt off 4.5pt}][line width=0.75]  (236.81,123.57) .. controls (236.81,111.71) and (284.84,102.1) .. (344.08,102.1) .. controls (403.32,102.1) and (451.34,111.71) .. (451.34,123.57) .. controls (451.34,135.42) and (403.32,145.03) .. (344.08,145.03) .. controls (284.84,145.03) and (236.81,135.42) .. (236.81,123.57) -- cycle ;
					
					\draw  [color={rgb, 255:red, 0; green, 76; blue, 255 }  ,draw opacity=1 ][fill={rgb, 255:red, 0; green, 76; blue, 255 }  ,fill opacity=1 ] (337.56,44.49) -- (348.69,44.49) -- (348.69,55.62) -- (337.56,55.62) -- cycle ;
					
					\draw  [color={rgb, 255:red, 0; green, 76; blue, 255 }  ,draw opacity=1 ][fill={rgb, 255:red, 0; green, 76; blue, 255 }  ,fill opacity=1 ] (430.88,63.17) -- (442.01,63.17) -- (442.01,74.29) -- (430.88,74.29) -- cycle ;
					
					\draw  [color={rgb, 255:red, 0; green, 76; blue, 255 }  ,draw opacity=1 ][fill={rgb, 255:red, 0; green, 76; blue, 255 }  ,fill opacity=1 ] (244.25,63.17) -- (255.37,63.17) -- (255.37,74.29) -- (244.25,74.29) -- cycle ;
					
					\draw  [color={rgb, 255:red, 208; green, 2; blue, 27 }  ,draw opacity=1 ][fill={rgb, 255:red, 208; green, 2; blue, 27 }  ,fill opacity=1 ] (245.78,68.73) .. controls (245.78,66.5) and (247.59,64.7) .. (249.81,64.7) .. controls (252.04,64.7) and (253.84,66.5) .. (253.84,68.73) .. controls (253.84,70.96) and (252.04,72.76) .. (249.81,72.76) .. controls (247.59,72.76) and (245.78,70.96) .. (245.78,68.73) -- cycle ;
					
					\draw  [color={rgb, 255:red, 208; green, 2; blue, 27 }  ,draw opacity=1 ][fill={rgb, 255:red, 208; green, 2; blue, 27 }  ,fill opacity=1 ] (432.41,68.73) .. controls (432.41,66.5) and (434.22,64.7) .. (436.44,64.7) .. controls (438.67,64.7) and (440.48,66.5) .. (440.48,68.73) .. controls (440.48,70.96) and (438.67,72.76) .. (436.44,72.76) .. controls (434.22,72.76) and (432.41,70.96) .. (432.41,68.73) -- cycle ;
					
					\draw  [color={rgb, 255:red, 139; green, 87; blue, 42 }  ,draw opacity=1 ][line width=3]  (249.81,172.31) .. controls (249.81,162) and (291.59,153.64) .. (343.13,153.64) .. controls (394.67,153.64) and (436.44,162) .. (436.44,172.31) .. controls (436.44,182.63) and (394.67,190.99) .. (343.13,190.99) .. controls (291.59,190.99) and (249.81,182.63) .. (249.81,172.31) -- cycle ;
					
					\draw  [color={rgb, 255:red, 208; green, 2; blue, 27 }  ,draw opacity=1 ][fill={rgb, 255:red, 208; green, 2; blue, 27 }  ,fill opacity=1 ] (360.71,191.14) .. controls (360.71,188.62) and (362.75,186.58) .. (365.27,186.58) .. controls (367.79,186.58) and (369.83,188.62) .. (369.83,191.14) .. controls (369.83,193.66) and (367.79,195.7) .. (365.27,195.7) .. controls (362.75,195.7) and (360.71,193.66) .. (360.71,191.14) -- cycle ;
					
					\draw  [color={rgb, 255:red, 208; green, 2; blue, 27 }  ,draw opacity=1 ][fill={rgb, 255:red, 208; green, 2; blue, 27 }  ,fill opacity=1 ] (376.67,155.64) .. controls (376.67,153.12) and (378.71,151.08) .. (381.23,151.08) .. controls (383.75,151.08) and (385.79,153.12) .. (385.79,155.64) .. controls (385.79,158.16) and (383.75,160.21) .. (381.23,160.21) .. controls (378.71,160.21) and (376.67,158.16) .. (376.67,155.64) -- cycle ;
					
					\draw  [color={rgb, 255:red, 208; green, 2; blue, 27 }  ,draw opacity=1 ][fill={rgb, 255:red, 208; green, 2; blue, 27 }  ,fill opacity=1 ] (307.21,154.64) .. controls (307.21,152.12) and (309.25,150.08) .. (311.77,150.08) .. controls (314.29,150.08) and (316.33,152.12) .. (316.33,154.64) .. controls (316.33,157.16) and (314.29,159.2) .. (311.77,159.2) .. controls (309.25,159.2) and (307.21,157.16) .. (307.21,154.64) -- cycle ;
					
					\draw  [color={rgb, 255:red, 208; green, 2; blue, 27 }  ,draw opacity=1 ][fill={rgb, 255:red, 208; green, 2; blue, 27 }  ,fill opacity=1 ] (290.21,188.64) .. controls (290.21,186.12) and (292.25,184.08) .. (294.77,184.08) .. controls (297.29,184.08) and (299.33,186.12) .. (299.33,188.64) .. controls (299.33,191.16) and (297.29,193.2) .. (294.77,193.2) .. controls (292.25,193.2) and (290.21,191.16) .. (290.21,188.64) -- cycle ;
					
					\draw  [color={rgb, 255:red, 0; green, 76; blue, 255 }  ,draw opacity=1 ][fill={rgb, 255:red, 0; green, 76; blue, 255 }  ,fill opacity=1 ] (325.71,185.08) -- (336.83,185.08) -- (336.83,196.2) -- (325.71,196.2) -- cycle ;
					
					\draw  [color={rgb, 255:red, 0; green, 76; blue, 255 }  ,draw opacity=1 ][fill={rgb, 255:red, 0; green, 76; blue, 255 }  ,fill opacity=1 ] (337.56,148.08) -- (348.69,148.08) -- (348.69,159.2) -- (337.56,159.2) -- cycle ;
					
					\draw  [color={rgb, 255:red, 0; green, 76; blue, 255 }  ,draw opacity=1 ][fill={rgb, 255:red, 0; green, 76; blue, 255 }  ,fill opacity=1 ] (430.88,166.75) -- (442.01,166.75) -- (442.01,177.88) -- (430.88,177.88) -- cycle ;
					
					\draw  [color={rgb, 255:red, 0; green, 76; blue, 255 }  ,draw opacity=1 ][fill={rgb, 255:red, 0; green, 76; blue, 255 }  ,fill opacity=1 ] (244.25,166.75) -- (255.37,166.75) -- (255.37,177.88) -- (244.25,177.88) -- cycle ;
					
					\draw  [color={rgb, 255:red, 208; green, 2; blue, 27 }  ,draw opacity=1 ][fill={rgb, 255:red, 208; green, 2; blue, 27 }  ,fill opacity=1 ] (245.78,172.31) .. controls (245.78,170.09) and (247.59,168.28) .. (249.81,168.28) .. controls (252.04,168.28) and (253.84,170.09) .. (253.84,172.31) .. controls (253.84,174.54) and (252.04,176.34) .. (249.81,176.34) .. controls (247.59,176.34) and (245.78,174.54) .. (245.78,172.31) -- cycle ;
					
					\draw  [color={rgb, 255:red, 208; green, 2; blue, 27 }  ,draw opacity=1 ][fill={rgb, 255:red, 208; green, 2; blue, 27 }  ,fill opacity=1 ] (432.41,172.31) .. controls (432.41,170.09) and (434.22,168.28) .. (436.44,168.28) .. controls (438.67,168.28) and (440.48,170.09) .. (440.48,172.31) .. controls (440.48,174.54) and (438.67,176.34) .. (436.44,176.34) .. controls (434.22,176.34) and (432.41,174.54) .. (432.41,172.31) -- cycle ;

					\draw (450,57.22) node [anchor=north west][inner sep=0.75pt]    {$\hat{x}$};

				\end{tikzpicture}
			}
		\end{center} 
		\caption{This figure refers to Example 1. It shows trajectories of a generic point $\hat x \in \pc{2}$   under the action of  three non-minimal groups from Example \hyperref[ex:16]{1(6)}. The red bullets are the 12-element set $[\hat x]_{G_{\tJ_1}}$, the blue squares are the 8-element set $[\hat x]_{G_{\tJ_2}}$, and the brown circles correspond to the uncountable set $[\hat x]_{G_{\tJ_3}}$. All three orbits contain $[\hat x]_{G_{\rm min}}$ as a proper subset, cf. \Cref{fig:ex1pic1}.} \label{fig:ex1pic2}
	\end{figure}

	\subsubsection*{Example 2. (Adjacent Bloch spheres)}\label{example2} 
	\noindent
	Consider two matrices:
	\[v_1 = \frac{1}{\sqrt 2} \begin{bmatrix}
		\cos\theta & \sin\theta & \cos\theta \\
		\sin\theta & -\cos\theta & \sin\theta \\
		0 & 0 & 0
	\end{bmatrix} \ \ \textrm{and} \ \ v_2 = \frac{1}{\sqrt 2} \begin{bmatrix}
		0 & 0 & 0 \\
		\cos\phi & -\sin\phi & -\cos\phi \\
		\sin\phi & \cos\phi & -\sin\phi
	\end{bmatrix}\]
	with $\theta, \phi\in\rr$.
	\smallskip
	
	\noindent
	\textbf{(0) Intuition}.
	To develop some intuition about the dynamics of this system, let us first discuss the action of $v_1$ and $v_2$ on some real vectors. 
	Denote \[S_{\alpha} = \begin{bmatrix}
		\cos\alpha & \sin\alpha  \\
		\sin\alpha & -\cos\alpha \end{bmatrix} \ \ \textrm{and} \ \ Q_{\alpha} = \begin{bmatrix}
		\cos\alpha & -\sin\alpha   \\
		\sin\alpha & \cos\alpha\end{bmatrix}.\]
	When acting on a plane, $S_{\alpha}$ describes the reflection of this plane about a line through the origin which makes an angle $\alpha/2$ with the horizontal axis, while $Q_{\alpha}$ is the (counterclockwise) rotation about the origin through the angle $\alpha$. 
	Consider now the two  planes \[P_a = \left\{[x,y,0]^T \:|\: x,y \in \rr\right\} \ \ \textrm{and} \ \ P_b = \left\{[0,x,y]^T \:|\: x,y \in \rr\right\}.\]
	The action of $v_1$ and $v_2$ on $P_a$ and $P_b$ can be described as follows. A vector in $P_a$ can either stay in $P_a$ while being reflected using $S_{\theta}$ (if $v_1$ acted), or it can jump to $P_b$ while being rotated by $Q_{\phi}$ (if $v_2$ acted). Similarly, a vector in $P_b$ can either stay in $P_b$ while being rotated by $Q_{\phi + \frac \pi 2}$ (if $v_2$ acted) or jump to $P_a$ while being rotated by $Q_{\theta - \frac \pi 2}$ (if $v_1$ acted). The jumps between planes happen with constant probability $1/2$.

	\smallskip
	
	Reverting back to the complex field, let $P_a$ and $P_b$ denote now their complex extension, namely
	\[P_a = \left\{[x,y,0]^T \:|\: x,y \in \cc\right\} \ \ \textrm{and} \ \ P_b = \left\{[0,x,y]^T\:|\: x,y \in \cc\right\}.\]
	\noindent
	\textbf{(1) Irreducibility}. 
	To ensure that $\mu = \delta_{v_1}+\delta_{v_2}$ is irreducible and avoid trivial statements, we assume that $\theta\notin \mathbb{Z}\pi$ and $\phi \notin \mathbb{Z}\pi/2$. Indeed, since $v_1\cc^3=P_a$ and $v_2\cc^3=P_b$, $\mu$ is not irreducible only if $P_a\cap P_b$ is invariant for $v_1$ and $v_2$, namely only if $e_1=[0,1,0]^{T}$ is a common eigenvector of $v_1$ and $v_2$. That is not the case either if $\theta\notin \mathbb{Z}\pi$ or $\phi\notin\pi/2+\mathbb{Z}\pi$.
	
	\smallskip
	
	\noindent
	\textbf{(2) Dark subspaces}.
	Both $P_a$ and $P_b$ are dark subspaces. Therefore, we denote them by $D_a$ and $D_b$ respectively. They share a one-dimensional subspace equal to $\cc e_1$.  
	Hence, the two projective dark subspaces $\P(D_a)$ and $\P(D_b)$ intersect at one ray and are both isomorphic to $\pc{2}$, so they can be visualized as two Bloch spheres tangent to each other.  
	
	To see that $D_a$ and $D_b$  are  maximal dark subspaces, it suffices to observe that the whole space $\cc^3$ is not dark. That is a direct consequence of  $\rank v_1=2$, which restricts dark subspaces to a dimension of at most $2$.

	\smallskip

	\noindent
	\textbf{(3) Invariant measure for dark spaces}.
	The Markov chain on $\{D_a, D_b\}$ has transition matrix $\frac 12 \begin{bsmallmatrix}1 & 1 \\ 1 & 1 	\end{bsmallmatrix} $;
	hence, its unique invariant measure is uniform: \[\chi_{\rm inv}=\tfrac 12 \delta_{D_a} + \tfrac 12 \delta_{D_b}.\] As in the previous example, \Cref{thm:invmeasDark} implies it is the unique $K$-invariant measure on $\mathcal D_m$, and so we restrict our attention to the dark subspaces $D_a, D_b$.
	
	\smallskip
	
	\noindent
	\textbf{(4) Minimal family}. 
	For the family of canonical embeddings  $\J_{\rm emb} = \{J_a^{\rm emb}, J_b^{\rm emb}\}$ with
	\[J_a^{\rm emb} \colon \begin{bsmallmatrix}   x\\y  \end{bsmallmatrix} \mapsto \begin{bsmallmatrix}   x\\y\\0 \end{bsmallmatrix} \ \ \text{ and } \ \ J_b^{\rm emb} \colon \begin{bsmallmatrix}   x\\y  \end{bsmallmatrix} \mapsto \begin{bsmallmatrix}  0 \\x\\y \end{bsmallmatrix}\]
	the group $G_{\J_{\rm emb}}$ is the closure of $\langle \i S_{\theta}, Q_{\theta - \frac \pi 2}, Q_{\phi}, Q_{\phi+\frac \pi 2} \rangle$. 
	As in the previous example, $\J_{\rm emb}$ need not be minimal.  
	Let us construct a $D_a-$smart family $\J = \{J_a, J_b\}$ with $J_a = J_a^{\rm emb}$. Setting  $J_b=\sqrt2\,v_2J_a=J_b^{\mathrm{emb}}Q_\phi$  we obtain 
	\[u_{D_a,1} = \i S_{\theta}, \   u_{D_b,1} = Q_{\theta - \frac \pi 2 + \phi}, \   u_{D_a,2} = \id, \   u_{D_b,2} = Q_{\phi+\frac \pi 2}.\]
	and $G_{\rm min}$ is the closure of $\langle \i S_{\theta},  Q_{\theta}, Q_{\phi+\frac \pi 2} \rangle$. If we revert to elements of $\mathcal{SU}(2)$, since $Q(\alpha)=R_y(-2\alpha)=\exp(-\i\alpha\sigma_y)$ and $\i S_\alpha=\i\sigma_zR_y(2\alpha)$ for any $\alpha\in \rr$, 
	\[G_{\rm min}=\overline{\langle \i\sigma_z,R_y(2\theta),R_y(2\phi+\pi)\rangle}.\]
	We can then distinguish two cases: 
	\begin{enumerate}[label={\bf (\alph*)}, leftmargin=8mm]
		\item    
		If either $\phi$ or $\theta$ is not in $\pi\mathbb Q$, then  
		\[G_{\rm min} = \big\{R_{y}(t),  \i\sigma_zR_{y}(t) \:|\: t\in \rr\big\}.\]
		Thus, denoting $G_{1,b}$ the minimal group of Example~\hyperref[ex:15b]{1(5b)}, 
		\[G_{\rm min}=R_x(\pi/2)R_y(\pi/2)G_{1,b}R_y(-\pi/2)R_x(-\pi/2).\]
		It follows that the orbits of $G_{\rm min}$ are the image by $R_x(\pi/2)R_y(\pi/2)$ of the orbits of Example~\hyperref[ex:15b]{1(5b)}.  For $c\in[0,1/\sqrt{2}]$, denoting $C_c=\{R_x(\pi/2)R_y(\pi/2)[c \mathrm{e}^{it}, \sqrt{1-c^2}]^T \:|\: t\in \rr\}$ and $A_c=C_c\cup \sigma_xC_c$, we have $\Upsilon_\J = \{\nu_c \:|\: c \in [0,1/\sqrt{2}]\}$, where  \[\nu_c = \tfrac 12 \operatorname{Unif}_{\P(J_{a}A_c)} + \tfrac 12 \operatorname{Unif}_{\P(J_{b}A_c)}.\]   Then, the only finitely-supported $\Pi$-ergodic measure is $\nu_{0}$, which is a uniform measure on the four atoms $\begin{bsmallmatrix} 1 \\ \mathrm{i} \\ 0 \end{bsmallmatrix}, \begin{bsmallmatrix} 1 \\ \mathrm{-i} \\ 0 \end{bsmallmatrix}, \begin{bsmallmatrix} 0 \\ 1 \\ \mathrm{i} \end{bsmallmatrix}, \begin{bsmallmatrix} 0 \\ 1 \\ \mathrm{-i} \end{bsmallmatrix}$, corresponding to the $y$ oriented poles of the two Bloch spheres.  See \Cref{fig:ex2pic1} for an illustration.
		
		\item    
		If $\phi, \theta \in \pi\mathbb{Q}$, then
		$G_{\rm min}$ is a finite group and every $\Pi$-ergodic measure is finitely supported.  
		More precisely, if  $\theta =  {\pi p_1}/{q_1}$ and $\phi+  \pi/2 =   {\pi p_2}/{q_2}$, where $\operatorname{gcd}(p_1,q_1)=\operatorname{gcd}(p_2,q_2)=1$ with $\operatorname{gcd}$ denoting the greatest common divisor, then
		\[G_{\rm min} = \langle \i\sigma_z,  R_y(2\pi/q_1), R_y(2\pi/q_2) \rangle   = \langle \i\sigma_z,  R_y(2\pi \operatorname{gcd}(q_1,q_2)/(q_1q_2)) \rangle.\]
		Hence, a generic trajectory under $G_{\rm min}$ consists of $2q_1q_2 / \operatorname{gcd}(q_1,q_2)$ points, and a generic $\Pi$-ergodic measure is uniform over $4q_1q_2/ \operatorname{gcd}(q_1,q_2)$ atoms, 
		see also \Cref{fig:ex2pic2}. 
		
		\noindent 
		However, non-uniform $\Pi$-ergodic measures exist as well: $\nu_{\hat x}$ is non-uniform iff the tangency point of the two Bloch spheres belongs to $\supp \nu_{\hat x}$. This happens, for instance, if $\hat x = [0,1]^T$, i.e., the starting point on the reference Bloch sphere corresponds to the tangency point.  
		Then  $\nu_{\hat x}$ has $4q_1q_2/ \operatorname{gcd}(q_1,q_2)-1$ atoms and  the atom at the tangency point, which is $\delta_{[0,1,0]^T}$, has double weight compared to all the other atoms.
	\end{enumerate}

	\begin{figure}[ h]
		\begin{center}
			\scalebox{0.6125}{
				
				\tikzset {_hf7d57enn/.code = {\pgfsetadditionalshadetransform{ \pgftransformshift{\pgfpoint{89.1 bp } { -108.9 bp }  }  \pgftransformscale{1.32 }  }}}
				\pgfdeclareradialshading{_eywscd9bn}{\pgfpoint{-72bp}{88bp}}{rgb(0bp)=(1,1,1);
					rgb(3.9186504908970425bp)=(1,1,1);
					rgb(25bp)=(0.67,0.64,0.55);
					rgb(25bp)=(0.78,0.75,0.66);
					rgb(400bp)=(0.78,0.75,0.66)}
				\tikzset{every picture/.style={line width=0.75pt}}        
				
				\begin{tikzpicture}[x=0.75pt,y=0.75pt,yscale=-1,xscale=1]
					
					\path  [shading=_eywscd9bn,_hf7d57enn] (159.3,150.8) .. controls (159.3,79.85) and (216.81,22.33) .. (287.76,22.33) .. controls (358.71,22.33) and (416.22,79.85) .. (416.22,150.8) .. controls (416.22,221.74) and (358.71,279.26) .. (287.76,279.26) .. controls (216.81,279.26) and (159.3,221.74) .. (159.3,150.8) -- cycle ;  
					\draw   (159.3,150.8) .. controls (159.3,79.85) and (216.81,22.33) .. (287.76,22.33) .. controls (358.71,22.33) and (416.22,79.85) .. (416.22,150.8) .. controls (416.22,221.74) and (358.71,279.26) .. (287.76,279.26) .. controls (216.81,279.26) and (159.3,221.74) .. (159.3,150.8) -- cycle ;

					\draw  [color={rgb, 255:red, 155; green, 155; blue, 155 }  ,draw opacity=1 ][dash pattern={on 4.5pt off 4.5pt}][line width=0.75]  (159.3,150.8) .. controls (159.3,136.6) and (216.81,125.09) .. (287.76,125.09) .. controls (358.71,125.09) and (416.22,136.6) .. (416.22,150.8) .. controls (416.22,164.99) and (358.71,176.5) .. (287.76,176.5) .. controls (216.81,176.5) and (159.3,164.99) .. (159.3,150.8) -- cycle ;
					
					\draw  [color={rgb, 255:red, 208; green, 2; blue, 27 }  ,draw opacity=1 ][line width=2.25]  (159.3,151.65) .. controls (159.3,137.45) and (216.81,125.94) .. (287.76,125.94) .. controls (358.71,125.94) and (416.22,137.45) .. (416.22,151.65) .. controls (416.22,165.85) and (358.71,177.36) .. (287.76,177.36) .. controls (216.81,177.36) and (159.3,165.85) .. (159.3,151.65) -- cycle ;
					
					\draw  [color={rgb, 255:red, 139; green, 87; blue, 42 }  ,draw opacity=1 ][line width=2.25]  (188.3,72.69) .. controls (188.3,61.78) and (232.47,52.94) .. (286.96,52.94) .. controls (341.45,52.94) and (385.62,61.78) .. (385.62,72.69) .. controls (385.62,83.59) and (341.45,92.43) .. (286.96,92.43) .. controls (232.47,92.43) and (188.3,83.59) .. (188.3,72.69) -- cycle ;
					
					\draw  [color={rgb, 255:red, 139; green, 87; blue, 42 }  ,draw opacity=1 ][line width=2.25]  (188.3,230.69) .. controls (188.3,219.78) and (232.47,210.94) .. (286.96,210.94) .. controls (341.45,210.94) and (385.62,219.78) .. (385.62,230.69) .. controls (385.62,241.59) and (341.45,250.43) .. (286.96,250.43) .. controls (232.47,250.43) and (188.3,241.59) .. (188.3,230.69) -- cycle ;
					
					\draw  [draw opacity=0][fill={rgb, 255:red, 0; green, 76; blue, 255 }  ,fill opacity=1 ][line width=1.5]  (281.91,22.33) .. controls (281.91,19.1) and (284.53,16.48) .. (287.76,16.48) .. controls (290.99,16.48) and (293.61,19.1) .. (293.61,22.33) .. controls (293.61,25.56) and (290.99,28.18) .. (287.76,28.18) .. controls (284.53,28.18) and (281.91,25.56) .. (281.91,22.33) -- cycle ;
					
					\draw  [draw opacity=0][fill={rgb, 255:red, 0; green, 76; blue, 255 }  ,fill opacity=1 ][line width=1.5]  (281.91,279.26) .. controls (281.91,276.03) and (284.53,273.41) .. (287.76,273.41) .. controls (290.99,273.41) and (293.61,276.03) .. (293.61,279.26) .. controls (293.61,282.49) and (290.99,285.11) .. (287.76,285.11) .. controls (284.53,285.11) and (281.91,282.49) .. (281.91,279.26) -- cycle ;

					\draw (264.26,292.01) node [anchor=north west][inner sep=0.75pt]  [font=\Large]  {$\mathsf{P}\mathbb{C}^{2}$};

				\end{tikzpicture}
			}
		\end{center} 
		\caption{This figure refers to Example 2. It shows invariant subsets of  $\pc{2}$. The Bloch sphere is so oriented that the $y$-axis is pointing upwards. A generic invariant subset consists of two circles of latitude symmetric with respect to the equator of the Bloch sphere (in brown). Degenerate invariant subsets are the equator (in red) and the poles (in blue). 
			Recall that on the Bloch sphere  $Q_{\alpha} = R_y(-2\alpha)$ is the rotation around the $y$-axis by the angle $2\alpha$, while $S_{\alpha}$ is the rotation by the angle $\pi$ around an axis determined by some real vector (so contained in the equator) whose exact coefficients depend on $\alpha$. In particular, $S_{\alpha}$ maps the equator to itself, while the upper brown circle is mapped to the lower brown circle and vice versa. 
		}\label{fig:ex2pic1} 
		
	\end{figure}
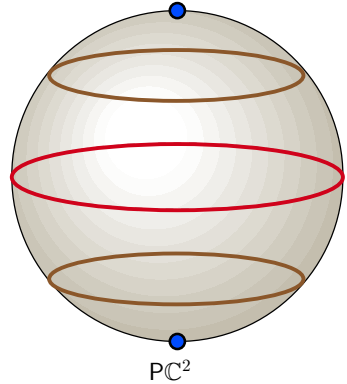

	\begin{figure}[ h]
		
		\medskip

		\tikzset {_26rc30cfz/.code = {\pgfsetadditionalshadetransform{ \pgftransformshift{\pgfpoint{89.1 bp } { -108.9 bp }  }  \pgftransformscale{1.32 }  }}}
		\pgfdeclareradialshading{_hab6c0ace}{\pgfpoint{-72bp}{88bp}}{rgb(0bp)=(1,1,1);
			rgb(3.9186504908970425bp)=(1,1,1);
			rgb(25bp)=(0.67,0.64,0.55);
			rgb(25bp)=(0.78,0.75,0.66);
			rgb(400bp)=(0.78,0.75,0.66)}

		\tikzset {_cxlet4wyq/.code = {\pgfsetadditionalshadetransform{ \pgftransformshift{\pgfpoint{89.1 bp } { -108.9 bp }  }  \pgftransformscale{1.32 }  }}}
		\pgfdeclareradialshading{_bz9to6kcq}{\pgfpoint{-72bp}{88bp}}{rgb(0bp)=(1,1,1);
			rgb(3.9186504908970425bp)=(1,1,1);
			rgb(25bp)=(0.67,0.64,0.55);
			rgb(25bp)=(0.78,0.75,0.66);
			rgb(400bp)=(0.78,0.75,0.66)}
		\tikzset{every picture/.style={line width=0.75pt}}  
		\begin{center}
			\scalebox{0.6125}{
				\begin{tikzpicture}[x=0.75pt,y=0.75pt,yscale=-1,xscale=1]
					
					\path  [shading=_hab6c0ace,_26rc30cfz] (25.3,153.8) .. controls (25.3,82.85) and (82.81,25.33) .. (153.76,25.33) .. controls (224.71,25.33) and (282.22,82.85) .. (282.22,153.8) .. controls (282.22,224.74) and (224.71,282.26) .. (153.76,282.26) .. controls (82.81,282.26) and (25.3,224.74) .. (25.3,153.8) -- cycle ; 
					\draw   (25.3,153.8) .. controls (25.3,82.85) and (82.81,25.33) .. (153.76,25.33) .. controls (224.71,25.33) and (282.22,82.85) .. (282.22,153.8) .. controls (282.22,224.74) and (224.71,282.26) .. (153.76,282.26) .. controls (82.81,282.26) and (25.3,224.74) .. (25.3,153.8) -- cycle ;

					\draw  [color={rgb, 255:red, 155; green, 155; blue, 155 }  ,draw opacity=1 ][dash pattern={on 4.5pt off 4.5pt}][line width=0.75]  (25.3,153.8) .. controls (25.3,139.6) and (82.81,128.09) .. (153.76,128.09) .. controls (224.71,128.09) and (282.22,139.6) .. (282.22,153.8) .. controls (282.22,167.99) and (224.71,179.5) .. (153.76,179.5) .. controls (82.81,179.5) and (25.3,167.99) .. (25.3,153.8) -- cycle ;
					
					\path  [shading=_bz9to6kcq,_cxlet4wyq] (282.22,153.8) .. controls (282.22,82.85) and (339.74,25.33) .. (410.69,25.33) .. controls (481.63,25.33) and (539.15,82.85) .. (539.15,153.8) .. controls (539.15,224.74) and (481.63,282.26) .. (410.69,282.26) .. controls (339.74,282.26) and (282.22,224.74) .. (282.22,153.8) -- cycle ;  
					\draw   (282.22,153.8) .. controls (282.22,82.85) and (339.74,25.33) .. (410.69,25.33) .. controls (481.63,25.33) and (539.15,82.85) .. (539.15,153.8) .. controls (539.15,224.74) and (481.63,282.26) .. (410.69,282.26) .. controls (339.74,282.26) and (282.22,224.74) .. (282.22,153.8) -- cycle ;

					\draw  [color={rgb, 255:red, 155; green, 155; blue, 155 }  ,draw opacity=1 ][dash pattern={on 4.5pt off 4.5pt}][line width=0.75]  (282.22,153.8) .. controls (282.22,139.6) and (339.74,128.09) .. (410.69,128.09) .. controls (481.63,128.09) and (539.15,139.6) .. (539.15,153.8) .. controls (539.15,167.99) and (481.63,179.5) .. (410.69,179.5) .. controls (339.74,179.5) and (282.22,167.99) .. (282.22,153.8) -- cycle ;
					
					\draw  [color={rgb, 255:red, 208; green, 2; blue, 27 }  ,draw opacity=1 ][line width=1.5]  (25.3,154.65) .. controls (25.3,140.45) and (82.81,128.94) .. (153.76,128.94) .. controls (224.71,128.94) and (282.22,140.45) .. (282.22,154.65) .. controls (282.22,168.85) and (224.71,180.36) .. (153.76,180.36) .. controls (82.81,180.36) and (25.3,168.85) .. (25.3,154.65) -- cycle ;
					
					\draw  [color={rgb, 255:red, 208; green, 2; blue, 27 }  ,draw opacity=1 ][line width=1.5]  (282.22,154.65) .. controls (282.22,140.45) and (339.74,128.94) .. (410.69,128.94) .. controls (481.63,128.94) and (539.15,140.45) .. (539.15,154.65) .. controls (539.15,168.85) and (481.63,180.36) .. (410.69,180.36) .. controls (339.74,180.36) and (282.22,168.85) .. (282.22,154.65) -- cycle ;
					
					\draw  [color={rgb, 255:red, 139; green, 87; blue, 42 }  ,draw opacity=1 ][line width=1.5]  (54.3,75.69) .. controls (54.3,64.78) and (98.47,55.94) .. (152.96,55.94) .. controls (207.45,55.94) and (251.62,64.78) .. (251.62,75.69) .. controls (251.62,86.59) and (207.45,95.43) .. (152.96,95.43) .. controls (98.47,95.43) and (54.3,86.59) .. (54.3,75.69) -- cycle ;
					
					\draw  [color={rgb, 255:red, 139; green, 87; blue, 42 }  ,draw opacity=1 ][line width=1.5]  (312.3,73.69) .. controls (312.3,62.78) and (356.47,53.94) .. (410.96,53.94) .. controls (465.45,53.94) and (509.62,62.78) .. (509.62,73.69) .. controls (509.62,84.59) and (465.45,93.43) .. (410.96,93.43) .. controls (356.47,93.43) and (312.3,84.59) .. (312.3,73.69) -- cycle ;
					
					\draw  [color={rgb, 255:red, 139; green, 87; blue, 42 }  ,draw opacity=1 ][line width=1.5]  (54.3,233.69) .. controls (54.3,222.78) and (98.47,213.94) .. (152.96,213.94) .. controls (207.45,213.94) and (251.62,222.78) .. (251.62,233.69) .. controls (251.62,244.59) and (207.45,253.43) .. (152.96,253.43) .. controls (98.47,253.43) and (54.3,244.59) .. (54.3,233.69) -- cycle ;
					
					\draw  [color={rgb, 255:red, 139; green, 87; blue, 42 }  ,draw opacity=1 ][line width=1.5]  (312.3,232.69) .. controls (312.3,221.78) and (356.47,212.94) .. (410.96,212.94) .. controls (465.45,212.94) and (509.62,221.78) .. (509.62,232.69) .. controls (509.62,243.59) and (465.45,252.43) .. (410.96,252.43) .. controls (356.47,252.43) and (312.3,243.59) .. (312.3,232.69) -- cycle ;
					
					\draw  [color={rgb, 255:red, 0; green, 0; blue, 0 }  ,draw opacity=1 ][fill={rgb, 255:red, 208; green, 2; blue, 27 }  ,fill opacity=1 ][line width=1.5]  (273.08,154.65) .. controls (273.08,149.6) and (277.17,145.5) .. (282.22,145.5) .. controls (287.28,145.5) and (291.37,149.6) .. (291.37,154.65) .. controls (291.37,159.7) and (287.28,163.8) .. (282.22,163.8) .. controls (277.17,163.8) and (273.08,159.7) .. (273.08,154.65) -- cycle ;
					
					\draw  [color={rgb, 255:red, 0; green, 0; blue, 0 }  ,draw opacity=1 ][fill={rgb, 255:red, 208; green, 2; blue, 27 }  ,fill opacity=1 ][line width=1.5]  (147.91,179.5) .. controls (147.91,176.27) and (150.53,173.65) .. (153.76,173.65) .. controls (156.99,173.65) and (159.61,176.27) .. (159.61,179.5) .. controls (159.61,182.73) and (156.99,185.35) .. (153.76,185.35) .. controls (150.53,185.35) and (147.91,182.73) .. (147.91,179.5) -- cycle ;
					
					\draw  [color={rgb, 255:red, 0; green, 0; blue, 0 }  ,draw opacity=1 ][fill={rgb, 255:red, 208; green, 2; blue, 27 }  ,fill opacity=1 ][line width=1.5]  (19.45,154.65) .. controls (19.45,151.42) and (22.07,148.8) .. (25.3,148.8) .. controls (28.53,148.8) and (31.15,151.42) .. (31.15,154.65) .. controls (31.15,157.88) and (28.53,160.5) .. (25.3,160.5) .. controls (22.07,160.5) and (19.45,157.88) .. (19.45,154.65) -- cycle ;
					
					\draw  [color={rgb, 255:red, 0; green, 0; blue, 0 }  ,draw opacity=1 ][fill={rgb, 255:red, 208; green, 2; blue, 27 }  ,fill opacity=1 ][line width=1.5]  (147.91,128.09) .. controls (147.91,124.86) and (150.53,122.24) .. (153.76,122.24) .. controls (156.99,122.24) and (159.61,124.86) .. (159.61,128.09) .. controls (159.61,131.32) and (156.99,133.94) .. (153.76,133.94) .. controls (150.53,133.94) and (147.91,131.32) .. (147.91,128.09) -- cycle ;
					
					\draw  [color={rgb, 255:red, 0; green, 0; blue, 0 }  ,draw opacity=1 ][fill={rgb, 255:red, 208; green, 2; blue, 27 }  ,fill opacity=1 ][line width=1.5]  (404.84,179.5) .. controls (404.84,176.27) and (407.45,173.65) .. (410.69,173.65) .. controls (413.92,173.65) and (416.53,176.27) .. (416.53,179.5) .. controls (416.53,182.73) and (413.92,185.35) .. (410.69,185.35) .. controls (407.45,185.35) and (404.84,182.73) .. (404.84,179.5) -- cycle ;
					
					\draw  [color={rgb, 255:red, 0; green, 0; blue, 0 }  ,draw opacity=1 ][fill={rgb, 255:red, 208; green, 2; blue, 27 }  ,fill opacity=1 ][line width=1.5]  (533.3,154.65) .. controls (533.3,151.42) and (535.92,148.8) .. (539.15,148.8) .. controls (542.38,148.8) and (545,151.42) .. (545,154.65) .. controls (545,157.88) and (542.38,160.5) .. (539.15,160.5) .. controls (535.92,160.5) and (533.3,157.88) .. (533.3,154.65) -- cycle ;
					
					\draw  [color={rgb, 255:red, 0; green, 0; blue, 0 }  ,draw opacity=1 ][fill={rgb, 255:red, 208; green, 2; blue, 27 }  ,fill opacity=1 ][line width=1.5]  (404.84,128.94) .. controls (404.84,125.71) and (407.45,123.1) .. (410.69,123.1) .. controls (413.92,123.1) and (416.53,125.71) .. (416.53,128.94) .. controls (416.53,132.17) and (413.92,134.79) .. (410.69,134.79) .. controls (407.45,134.79) and (404.84,132.17) .. (404.84,128.94) -- cycle ;
					
					\draw  [color={rgb, 255:red, 0; green, 0; blue, 0 }  ,draw opacity=1 ][fill={rgb, 255:red, 139; green, 87; blue, 42 }  ,fill opacity=1 ][line width=1.5]  (135.11,251.86) .. controls (135.11,249.5) and (137.03,247.58) .. (139.39,247.58) .. controls (141.75,247.58) and (143.67,249.5) .. (143.67,251.86) .. controls (143.67,254.22) and (141.75,256.14) .. (139.39,256.14) .. controls (137.03,256.14) and (135.11,254.22) .. (135.11,251.86) -- cycle ;
					
					\draw  [color={rgb, 255:red, 0; green, 0; blue, 0 }  ,draw opacity=1 ][fill={rgb, 255:red, 139; green, 87; blue, 42 }  ,fill opacity=1 ][line width=1.5]  (237.11,242.86) .. controls (237.11,240.5) and (239.03,238.58) .. (241.39,238.58) .. controls (243.75,238.58) and (245.67,240.5) .. (245.67,242.86) .. controls (245.67,245.22) and (243.75,247.14) .. (241.39,247.14) .. controls (239.03,247.14) and (237.11,245.22) .. (237.11,242.86) -- cycle ;
					
					\draw  [color={rgb, 255:red, 0; green, 0; blue, 0 }  ,draw opacity=1 ][fill={rgb, 255:red, 139; green, 87; blue, 42 }  ,fill opacity=1 ][line width=1.5]  (165.11,214.86) .. controls (165.11,212.5) and (167.03,210.58) .. (169.39,210.58) .. controls (171.75,210.58) and (173.67,212.5) .. (173.67,214.86) .. controls (173.67,217.22) and (171.75,219.14) .. (169.39,219.14) .. controls (167.03,219.14) and (165.11,217.22) .. (165.11,214.86) -- cycle ;
					
					\draw  [color={rgb, 255:red, 0; green, 0; blue, 0 }  ,draw opacity=1 ][fill={rgb, 255:red, 139; green, 87; blue, 42 }  ,fill opacity=1 ][line width=1.5]  (64.11,223.86) .. controls (64.11,221.5) and (66.03,219.58) .. (68.39,219.58) .. controls (70.75,219.58) and (72.67,221.5) .. (72.67,223.86) .. controls (72.67,226.22) and (70.75,228.14) .. (68.39,228.14) .. controls (66.03,228.14) and (64.11,226.22) .. (64.11,223.86) -- cycle ;

					\draw  [color={rgb, 255:red, 0; green, 0; blue, 0 }  ,draw opacity=1 ][fill={rgb, 255:red, 0; green, 76; blue, 255 }  ,fill opacity=1 ][line width=1.5]  (144.61,25.33) .. controls (144.61,20.28) and (148.71,16.19) .. (153.76,16.19) .. controls (158.81,16.19) and (162.91,20.28) .. (162.91,25.33) .. controls (162.91,30.39) and (158.81,34.48) .. (153.76,34.48) .. controls (148.71,34.48) and (144.61,30.39) .. (144.61,25.33) -- cycle ;
					
					\draw  [color={rgb, 255:red, 0; green, 0; blue, 0 }  ,draw opacity=1 ][fill={rgb, 255:red, 0; green, 76; blue, 255 }  ,fill opacity=1 ][line width=1.5]  (401.54,25.33) .. controls (401.54,20.28) and (405.63,16.19) .. (410.69,16.19) .. controls (415.74,16.19) and (419.83,20.28) .. (419.83,25.33) .. controls (419.83,30.39) and (415.74,34.48) .. (410.69,34.48) .. controls (405.63,34.48) and (401.54,30.39) .. (401.54,25.33) -- cycle ;
					
					\draw  [color={rgb, 255:red, 0; green, 0; blue, 0 }  ,draw opacity=1 ][fill={rgb, 255:red, 0; green, 76; blue, 255 }  ,fill opacity=1 ][line width=1.5]  (144.61,282.26) .. controls (144.61,277.2) and (148.71,273.11) .. (153.76,273.11) .. controls (158.81,273.11) and (162.91,277.2) .. (162.91,282.26) .. controls (162.91,287.31) and (158.81,291.4) .. (153.76,291.4) .. controls (148.71,291.4) and (144.61,287.31) .. (144.61,282.26) -- cycle ;
					
					\draw  [color={rgb, 255:red, 0; green, 0; blue, 0 }  ,draw opacity=1 ][fill={rgb, 255:red, 0; green, 76; blue, 255 }  ,fill opacity=1 ][line width=1.5]  (401.54,282.26) .. controls (401.54,277.2) and (405.63,273.11) .. (410.69,273.11) .. controls (415.74,273.11) and (419.83,277.2) .. (419.83,282.26) .. controls (419.83,287.31) and (415.74,291.4) .. (410.69,291.4) .. controls (405.63,291.4) and (401.54,287.31) .. (401.54,282.26) -- cycle ;
					
					\draw  [color={rgb, 255:red, 0; green, 0; blue, 0 }  ,draw opacity=1 ][fill={rgb, 255:red, 139; green, 87; blue, 42 }  ,fill opacity=1 ][line width=1.5]  (165.11,93.86) .. controls (165.11,91.5) and (167.03,89.58) .. (169.39,89.58) .. controls (171.75,89.58) and (173.67,91.5) .. (173.67,93.86) .. controls (173.67,96.22) and (171.75,98.14) .. (169.39,98.14) .. controls (167.03,98.14) and (165.11,96.22) .. (165.11,93.86) -- cycle ;
					
					\draw  [color={rgb, 255:red, 0; green, 0; blue, 0 }  ,draw opacity=1 ][fill={rgb, 255:red, 139; green, 87; blue, 42 }  ,fill opacity=1 ][line width=1.5]  (234.11,67.86) .. controls (234.11,65.5) and (236.03,63.58) .. (238.39,63.58) .. controls (240.75,63.58) and (242.67,65.5) .. (242.67,67.86) .. controls (242.67,70.22) and (240.75,72.14) .. (238.39,72.14) .. controls (236.03,72.14) and (234.11,70.22) .. (234.11,67.86) -- cycle ;
					
					\draw  [color={rgb, 255:red, 0; green, 0; blue, 0 }  ,draw opacity=1 ][fill={rgb, 255:red, 139; green, 87; blue, 42 }  ,fill opacity=1 ][line width=1.5]  (136.11,56.86) .. controls (136.11,54.5) and (138.03,52.58) .. (140.39,52.58) .. controls (142.75,52.58) and (144.67,54.5) .. (144.67,56.86) .. controls (144.67,59.22) and (142.75,61.14) .. (140.39,61.14) .. controls (138.03,61.14) and (136.11,59.22) .. (136.11,56.86) -- cycle ;
					
					\draw  [color={rgb, 255:red, 0; green, 0; blue, 0 }  ,draw opacity=1 ][fill={rgb, 255:red, 139; green, 87; blue, 42 }  ,fill opacity=1 ][line width=1.5]  (59.11,83.86) .. controls (59.11,81.5) and (61.03,79.58) .. (63.39,79.58) .. controls (65.75,79.58) and (67.67,81.5) .. (67.67,83.86) .. controls (67.67,86.22) and (65.75,88.14) .. (63.39,88.14) .. controls (61.03,88.14) and (59.11,86.22) .. (59.11,83.86) -- cycle ;
					
					\draw  [color={rgb, 255:red, 0; green, 0; blue, 0 }  ,draw opacity=1 ][fill={rgb, 255:red, 139; green, 87; blue, 42 }  ,fill opacity=1 ][line width=1.5]  (393.11,250.86) .. controls (393.11,248.5) and (395.03,246.58) .. (397.39,246.58) .. controls (399.75,246.58) and (401.67,248.5) .. (401.67,250.86) .. controls (401.67,253.22) and (399.75,255.14) .. (397.39,255.14) .. controls (395.03,255.14) and (393.11,253.22) .. (393.11,250.86) -- cycle ;
					
					\draw  [color={rgb, 255:red, 0; green, 0; blue, 0 }  ,draw opacity=1 ][fill={rgb, 255:red, 139; green, 87; blue, 42 }  ,fill opacity=1 ][line width=1.5]  (495.11,241.86) .. controls (495.11,239.5) and (497.03,237.58) .. (499.39,237.58) .. controls (501.75,237.58) and (503.67,239.5) .. (503.67,241.86) .. controls (503.67,244.22) and (501.75,246.14) .. (499.39,246.14) .. controls (497.03,246.14) and (495.11,244.22) .. (495.11,241.86) -- cycle ;
					
					\draw  [color={rgb, 255:red, 0; green, 0; blue, 0 }  ,draw opacity=1 ][fill={rgb, 255:red, 139; green, 87; blue, 42 }  ,fill opacity=1 ][line width=1.5]  (423.11,213.86) .. controls (423.11,211.5) and (425.03,209.58) .. (427.39,209.58) .. controls (429.75,209.58) and (431.67,211.5) .. (431.67,213.86) .. controls (431.67,216.22) and (429.75,218.14) .. (427.39,218.14) .. controls (425.03,218.14) and (423.11,216.22) .. (423.11,213.86) -- cycle ;
					
					\draw  [color={rgb, 255:red, 0; green, 0; blue, 0 }  ,draw opacity=1 ][fill={rgb, 255:red, 139; green, 87; blue, 42 }  ,fill opacity=1 ][line width=1.5]  (322.11,222.86) .. controls (322.11,220.5) and (324.03,218.58) .. (326.39,218.58) .. controls (328.75,218.58) and (330.67,220.5) .. (330.67,222.86) .. controls (330.67,225.22) and (328.75,227.14) .. (326.39,227.14) .. controls (324.03,227.14) and (322.11,225.22) .. (322.11,222.86) -- cycle ;
					
					\draw  [color={rgb, 255:red, 0; green, 0; blue, 0 }  ,draw opacity=1 ][fill={rgb, 255:red, 139; green, 87; blue, 42 }  ,fill opacity=1 ][line width=1.5]  (422.11,91.86) .. controls (422.11,89.5) and (424.03,87.58) .. (426.39,87.58) .. controls (428.75,87.58) and (430.67,89.5) .. (430.67,91.86) .. controls (430.67,94.22) and (428.75,96.14) .. (426.39,96.14) .. controls (424.03,96.14) and (422.11,94.22) .. (422.11,91.86) -- cycle ;
					
					\draw  [color={rgb, 255:red, 0; green, 0; blue, 0 }  ,draw opacity=1 ][fill={rgb, 255:red, 139; green, 87; blue, 42 }  ,fill opacity=1 ][line width=1.5]  (491.11,65.86) .. controls (491.11,63.5) and (493.03,61.58) .. (495.39,61.58) .. controls (497.75,61.58) and (499.67,63.5) .. (499.67,65.86) .. controls (499.67,68.22) and (497.75,70.14) .. (495.39,70.14) .. controls (493.03,70.14) and (491.11,68.22) .. (491.11,65.86) -- cycle ;
					
					\draw  [color={rgb, 255:red, 0; green, 0; blue, 0 }  ,draw opacity=1 ][fill={rgb, 255:red, 139; green, 87; blue, 42 }  ,fill opacity=1 ][line width=1.5]  (393.11,54.86) .. controls (393.11,52.5) and (395.03,50.58) .. (397.39,50.58) .. controls (399.75,50.58) and (401.67,52.5) .. (401.67,54.86) .. controls (401.67,57.22) and (399.75,59.14) .. (397.39,59.14) .. controls (395.03,59.14) and (393.11,57.22) .. (393.11,54.86) -- cycle ;
					
					\draw  [color={rgb, 255:red, 0; green, 0; blue, 0 }  ,draw opacity=1 ][fill={rgb, 255:red, 139; green, 87; blue, 42 }  ,fill opacity=1 ][line width=1.5]  (316.11,81.86) .. controls (316.11,79.5) and (318.03,77.58) .. (320.39,77.58) .. controls (322.75,77.58) and (324.67,79.5) .. (324.67,81.86) .. controls (324.67,84.22) and (322.75,86.14) .. (320.39,86.14) .. controls (318.03,86.14) and (316.11,84.22) .. (316.11,81.86) -- cycle ;
					
					\draw  [fill={rgb, 255:red, 208; green, 2; blue, 27 }  ,fill opacity=1 ][line width=1.5]  (236.61,129.82) -- (241.44,134.65) -- (236.61,139.49) -- (231.78,134.65) -- cycle ;
					
					\draw  [fill={rgb, 255:red, 208; green, 2; blue, 27 }  ,fill opacity=1 ][line width=1.5]  (236.11,168.82) -- (240.94,173.65) -- (236.11,178.49) -- (231.28,173.65) -- cycle ;
					
					\draw  [fill={rgb, 255:red, 208; green, 2; blue, 27 }  ,fill opacity=1 ][line width=1.5]  (91.11,127.32) -- (95.94,132.15) -- (91.11,136.99) -- (86.28,132.15) -- cycle ;
					
					\draw  [fill={rgb, 255:red, 208; green, 2; blue, 27 }  ,fill opacity=1 ][line width=1.5]  (213.11,126.32) -- (217.94,131.15) -- (213.11,135.99) -- (208.28,131.15) -- cycle ;
					
					\draw  [fill={rgb, 255:red, 208; green, 2; blue, 27 }  ,fill opacity=1 ][line width=1.5]  (90.61,171.82) -- (95.44,176.65) -- (90.61,181.49) -- (85.78,176.65) -- cycle ;
					
					\draw  [fill={rgb, 255:red, 208; green, 2; blue, 27 }  ,fill opacity=1 ][line width=1.5]  (212.61,171.32) -- (217.44,176.15) -- (212.61,180.99) -- (207.78,176.15) -- cycle ;
					
					\draw  [fill={rgb, 255:red, 208; green, 2; blue, 27 }  ,fill opacity=1 ][line width=1.5]  (66.61,130.32) -- (71.44,135.15) -- (66.61,139.99) -- (61.78,135.15) -- cycle ;
					
					\draw  [fill={rgb, 255:red, 208; green, 2; blue, 27 }  ,fill opacity=1 ][line width=1.5]  (66.11,167.32) -- (70.94,172.15) -- (66.11,176.99) -- (61.28,172.15) -- cycle ;
					
					\draw  [fill={rgb, 255:red, 208; green, 2; blue, 27 }  ,fill opacity=1 ][line width=1.5]  (497.61,130.82) -- (502.44,135.65) -- (497.61,140.49) -- (492.78,135.65) -- cycle ;
					
					\draw  [fill={rgb, 255:red, 208; green, 2; blue, 27 }  ,fill opacity=1 ][line width=1.5]  (497.11,169.82) -- (501.94,174.65) -- (497.11,179.49) -- (492.28,174.65) -- cycle ;
					
					\draw  [fill={rgb, 255:red, 208; green, 2; blue, 27 }  ,fill opacity=1 ][line width=1.5]  (352.11,128.32) -- (356.94,133.15) -- (352.11,137.99) -- (347.28,133.15) -- cycle ;
					
					\draw  [fill={rgb, 255:red, 208; green, 2; blue, 27 }  ,fill opacity=1 ][line width=1.5]  (474.11,127.32) -- (478.94,132.15) -- (474.11,136.99) -- (469.28,132.15) -- cycle ;
					
					\draw  [fill={rgb, 255:red, 208; green, 2; blue, 27 }  ,fill opacity=1 ][line width=1.5]  (351.61,172.82) -- (356.44,177.65) -- (351.61,182.49) -- (346.78,177.65) -- cycle ;
					
					\draw  [fill={rgb, 255:red, 208; green, 2; blue, 27 }  ,fill opacity=1 ][line width=1.5]  (473.61,172.32) -- (478.44,177.15) -- (473.61,181.99) -- (468.78,177.15) -- cycle ;
					
					\draw  [fill={rgb, 255:red, 208; green, 2; blue, 27 }  ,fill opacity=1 ][line width=1.5]  (327.61,131.32) -- (332.44,136.15) -- (327.61,140.99) -- (322.78,136.15) -- cycle ;
					
					\draw  [fill={rgb, 255:red, 208; green, 2; blue, 27 }  ,fill opacity=1 ][line width=1.5]  (327.11,168.32) -- (331.94,173.15) -- (327.11,177.99) -- (322.28,173.15) -- cycle ;
					
					\draw (140.49,296.1) node [anchor=north west][inner sep=0.75pt]  [font=\Large]  {$D_{a}$};
					
					\draw (398.19,294.63) node [anchor=north west][inner sep=0.75pt]  [font=\Large]  {$D_{b}$};

				\end{tikzpicture}
			}
		\end{center} 
		\caption{This figure refers to Example 2. The Bloch sphere is so oriented that the $y$-axis is pointing upwards. We see three $\Pi$-minimal sets in $\pc{3}$: the set of four brown circles (generic),  of two red equators,  and of four poles (cf.  \Cref{fig:ex2pic1}). If $G$ is infinite,  the uniform measures on these sets are   the $\Pi$-ergodic measures. 
			If $G$ is finite, all $\Pi$-ergodic measures are finitely-supported (here we assume the projective action of $G$ has 8 elements, which corresponds to setting, e.g.,   $\theta = \pi/2 $ and $\phi = \pi/4$). Brown bullets: a generic $\Pi$-ergodic measure (uniform, each atom with weight $1/16$).   Blue bullets: the ergodic measure supported on the poles (uniform, each atom with weight $1/4$). Red bullets: a non-uniform $\Pi$-ergodic measure (the atom at the touching point  has weight $1/4$,  the other atoms have weight $1/8$). Red diamonds: a    $\Pi$-ergodic measure (uniform, each atom with weight $1/16$).}\label{fig:ex2pic2} 
	\end{figure}

	\subsection*{Example 3.}  Here we recall a classical example of a measure $\mu$ with uncountably many dark subspaces \cite[Example~2]{MaaKumm} and describe its ergodic measures by using our framework of smart families. 
	
	Let $\ell, s,r \in \nnone$. We fix a family of complex $s\times s$ matrices $b_1, \ldots, b_\ell$ such that they admit only dark subspaces of dimension $1$ and satisfy the stochasticity condition $\sum_{i=1}^\ell b_i^*b_i = \id$.  
	We extend the system by fixing a family $u_1, \ldots, u_\ell$ of special unitary operators on $\cc^r$. We shall consider the joint space $\cc^s \otimes \cc^r$  with the operators
	\[v_i = b_i  \otimes u_i, \ \ i \in \{1, \ldots, \ell\}.\]
	Then $\sum_{i=1}^\ell v_i^*v_i = \id$. We assume that $ \mu=\sum_{i=1}^\ell \delta_{v_i}$ is irreducible. Note that the irreducibility of 
	$ \sum_{i=1}^\ell \delta_{b_i}$ follows. 
	
	\smallskip 
	
	The maximal dark subspaces of the joint system are $r$-dimensional and contain the elements of the set  
	\[\mathcal D_\infty=\{y \otimes \cc^r \:|\: y \in \cc^s, \: y\neq 0\}.\]
	This is an $\W$-invariant set on which the process of dark subspaces concentrates.
	Indeed, observe that the process $(M_n)_n$ can be written $M_n=M_n^1\otimes M_n^2$ with $M_n^1$ and $M_n^2$ operators on, respectively, $\cc^s$ and $\cc^r$. The assumption of triviality of the dark subspaces for the $\{b_1,\dotsc,b_\ell\}$ and \Cref{prop:rankM} guarantee that $M_n^1$ converges to some rank-one projector $\pi_{\hat z}$ for a random $z \in \cc^s$ as $n$ tends to infinity, while on $\cc^r$ the operators are unitary, so $M_n^2 = \id_r/r$ for every $n$. As a consequence, $M_\infty = \pi_{\hat z} \otimes \id_r/r$. Then, since for any product $w$ of matrices from $\{v_1,\dotsc,v_\ell\}$, $w=b\otimes u$ for some $s\times s$ matrix $b$ and a unitary $r\times r$ matrix $u$, $U_n$ in the polar decomposition is of the form $\tilde u_n\otimes (u_{i_n}\dotsb u_{i_1})$ with $\tilde u_n$ an $s\times s$ unitary matrix. Hence, every maximal dark subspace that is charged asymptotically coincides with the range of some $\pi_{U \cdot \hat z_0} \otimes \id_r$ with $U \in \mathcal{SU}(s)$ thanks to \Cref{prop:Minfdark}. 
	
	We now restrict our discussion to dark subspaces from $\mathcal D_\infty$.
	
	By \Cref{thm:invmeasDark}, the kernel  $K$  on $\mathcal D_m$, defined by
	\[Kf(D)=\sum_{i=1}^\ell f(v_iD)\operatorname{tr}\!\big(v_i \tfrac{\pi_D}{r}v_i^*\big),\]
	where $f \in \mathcal C (\mathcal D_m)$, 
	admits a unique invariant probability measure  $\chi_{{\rm inv}}$.  
	Also, by \cite[Theorem~1.1]{BenFra}, the transition kernel $\Pi_b$ on $\pc{s}$, defined by \[\Pi_b  f(\hat y) = \sum_{i=1}^\ell  f(b_i \cdot \hat y)\|b_i y\|^2,\] where $ f \in \mathcal C (\pc{s})$, admits a unique invariant measure, which we denote by $\gamma_{{\rm inv}}$.  
	We shall now establish the connection between $\chi_{{\rm inv}}$ and $\gamma_{{\rm inv}}$.
	
	The dark subspaces of $\mathcal D_\infty$ can be associated with rays in $\cc^{s}$ via
	\[\xi \colon  y \otimes \cc^r \mapsto  \hat y  \in \pc{s}.\]
	For $\hat y \in \pc{s}$ we denote $D_{\hat y} = y \otimes \cc^r$.  
	Observe that $v_i  D_{\hat y} =
	D_{b_i \cdot \hat y}$ if $b_i y \neq 0$ and also
	\[\operatorname{tr}\!\big(v_i {\pi_{D_{\hat y}}}  v_i^*\big) =    \tr(b_i \pi_{{\hat y}} b_i^* \otimes u_i \id_{\cc^r} u_i^*) = r \|b_i y\|^2.\]
	Thus, by direct calculation, 
	\[\Pi_b f (\hat y) = K[f\circ \xi](D_{\hat y})\]
	for any $f \in \mathcal C (\pc{s})$. 
	From the $K$-invariance of $\chi_{{\rm inv}}$ we deduce the $ \Pi_b$-invariance of $\xi_\star\chi_{{\rm inv}}$ and, since $\gamma_{{\rm inv}}$ is the unique $\Pi_b$-invariant  probability measure, we conclude that 
	\begin{equation}\label{eq:ex3imgmeasure}
		\gamma_{{\rm inv}} = \xi_\star\chi_{{\rm inv}}.
	\end{equation}

	\smallskip
	
	Let us apply the  minimal family approach to 
	describe the set of all $\Pi$-ergodic measures on $\P(\cc^s \otimes \cc^r)$, where we recall that $\Pi$ acts on $f \in \mathcal C (\P(\cc^s \otimes \cc^r))$ as  
	\begin{align*}
		\Pi f(\hat x) 
		= \sum_{i=1}^\ell    f(v_i \cdot \hat x)\|v_i x\|^2.   
	\end{align*} 
	The convergence of quantum trajectories towards dark subspaces allows us to restrict our attention to product states of the form $\hat x=\hat y \otimes \hat z$ with $y\in \cc^s$ and $z\in \cc^r$, where we write $\hat y\otimes \hat z$ for $\widehat{y\otimes z}$. On this set of states,
	\begin{align*}
		\Pi f(\hat y \otimes \hat z) = \sum_{i=1}^\ell    f(b_i \cdot \hat y \otimes u_i \cdot \hat z)\|b_i y\|^2.  
	\end{align*}
	Denote by $G$  the smallest closed subgroup of $\mathcal {SU}(r)$ containing $\{u_1, \ldots, u_\ell\}$, and by $\operatorname{H}_G$ the normalized Haar measure on $G$. Consider the image through $\Psi=(\hat y,\hat z)\mapsto \hat y\otimes \hat z$ of the measures 
	\[\gamma_{{\rm inv}}\otimes({g_{\hat z}})_\star \mathrm{H}_G,\]
	where  $\hat z \in \pc{r}$ and $g_{\hat z} \colon \mathcal{SU}(r) \ni u \mapsto u \cdot \hat z \in \pc{r}$.  
	We make no notational distinction between a measure on $\P(\cc^s)\times\P(\cc^r)$ and its pushforward under $\Psi$.
	Their $\Pi$-invariance can be verified by direct calculation. In what follows, we show that, quite counter-intuitively, they  may or may not form the set of $\Pi$-ergodic measures. 
	
	First, remark that the family of isometries $\widetilde{\J}$ defined by $\widetilde{J}_{\hat y}\colon x\mapsto y\otimes x$, where $y$ denotes a normalized representative of $\hat y$, satisfies $G_{\widetilde{\J}}=G$. We show that this family may or may not be minimal through an example. Let $s=r =2$, $q\in (0,1/2)$  and
	\[b_1=\begin{pmatrix}
		\sqrt{q} &0\\
		0&\sqrt{1-q}
	\end{pmatrix},\quad b_2=\begin{pmatrix}
		0&\sqrt{q}\\
		\sqrt{1-q}&0
	\end{pmatrix},\quad u_1=\i\sigma_x\quad \mbox{and}\quad u_2=\i\sigma_z.\]
	Then, $G=\{\pm\id,\pm \i\sigma_x,\pm\i\sigma_y,\pm\i\sigma_z\}$. Computing explicitly the eigenvectors of $v_1=b_1\otimes u_1$ and $v_2=b_2\otimes u_2$, one checks that $\mu$ is irreducible.
	
	Following \cite[Section~8]{BenFat}, $\gamma_{\rm inv}=q\delta_{\hat e_0}+(1-q)\delta_{\hat e_1}$ and the only maximal dark subspaces charged by $\chi_{\rm inv}$ are $e_0\otimes \cc^2$ and $e_1\otimes \cc^2$ with $\{e_0,e_1\}$ the canonical basis of $\cc^2$. Let $J_{\hat e_0}\colon z\mapsto e_0\otimes z$ and $J_{\hat e_1}=v_2J_{\hat e_0}/\sqrt{1-q}$. Then $J_{\hat e_1}z=e_1\otimes \i\sigma_z z$ and by construction $\J=\{J_{\hat e_0}, J_{\hat e_1}\}$ (completed with arbitrary isometries for other maximal dark subspaces) is $D_{\hat e_0}$-smart, thus smart and minimal thanks to \Cref{thm:smartiffminimal}. Using composition rules, to determine $G_\J$ we only need to compute four unitary matrices that will generate $G_\J$. Direct computations lead to
	\begin{align*}
		J_{\hat e_0}^{-1}v_1J_{\hat e_0}^{\vphantom{-1}}&\propto\i\sigma_x &
		J_{\hat e_1}^{-1}v_2J_{\hat e_0}^{\vphantom{-1}}&\propto\id\\
		J_{\hat e_1}^{-1}v_1J_{\hat e_1}^{\vphantom{-1}}&\propto-\i\sigma_x&
		J_{\hat e_0}^{-1}v_2J_{\hat e_1}^{\vphantom{-1}}&\propto-\id.
	\end{align*}
	Therefore, $G_\J=\{\pm\id,\pm\i\sigma_x\}\subsetneq G$. Hence, $\widetilde{\J}$ is not minimal and $\Upsilon_{\widetilde{\J}}$ contains non-ergodic $\Pi$-invariant measures. 
	In particular, \[\nu_{\hat x,\J}=\frac{q}{2}(\delta_{\hat e_0\otimes \hat x}+ \delta_{\hat e_0\otimes \sigma_x\cdot\hat x})+\frac{1-q}2(\delta_{\hat e_1\otimes \sigma_y\cdot \hat x}+\delta_{\hat e_1\otimes \sigma_z\cdot\hat x})\]
	is $\Pi$-ergodic, yet, for most $\hat x$, it is not in $\Upsilon_{\widetilde{\J}}$. Indeed, the unique measure in $\Upsilon_\tJ$ 
	that contains $\hat e_0\otimes \hat x$ in its support is 
	\[\nu_{\hat x,\tJ}=\frac{q}{4}(\delta_{\hat e_0\otimes \hat x}+\delta_{\hat e_0\otimes \sigma_x\cdot\hat x}+\delta_{\hat e_0\otimes \sigma_y\cdot\hat x}+
	\delta_{\hat e_0\otimes \sigma_z\cdot\hat x})+\frac{1-q}{4}(\delta_{\hat e_1\otimes \hat x}+\delta_{\hat e_1\otimes \sigma_x\cdot\hat x}+\delta_{\hat e_1\otimes \sigma_y\cdot\hat x}+\delta_{\hat e_1\otimes \sigma_z\cdot\hat x}).\]
	So, if $x$ is not an eigenvector of $\sigma_y$ or $\sigma_z$, then  $\nu_{\hat x,\J}$ is not in $\Upsilon_\tJ$. Moreover, $\nu_{\hat x,\tJ}=\frac12\nu_{\hat x,\J}+\frac12\nu_{\sigma_y\cdot\hat x,\J}$, which proves  it is not $\Pi$-ergodic.
	
	Now, let us introduce a new set of matrices $\{{v_1}/\sqrt{2},{v_2}/\sqrt{2},{v_3}/\sqrt{2}\}$ with the same $v_1$ and $v_2$ and adding $v_3=\id\otimes \i\sigma_y$. The group $G$ generated by the unitary operators $u_i$ is not changed by this addition. The new invariant measure $\chi_{\rm inv}$ is supported on the same two maximal dark subspaces. Keeping the same definition of $\J$, it is still minimal since it is still $D_{\hat e_0}$-smart. However, the generator $\i\sigma_y$ is added to $G_\J$ since
	\begin{align*}
		J_{\hat e_0}^{-1}v_3J_{\hat e_0}^{\vphantom{-1}}=J_{\hat e_1}^{-1}v_3J_{\hat e_1}^{\vphantom{-1}}\propto\i\sigma_y.
	\end{align*}
	It follows that $G_\J=G$ and $\Upsilon_{\widetilde{\J}}=\Upsilon_{\J \vphantom{\widetilde{\J}}}$.

	\appendix
	
	\section{Invariant measures on mixed states}\label{sec:app}
	
	In view of the results of \Cref{sec:invmeasures},  there naturally arises a question about $\Pi$-invariant measures on  $\Sd$. We immediately see that, if $r_m\geq 2$, uniqueness can no longer hold as we always have $\nu_{\rm unif} = \int_{\D} \operatorname{Unif}_{\P(D)}\d \chi_{\rm inv}(D)$ on pure states and $\nu_{\rm ch}= \int_{\D} \delta_{\pi_D / r_m} \d \chi_{\rm inv}(D)$ on (normalized) maximal dark projectors. 
	In what follows, we provide the characterization of $\Pi$-invariant probability measures on $\Sd$. Throughout the appendix, for $D \in \mathcal D_m$ we denote  $\mathcal{S}(D) = \{\rho \in \Sd \:|\: \supp \rho \subset D\}$. 
	
	\smallskip 
	
	First, we observe that the convergence of quantum trajectories towards dark subspaces (\Cref{thm:conv to Dark}) gives the following counterpart of \Cref{prop:invariantindark}:
	\begin{proposition} 
		The support of any $\Pi$-invariant Borel probability measure on $\Sd$ is contained in $\bigcup_{D \in \mathcal D_m}\mathcal{S}(D)$. 
	\end{proposition}
	
	Now, we fix a family of linear isometries $\J =\{J_D\}_{D \in \mathcal D_m}$ such that $J_D\colon \cc^{r_m} \to D$. For each $D \in \mathcal D_m$ we define the map 
	\[\begin{array}{rcl} T_D\colon\ \Srm&\longrightarrow&\mathcal{S}(D) \\[0.33em] \rho &\longmapsto& J_D^{\vphantom{*}} \rho J_D^*. \end{array}\]
	Note that
	$\operatorname{tr}(vT_D(\rho)v^*) = \operatorname{tr}\!\big(v\tfrac{\pi_D}{r_m}v^*\big)$ and   
	\begin{equation}\label{eq:mixed}
		T^{-1}_{vD}\Bigg(\frac{vT_D(\rho)v^*}{\operatorname{tr}\!\big(v\tfrac{\pi_D}{r_m}v^*\big)}\Bigg) =  u_{v,D}^{\hphantom{*}} \rho u_{v,D}^*
	\end{equation}
	for $v \in \supp \mu$ and $D \in \D$ such that $\operatorname{tr}(v\pi_D v^*)>0$ since $u_{v,D}\propto J_{vD}^{-1}vJ_D$.

	We define the group $G_\J$ as before, i.e., as 
	the smallest closed subgroup of $\mathcal {SU}(r_m)$ containing the set \[
	\{u_{v,D} \in {\mathcal {SU}(r_m)} \,|\, v \in \supp\mu,\: D \in \supp\chi_{{\rm inv}},\: \operatorname{tr}(v \pi_D v^*)>0\}.\] 
	For $\rho \in \Sr$ we define
	\[m_{\rho,\J} = (g_{\rho})_\star{\rm H}_{G_\J},\] where  ${\rm H}_{G_\J}$ is
	the normalized Haar measure on $G_\J$ and  
	\[\begin{array}{rcl} g_{\rho}\colon\ G_\J&\longrightarrow&\Sr\\[0.33em] u &\longmapsto& u\rho u^*. \end{array}\]
	Then we set
	\[\nu_{\rho,\J} = (\Psi_\J)_\star(\chi_{{\rm inv}} \otimes m_{\rho, \J}),\]
	where 
	\[\begin{array}{rcl}\Psi_\J\colon\  \mathcal D_m \times \Sr& \longrightarrow&\bigcup_{D \in \mathcal D_m}\mathcal{S}(D) \subset \Sd \\[0.33em](D,\rho) &\longmapsto& T_D (\rho). \end{array}\]
	Explicitly, for $f\in \mathcal C(\Sd)$,
	\begin{align*}
		\int_{ \Sd }f(\sigma)\d \nu_{\rho,\J}  (\sigma)
		& 
		=
		\int_{\mathcal D_m}\int_{G_\J} f (T_D(u \rho u^*)) \d\mathrm{H}_{G_\J}  (u) \d\chi_{\rm inv}(D). 
	\end{align*} 
	\begin{proposition}
		$\nu_{\rho,\J}$ is $\Pi$-invariant.
	\end{proposition}   
	\begin{proof}
		Recall that the invariance of Haar measure  guarantees that
		\[\int_{G_\J} h(\tilde u  u \rho u^* \tilde u^*) \,\mathrm{dH}_{G_\J}(u)  = \int_{G_\J} h(   u \rho u^* ) \,\mathrm{dH}_{G_\J}( u)\]
		for every $\tilde u \in G_\J$ and $h\in \mathcal C(\Sr)$. Using  also \Cref{eq:mixed} and the  $K$-invariance of $\chi_{{\rm inv}}$,  we obtain 
		\begin{align*}
			\int_{ \Sd }\Pi & f( \sigma) \,\mathrm{d}  \nu_{\rho,\J} (\sigma) 
			\\ &  =
			\int_{\Ld} \int_{\mathcal D_m} \left[\int_{G_\J} (f\circ T_{vD}) \big(u_{v,D}^{\vphantom{*}} u^{\vphantom{*}} \rho u^*u_{v,D}^* \big) \,\mathrm{dH}_{G_\J}(u)\right] \tr(v\tfrac{\pi_D}{r_m}v^*)\d\chi_{\rm inv}(D)\d\mu(v)
			\\   & =
			\int_{\Ld} \int_{\mathcal D_m} \left[\int_{G_\J} (f\circ T_{vD}) (u \rho u^*) \,\mathrm{dH}_{G_\J}(u)\right] \tr(v\tfrac{\pi_D}{r_m}v^*)\d\chi_{\rm inv}(D)\d\mu(v)
			\\  & =
			\int_{\mathcal D_m}\int_{G_\J} (f\circ T_{D}) (u \rho u^*) \, \mathrm{dH}_{G_\J}(u)  \mathrm{d}[\chi_{\rm inv}K](D)
			\\  & =   \int_{ \Sd }  f( \sigma) \, \mathrm{d} \nu_{\rho,\J} (\sigma)
		\end{align*}          
		for every $f\in \mathcal C(\Sd)$, as desired.   
	\end{proof}

	Therefore, $\Upsilon_{\J}  = \{\nu_{\rho,\J} \colon \rho \in \Sr\}$ is a set of $\Pi$-invariant measures.         
	Following the same strategy as for $\pcd$, we can show via Raugi's theorem that $\Upsilon_\J$ coincides with the set of all $\Pi$-ergodic measures on $\Sd$ if $\J$  is a minimal family. To reach this conclusion, we need to check the equicontinuity of $\Pi$ on $\Sd$ and characterize the $\Pi$-minimal subsets of $\Sd$. 
	
	\smallskip 
	
	For the latter task, the following result is a counterpart of \Cref{thm:piminimalsets}. 
	We do not provide the proof as it is fully analogous to that for pure states, up to writing $u\rho u^*$ instead of $u \cdot \hat x$ on the reference space and $w \tilde\rho  w^*/\tr(w \tilde \rho w^*)$ instead of $w \cdot \hat y$ on dark subspaces. In the $(\Rightarrow)$ part of the proof, the family of sets $X_D = T_D^{-1}(F \cap \mathcal{S}(D))\subset \Sr$ over $D \in \mathcal D_m$ is to be examined, where $F$ is a $\Pi$-minimal subset of $\Sd$.
	\begin{theorem} 
		Let $\J$ be a minimal family. Then every $\Pi$-minimal subset of $\Sd$ is of the form $\supp \nu_{\rho,\J}$ for some $\rho \in \Sr$.   
	\end{theorem}

	For the equicontinuity of $\Pi$, we need the following two lemmas, which are counterparts of \Cref{lem:eqM,lem:eqVn}, respectively. 
	\begin{lemma}\label{lem:A2} 
		Let  $f \in \mathcal C(\Sd)$ be  Lipschitz. Let  $\rho_1,\rho_2 \in \Sd$ and $n \in \nnone$. Then
		\[\int_{W_n \in A} \left| f\left(\tfrac{W_n \rho_1 W_n^*}{\tr(W_n \rho_1 W_n^*)}\right) - f\left(\tfrac{W_n \rho_2 W_n^*}{\tr(W_n \rho_2 W_n^*)}\right)\right|  {\tr(W_n \rho_1 W_n^*)} \d\mutensorn
		\leq  C  \|\rho_1 - \rho_2\|_{\rm tr},\]
		where $A = \{w \in \V_n \:|\: {\tr(w \rho_2 w^*)} \geq  {\tr(w \rho_1 w^*)}\}$ and
		$C>0$ is a constant depending on the Lipschitz constant of $f$ and the dimension $d$ of the system.
	\end{lemma}
	\begin{proof}  
		Fix \(w\in\mathcal M_d\) such that $ \operatorname{tr}(w\rho_2w^*) \geq  \operatorname{tr}(w\rho_1w^*)>0$ and observe that 
		\begin{align} \nonumber 
			\Big \|  \tfrac{w \rho_1 w^*}{\operatorname{tr}(w \rho_1 w^*)}   -  \tfrac{w \rho_2 w^*}{\operatorname{tr}(w \rho_2 w^*)}\Big \|_{\rm tr}   \operatorname{tr}(w \rho_1 w^*)   \operatorname{tr}(w \rho_2 w^*) 
			\leq \ &   {\left \|  {w \rho_1 w^*}{\operatorname{tr}(w \rho_2 w^*)} -  {w \rho_2 w^*}{\operatorname{tr}(w \rho_2 w^*)}\right \|_{\rm tr}  }   \nonumber   
			\\[0.5em]  \nonumber    & +  { \left \|{w \rho_2 w^*}{\operatorname{tr}(w \rho_1 w^*)} -   {w \rho_2 w^*}{\operatorname{tr}(w \rho_2 w^*)}\right \|_{\rm tr}}  
			\\[0.75em]  = \  &   \left \|  {w (\rho_1 - \rho_2) w^*}{}     \right \|_{\rm tr} \operatorname{tr}(w \rho_2 w^*)  \label{eq:app1a}
			\\[0.5em] &+  \big|\tr(w^*w (\rho_1 - \rho_2) )  \big|\: \| w \rho_2 w^* \|_{\rm tr}.   \label{eq:app1b}
		\end{align}  In what follows, we bound the two terms in the last sum above separately.
		
		For \Cref{eq:app1a}, H\"{o}lder's inequality for Schatten norms used twice implies $\|w(\rho_1-\rho_2)w^*\|_{\tr}\leq \|\rho_1-\rho_2\|_{\tr} \|w\|^2$. Then, the inequality $\|w\|^2\leq \tr(w^*w)$ implies
		\begin{align}\label{eq:app1aa}
			\left \|  {w (\rho_1 - \rho_2) w^*}{}     \right \|_{\rm tr}  \operatorname{tr}(w \rho_2 w^*)  
			\leq   \|\rho_1 - \rho_2\|_{\rm tr} \operatorname{tr}(w^*w)  \operatorname{tr}(w \rho_2 w^*).  
		\end{align} 
		
		For \Cref{eq:app1b}, matrix H\"{o}lder's inequality and the domination of the sup norm by the trace norm lead to
		\begin{align}\label{eq:app1bb}
			\big|\tr(w^*w (\rho_1 - \rho_2) )  \big| \: \| w \rho_2 w^* \|_{\rm tr}   
			\leq  \|\rho_1 - \rho_2\|_{\rm tr}  \operatorname{tr}(w^*w)  \operatorname{tr}(w \rho_2 w^*).
		\end{align}
		As a consequence, for every \(w\in\mathcal M_d\) such that $ \operatorname{tr}(w\rho_2w^*) \geq  \operatorname{tr}(w\rho_1w^*)>0$, merging bounds \eqref{eq:app1aa} and \eqref{eq:app1bb} and dividing by $\operatorname{tr}(w \rho_2 w^*)$ lead to
		\begin{align*}
			\Big \| \tfrac{w \rho_1 w^*}{\operatorname{tr}(w \rho_1 w^*)}   -   \tfrac{w \rho_2 w^*}{\operatorname{tr}(w \rho_2 w^*)}\Big \|_{\rm tr}  \operatorname{tr}(w \rho_1 w^*)  &   \leq    2  \|\rho_1 - \rho_2\|_{\rm tr} \operatorname{tr}(w^*w).  
		\end{align*}
		Therefore, using the Lipschitz property of $f$, 
		\begin{align*} 
			\bigg| f\Big(\tfrac{w \rho_1 w^*}{\tr(w \rho_1 w^*)}\Big)- f\Big(\tfrac{w \rho_2 w^*}{\tr(w \rho_2 w^*)}\Big)\bigg|  \operatorname{tr}(w \rho_1 w^*)  
			\leq 
			2L   \|\rho_1 - \rho_2\|_{\rm tr} \operatorname{tr}(w^*w).
		\end{align*}
		
		It remains to note that both ${W_n \rho_1 W_n^*}/{\operatorname{tr}(W_n \rho_1 W_n^*)}$ and ${W_n \rho_2 W_n^*}/{\operatorname{tr}(W_n \rho_2 W_n^*)}$ are almost surely well defined with respect to the measure $ \operatorname{tr}(W_n \rho_1 W_n^*)  \d\mutensorn$ on $A$. The stochasticity condition ensures that $\int_{W_n \in A} \operatorname{tr}(W_n^*W_n)  \d\mutensorn \leq d$, so 
		\begin{align*}\int\limits_{W_n \in A}\left| f\left(\tfrac{W_n \rho_1 W_n^*}{\tr(W_n \rho_1 W_n^*)}\right) - f\left(\tfrac{W_n \rho_2 W_n^*}{\tr(W_n \rho_2 W_n^*)}\right)\right|  {\tr(W_n \rho_1 W_n^*)}    \d\mutensorn 
			\leq 2L d    \|\rho_1 - \rho_2\|_{\rm tr}, 
		\end{align*}
		which concludes the proof.
	\end{proof} 
	
	\begin{lemma}\label{lem:A3}  Let $\rho_1, \rho_2 \in \Sd$ and $n \in \nnone$. Then
		\begin{align*} \int_{\Omega}  \big| \tr(W_n \rho_1 W_n^*) - \tr(W_n \rho_2 W_n^*) \big|   \d\mutensorn
			\leq    \|\rho_1 - \rho_2\|_{\rm tr}.\end{align*}
	\end{lemma}
	\begin{proof}  
		As in   \Cref{lem:eqVn}, $| \operatorname{tr}(w \rho_1 w^*) - \operatorname{tr}(w \rho_2 w^*) |  \leq \operatorname{tr}(  w^* w   |\rho_1 - \rho_2|)$ for every  $w \in \V_n$, so  
		\begin{align*} \int_{\Omega}  \big| \tr(W_n \rho_1 W_n^*) - \tr(W_n \rho_2 W_n^*) \big|   \d\mutensorn
			\leq      
			\tr \bigg( \int_{\Omega} W_n^* W_n  \d\mutensorn |\rho_1 - \rho_2| \bigg) 
			=    \|\rho_1 - \rho_2 \|_{\rm tr},\end{align*} 
		as desired.\end{proof}
	
	Estimates analogous to those in the proof of \Cref{lem:eqLip} together with \Cref{lem:A2,lem:A3} lead to the following result.
	\begin{theorem} 
		$\Pi$ is equicontinuous on $\Sd$.  
	\end{theorem}
	
	Therefore, applying Raugi's \Cref{thm:raugi}, we arrive at the desired conclusion.
	\begin{theorem} 
		Let $\J$ be a minimal family. Then $\Upsilon_\J$ coincides with the set of all $\Pi$-ergodic measures on $\Sd$. 
		As a consequence, a  probability measure  $\nu$ on $\Sd$  is  $\Pi$-invariant    iff 
		\[\nu =  \int_{ \Sr} \nu_{\rho,\J}  \d\lambda(\rho)\] for some probability measure $\lambda$ on $ \Sr$.  
	\end{theorem}
	
	As already mentioned, if $r_m \geq 2$, for density matrices we always have multiple $\Pi$-invariant mutually singular measures. 
	This is because the unitary dynamics that takes place on dark subspaces preserves the state spectrum. Thus, each set of states with the same nondegenerate spectrum carries a $\Pi$-invariant probability measure and this measure is unique iff $G_{\rm min}$ acts transitively on the set $\{(\hat x_1,\dotsc,\hat x_{r_m}) \:|\: (x_1,\dotsc,x_{r_m})\mbox{ is an orthonormal basis of }\cc^{r_m}\}$, e.g., if $G_{\rm min} = \mathcal{SU}(r_m)$.

	\begin{theorem}  Fix $\lambda_1 \geq \ldots \geq \lambda_{r_m} \geq 0$ such that $\sum_{i=1}^{r_m} \lambda_i = 1$.  If $G_{\rm min} = \mathcal{SU}(r_m)$, there exists a unique $\Pi$-invariant probability measure  on $\{\rho \in \Sd \,|\, \operatorname{ev}_{i}(\rho) = \lambda_i, i = 1, \ldots, r_m\}$, where $\operatorname{ev}_{i}(\rho)$ is the $i$-th largest eigenvalue of $\rho$ (counted with multiplicities), and  the formula for this measure reads 
		\[\int \operatorname{Unif}_{\mathcal S_\lambda(D)} \d\chi_{\rm inv}(D),\]
		where $\mathcal S_\lambda(D) = \{\rho \in \mathcal S(D) \,|\,  \operatorname{ev}_{i}(\rho) = \lambda_i,\, i = 1, \ldots, r_m\}$. 
		In particular,  for $\lambda_1 =  1$ this measure is $\nu_{\rm unif}$, and for $\lambda_1 = \ldots = \lambda_{r_m} =\frac {1}{r_m}$ it is $\nu_{\rm ch}$. 
	\end{theorem}

	\begin{remark}In the case $r_m = 2$, the purity of a state (i.e. $\operatorname{tr}(\rho^2)$) determines its spectrum. Assume that $G_{\rm min}   = \mathcal{SU}(2)$. Then for every given level of purity there exists a unique $\Pi$-invariant probability measure. That is, if $s \in [1/2, 1]$, the unique $\Pi$-invariant probability measure on $\{\rho \in \Sd \,|\, \tr(\rho^2) = s\}$ is 
		\[\int \operatorname{Unif}_{\mathcal S_s(D)} \d\chi_{\rm inv}(D),\]
		where  $\mathcal S_s(D) = \{\rho \in \mathcal S(D)\,|\,  \operatorname{tr}(\rho^2) = s\}$. In particular, this remark applies to Examples \hyperref[example1]{1}~\&~\hyperref[example2]{2}. 
	\end{remark}

	\addtocontents{toc}{\protect\setcounter{tocdepth}{0}}
	\section*{Acknowledgments}
	The authors were supported by the ANR project ``ESQuisses'' grant no. \mbox{ANR-20-CE47-0014-01} and by the ANR project ``Quantum Trajectories'' grant no. ANR-20-CE40-0024-01.
	\addtocontents{toc}{\protect\setcounter{tocdepth}{1}}

	\printbibliography
	
\end{document}